\documentclass[11pt]{article}

\usepackage[total={6in, 9in}]{geometry}

\usepackage{amsmath}
\usepackage {amssymb}
\usepackage{booktabs}
\usepackage[center]{caption}
\usepackage {verbatim}
\usepackage{epsfig,psfrag}
\usepackage{algorithm}
\usepackage{algorithmic}
\usepackage{url}
\usepackage{float}
\usepackage{mathtools}
\usepackage{graphicx}
\usepackage[caption=false]{subfig}
\usepackage{enumerate}
\usepackage{textgreek}
\usepackage{bbm}
\usepackage{comment}
\usepackage{amsthm}

\def\R{{\mathbb R}}
\def\Z{{\mathbb Z}}

\def \Sy{{\mathbb{S}}}

\def\1_xi{\frac{1}{\xi}}

\def\B{{\mathcal B}}

\def\F{{\mathcal F}}

\def\H{{\mathcal H}}

\def\J{{\mathcal J}}

\def\01{\ensuremath{0\mathord{-}1}}

\newcommand{\st} {{\rm s.t.}}

\newtheorem{proposition}{Proposition}
\newtheorem{theorem}{Theorem}

\allowdisplaybreaks

\title{Spectral relaxations and branching strategies for global optimization of mixed-integer quadratic programs}

\author{Carlos J. Nohra\thanks{Mitsubishi Electric Research Laboratories
        (\texttt{nohra@merl.com}).}
		\and Arvind U. Raghunathan\thanks{Mitsubishi Electric Research Laboratories
        (\texttt{raghunathan@merl.com})}.        
        \and Nikolaos V. Sahinidis\thanks{Department of Chemical Engineering,
        Carnegie Mellon University
        (\texttt{sahinidis@cmu.edu}).}}

\usepackage{amsopn}

\begin{document}

\maketitle

\begin{abstract}
We consider the global optimization of nonconvex quadratic programs and mixed-integer quadratic programs. We present a family of convex quadratic relaxations which are derived by convexifying nonconvex quadratic functions through perturbations of the quadratic matrix. We investigate the theoretical properties of these quadratic relaxations and show that they are equivalent to some particular semidefinite programs. We also introduce novel branching variable selection strategies which can be used in conjunction with the quadratic relaxations investigated in this paper. We integrate the proposed relaxation and branching techniques into the global optimization solver BARON, and test our implementation by conducting numerical experiments on a large collection of problems. Results demonstrate that the proposed implementation leads to very significant reductions in BARON's computational times to solve the test problems.
\end{abstract}

\section{Introduction}
\label{intro}

We address the global optimization of nonconvex quadratic programs (QPs) and mixed-integer quadratic programs (MIQPs) of the form:
\begin{equation}
\label{problem_statement}
\begin{array}{cl}
 \underset{x \in \R^n}{\text{min}}\; &x^T Q x + q^T x\\
 \st           & A x = b \\
 			   & C x \leq d \\
               & l \leq x \leq u \\
               & x_i \in \Z, \,\,\, \forall i \in J \subseteq \{1, \dots, n \}               
\end{array}
\end{equation}
where $Q \in \R^{n \times n}$ is a symmetric matrix which may be indefinite, $q \in {\R}^n$, $A \in {\R}^{m \times n}$, $b \in {\R}^m$, $C \in {\R}^{p \times n}$, and $d \in {\R}^p$. We assume that lower and 
upper bounds are finite, i.e., $-\infty < l_i < u_i < \infty, \,\,\, \forall i \in \{1, \dots, n \}$. For the sake of brevity, we use the notation ${\cal X} = \{x \in  {\R}^{n} \,\,\, | \,\,\, A x = b, \,\,\, C x \leq d, \,\,\,  l \leq x \leq u \}$ in the rest of the paper. Note also that even though we allow~\eqref{problem_statement} to include constraints of the form $C x \leq d$, we do not use information from these inequalities in order to convexify this problem.

QPs and MIQPs of the form~(\ref{problem_statement}) arise in a wide variety of applications including facility location and quadratic assignment~\cite{ldbhq:07}, molecular conformation~\cite{pr:94} and max-cut problems~\cite{gw95}. Given their practical importance, these problems have been studied extensively and are known to be very challenging to solve to global optimality. 

State-of-the-art global optimization solvers rely on spatial branch-and-bound algorithms to solve (\ref{problem_statement}) to global optimality. The efficiency of these algorithms primarily depends on the quality of the relaxations utilized in the bounding step. Commonly used relaxations for nonconvex QPs and MIQPs can be broadly classified in three groups. The first group consists of polyhedral relaxations typically derived via factorable programming methods~\cite{mc76,ts:comp:04} and reformulation-linearization techniques (RLT)~\cite{sa99}. The second group is given by semidefinite programming (SDP) relaxations~\cite{bst:11:mp,bw:13,s:87}. The third group involves convex quadratic relaxations derived through separable programming procedures~\cite{pgr:87}, d.c. programming techniques~\cite{t:95}, and quadratic convex reformulation methods~\cite{bel:13,bep:09}.

In this paper, we investigate a family of relaxations which falls under the third group. In particular, we consider convex quadratic relaxations derived by convexifying the objective function of~(\ref{problem_statement}) through uniform diagonal perturbations of $Q$. We revisit a very well-known technique which uses the smallest eigenvalue of $Q$ to convexify the function $x^T Q x$. Through numerical experiments, we show that, despite its simplicity, this technique leads to convex quadratic relaxations which are often significantly tighter than the polyhedral relaxations typically used by state-of-the-art global optimization solvers. Motivated by these promising results, we refine this approach in several directions and make several theoretical and algorithmic contributions.

Our first contribution is a novel convex quadratic relaxation for problems of the form~(\ref{problem_statement}), derived by using information from both $Q$ and the equality constraints $A x = b$. Under this approach, the function $x^T Q x$ is convexified by constructing a perturbation of $Q$ obtained by solving a generalized eigenvalue problem involving both $Q$ and $A$. We show that the resulting relaxation is at least as tight as the relaxation constructed by using the smallest eigenvalue of $Q$.

In our second contribution, we consider another convex quadratic relaxation in which the function $x^T Q x$ is convexified by using the smallest eigenvalue of $Z^T Q Z$, where $Z$ is a basis for the nullspace of $A$. We devise a simple procedure which allows us to approximate the bound given by this relaxation without having to compute $Z$. Moreover, we show that the relaxations obtained through this technique are at least as tight as the other two quadratic relaxations mentioned above. Unlike the polyhedral and SDP relaxations which are formulated in a higher dimensional space, the quadratic relaxations considered in this paper are constructed in the space of the original problem variables, which makes them very inexpensive to solve.

In our third contribution, we prove that the aforementioned quadratic relaxations are equivalent to some particular SDP relaxations. These results facilitate the theoretical comparisons with other relaxations that have been proposed in the literature. In particular, we show that the relaxation based on the smallest eigenvalue of $Z^T Q Z$ is the best among the class of relaxations considered in this paper.

Our fourth contribution is a method for improving the proposed quadratic relaxations with branching. We introduce a novel eigenvalue-based branching variable selection strategy for nonconvex binary quadratic programs. This strategy involves an effective approximation of the impact of branching decisions on relaxation quality.

In order to investigate the impact of the proposed techniques on the performance of branch-and-bound algorithms, we implement the quadratic relaxations and branching strategies considered in this paper in the state-of-the-art global optimization solver BARON~\cite{sahinidis:baron:96}. The new quadratic relaxations are incorporated in BARON's portfolio of relaxations and are invoked according to a new dynamic relaxation selection rule which switches between different classes of relaxations based on their relative strength. We test our implementation by conducting numerical experiments on a large collection of problems. Results demonstrate that the proposed implementation leads to a very significant improvement in the performance of BARON. Moreover, for many of the test problems, our implementation results in a new version of BARON which outperforms other state-of-the-art solvers including CPLEX and GUROBI.

The remainder of this paper is organized as follows. In \S\ref{polyhedral_sdp_quadratic_relaxations} we review various relaxations which have been considered in the literature for bounding nonconvex QPs and MIQPs. Then, in \S\ref{spectral_relaxations} we present the convex quadratic relaxations considered in this paper and investigate their theoretical properties. In \S\ref{spectral_branching} we introduce novel eigenvalue-based branching strategies. This is followed by a description of our implementation in \S\ref{implementation}. In \S\ref{computational_results}, we present the results of a computational study which includes a comparison between different classes of relaxations, an analysis of the impact of the proposed implementation on the performance of BARON, and a comparison between several global optimization solvers. Finally, \S\ref{conclusions} presents conclusions from this work.

\subsection*{Notation}
We denote by $\Z$, $\R$ and $\R_{\geq 0}$ the set integer, real and nonnegative real numbers, respectively. We use $\mathbbm{1} \in \R^n$ to denote a vector of ones. The $i$-th element of $x \in \R^n$ is denoted by $x_i$. Given $d \in \R^n$, $\text{diag}(d)$ denotes the diagonal matrix whose diagonal entries are given by the elements of $d$. The $i$-th row of a matrix $A \in \R^{m \times n}$ is denoted by $A_{i \cdot}$, and its $(i,j)$-th entry by $A_{ij}$. Let $\Sy^{n}$ denote the set of $n \times n$ real, symmetric matrices. Given $M \in \Sy^n$, we use $\lambda_i$ to represent its $i$-th eigenvalue and $v^i$ for the corresponding eigenvector. For $M \in \Sy^n$, the notation $M \succcurlyeq 0$ and $M \succ 0$, indicates that $M$ is positive semidefinite and positive definite, respectively. We denote by $I_n$ the $n \times n$ identity matrix. Let $M, N \in {\Sy}^{n}$ with $N \succ 0$. We use $\lambda_{\text{min}} (M)$ to represent the smallest eigenvalue of $M$. Similarly, we denote by $\lambda_{\text{min}} (M, N)$ the smallest generalized eigenvalue of the problem 
$M v = \lambda N v$, where $v \in {\R}^{n}$.  The inner product between $M, P \in \Sy^{n}$ is denoted by $\langle M, P \rangle = \sum_{i=1}^n\sum_{j=1}^n M_{ij}P_{ij}$.

\section{Current relaxations for nonconvex QPs and MIQPs}
\label{polyhedral_sdp_quadratic_relaxations}

In this section, we review various types of relaxations that have been proposed for bounding~\eqref{problem_statement}.

\subsection{Polyhedral relaxations}
\label{sec:polyhedral_relaxations}

One of the simplest relaxations for~\eqref{problem_statement} can be derived via factorable programming techniques~\cite{mc76,ts:comp:04}, leading to the linear program:
\begin{subequations}
\label{McCormick}
\begin{align}
\underset{x \in {\cal X}, X}{\text{min}}\;\; & \sum\limits_{i = 1}^{n}
\sum\limits_{j = 1}^{n} Q_{ij} X_{ij}  + \sum\limits_{i = 1}^{n} q_i x_{i} \label{McCormick_obj} \\
  \st\;\;      & X_{ij} \geq l_i x_j + l_j x_i - l_i l_j, \; i = 1,\ldots,n, j = i,\ldots,n, \label{McCormick_c1} \\
			   & X_{ij} \geq u_i x_j + u_j x_i - u_i u_j, \; i = 1,\ldots,n, j = i,\ldots,n, \label{McCormick_c2} \\
			   & X_{ij} \leq l_i x_j + u_j x_i - l_i u_j, \; i = 1,\ldots,n, j = i,\ldots,n, \label{McCormick_c3} \\
			   & X_{ij} \leq u_i x_j + l_j x_i - u_i l_j, \; i = 1,\ldots,n, j = i,\ldots,n, \label{McCormick_c4} \\
			   & X_{ij} = X_{ji}, \; i = 1,\ldots,n, j = (i+1),\ldots,n, \label{McCormick_c5}
\end{align}
\end{subequations}
where $X$ is a symmetric matrix of introduced variables and~\eqref{McCormick_c1}--\eqref{McCormick_c4} are the so-called \textit{McCormick inequalities}. This relaxation is often referred to as the \textit{McCormick relaxation} of~\eqref{problem_statement}. Even though this relaxation is simple to implement, it often leads to relatively weak bounds, and as a result, it is typically tightened by adding various classes of valid inequalities~\cite{bkst:15,bgl:18,msf:15,zs:13:oms}.

Another polyhedral relaxation for~\eqref{problem_statement} can be constructed through the reformulation linearization techniques (RLT), obtaining the linear program~\cite{sa99}:
\begin{subequations}
\label{RLT_1}
\begin{align}
\underset{x \in {\cal X}, X}{\text{min}}\;\; & \sum\limits_{i = 1}^{n} \sum\limits_{j = 1}^{n} Q_{ij} X_{ij}  + \sum\limits_{i = 1}^{n} q_i x_{i} \\
  \st\;\;      & \textrm{Eqs.~}\eqref{McCormick_c1}-\eqref{McCormick_c5}\\
  			   & \sum\limits_{i = 1}^{n} A_{ki} X_{ij} = b_k x_j, \; k = 1, \dots, m, \; j = 1, \dots n \\
			   & \sum\limits_{i = 1}^{n} C_{ki} X_{ij} - l_j \sum\limits_{i = 1}^{n} C_{ki} x_i - d_k x_j \leq - l_j d_k, \; k = 1, \dots, p, \; j = 1, \dots n \\
			   & \hspace{-1em} -\sum\limits_{i = 1}^{n} C_{ki} X_{ij} + u_j \sum\limits_{i = 1}^{n} C_{ki} x_i + d_k x_j \leq u_j d_k, \; k = 1, \dots, p, \; j = 1, \dots n, \\
			   & \hspace{-1em} -\sum\limits_{i = 1}^{n} \sum\limits_{j = 1}^{n} C_{ki} C_{lj} X_{ij} + \sum\limits_{i = 1}^{n} (d_l C_{ki} + d_k C_{li}) x_i \leq d_k d_l, \; k, l = 1, \dots, p   
\end{align}
\end{subequations}

This relaxation is often referred to as the \textit{first-level RLT relaxation} of~\eqref{problem_statement}.

\subsection{SDP relaxations}
\label{sec:semidefinite_relaxations}

The simplest SDP relaxation for~\eqref{problem_statement} is given by:
\begin{subequations}
\label{Shor_relaxation}
\begin{align}
\underset{x \in {\cal X}, X}{\text{min}}\;\; & \langle Q, X \rangle + q^T x \label{Shor_relaxation_obj} \\
  \st\;\;      & X - x x^T \succcurlyeq 0 \label{Shor_relaxation_c1}
\end{align}
\end{subequations}

This relaxation is often referred to as the \textit{Shor relaxation} of~\eqref{problem_statement}~\cite{s:87}. The Shor relaxation can be strengthened by including additional valid constraints. This can be achieved, for instance, by adding the McCormick inequalities corresponding to the diagonal elements of the matrix $X$, which results in the following SDP:
\begin{subequations}
\label{SDP_d}
\begin{align}
\underset{x \in {\cal X}, X}{\text{min}}\;\; & \langle Q, X \rangle + q^T x \label{SDP_d_obj} \\
  \st\;\;      & X - x x^T \succcurlyeq 0 \label{SDP_d_c1} \\
               & X_{ii} \geq 2 l_i x_i - l_i^2, \; i = 1, \dots, n \label{SDP_d_c2} \\
			   & X_{ii} \geq 2 u_i x_i - u_i^2, \; i = 1, \dots, n \label{SDP_d_c3} \\
			   & X_{ii} \leq u_i x_i + l_i x_i - u_i l_i, \; i = 1, \dots, n \label{SDP_d_c4}
\end{align}
\end{subequations}

It is easy to verify that~\eqref{SDP_d_c2} and~\eqref{SDP_d_c3} are implied by $X - x x^T \succcurlyeq 0$ and hence redundant in this formulation. The SDP~\eqref{SDP_d} can be further tightened by including constraints derived from $Ax = b$. For example, we can add the constraints obtained by lifting the valid equalities $\sum_{j = 1}^{n} A_{kj} x_i x_j = b_k x_i, \; k = 1, \dots, m, \; i = 1, \dots, n$, to the space of $(x,X)$. This leads to the following SDP:
\begin{subequations}
\label{SDP_d_a_x}
\begin{align}
\underset{x \in {\cal X}, X}{\text{min}}\;\; & \langle Q, X \rangle + q^T x \label{SDP_d_a_x_obj} \\
  \st\;\;      & \textrm{Eqs.~}\eqref{SDP_d_c1}-\eqref{SDP_d_c4} \label{SDP_d_a_x_c1} \\
			   & \sum\limits_{j = 1}^{n} A_{kj} X_{ij} = b_k x_i, \; k = 1, \dots, m, \; i = 1, \dots, n \label{SDP_d_a_x_c2}
\end{align}
\end{subequations}

Alternatively, we can add a single constraint derived by lifting the equality $(Ax-b)^T (Ax-b) = 0$ to the space of $(x,X)$, obtaining the following SDP:
\begin{subequations}
\label{SDP_d_a}
\begin{align}
\underset{x \in {\cal X}, X}{\text{min}}\;\; & \langle Q, X \rangle + q^T x \\
  \st\;\;      & \textrm{Eqs.~}\eqref{SDP_d_c1}-\eqref{SDP_d_c4}
  \label{SDP_d_a_c1}\\
			   & \langle A^T A, X \rangle - 2 (A^T b)^T x + b^T b  = 0 \label{SDP_d_a_c2}
\end{align}
\end{subequations}

The SDPs~\eqref{SDP_d_a_x} and~\eqref{SDP_d_a} are equivalent (see Proposition 5 in~\cite{fr:07} for details).

\subsection{Convex quadratic relaxations}
\label{sec:convex_quadratic_relaxations}

In the following, we briefly review three important classes of convex quadratic relaxations for problems of the form~\eqref{problem_statement}.

\subsubsection{Separable programming relaxation}
\label{sec:separable_programming_relaxation}

This relaxation, which is constructed through the eigendecomposition of $Q$ (see~\cite{pgr:87} for details), is given by:
\begin{equation}
\label{EIGDC2}
\begin{array}{cl}
 \underset{x \in {\cal X}, y}{\text{min}}\;\; & \sum\limits_{i: \lambda_i > 0} \lambda_i y_i^2 + \sum\limits_{i: \lambda_i < 0} \lambda_i \left( (L_i + U_i) y_i - L_i U_i \right)  + q^T x \\
 \st\;\;       & y_i = {v^i}^T x, \; i = 1, \dots, n \\
               & L_i \leq y_i \leq U_i, \; i = 1,\dots, n \\
\end{array}
\end{equation}
where $y_i$ are introduced variables and the bounds $L_i$ and $U_i$ are respectively determined by minimizing and maximizing the linear function ${v^i}^T x$ over ${\cal X}$.

\subsubsection{D.C. programming relaxations}
\label{sec:d_c_programming_relaxations}

The key idea behind this approach is to decompose the objective function of~\eqref{problem_statement} as $f(x) = x^T Q x + q^T x = g(x) - h(x)$, where $g(x)$ and $h(x)$ are both convex quadratic functions~\cite{t:95}. In their most generic form, these relaxations can be expressed as:
\begin{equation}
\label{DC_programming_relaxation}
\begin{array}{cl}
 \underset{x \in {\cal X}}{\text{min}}\;\; & g(x) - \bar{h}_{\cal X}(x)
\end{array}
\end{equation}
where $\bar{h}_{\cal X}(x)$ is a concave overestimator of $h(x)$ over ${\cal X}$. A particular d.c. programming relaxation is the classical \textalpha BB relaxation~\cite{amf:95}, which for~\eqref{problem_statement} takes the form:
\begin{equation}
\label{alphaBB_relaxation}
\begin{array}{cl}
 \underset{x \in {\cal X}}{\text{min}}\;\; & x^T Q x + q^T x - \sum_{i = 1}^{n} \alpha_i (x_i - l_i) (u_i - x_i)
\end{array}
\end{equation}
where $\alpha_i, \; i = 1, \dots, n$, are nonnegative parameters chosen such that the objective function of~\eqref{alphaBB_relaxation} is convex over ${\cal X}$. Other particular examples of d.c. programming relaxations are the relaxations constructed via undominated d.c. decompositions~\cite{bl:04}.

\subsubsection{Relaxations based on quadratic convex reformulations}
\label{sec:qcr_based_relaxations}

These relaxations are used in the Quadratic Convex Reformulation (QCR) methods. To illustrate these techniques, consider the binary quadratic program:
\begin{equation}
\label{problem_statement_binary}
\begin{array}{cl}
 \underset{x \in {\cal X}_B}{\text{min}}\; &x^T Q x + q^T x
\end{array}
\end{equation}
where ${\cal X}_B = \{x \in \{0,1\}^n \,\,\, | \,\,\, A x = b, \; C x \leq d \}$. The QCR approaches involve two steps. The first step consists in reformulating~\eqref{problem_statement_binary} to an equivalent binary quadratic program whose continuous relaxation is convex. In the second step, the reformulated problem is solved using a branch-and-bound algorithm. At each node of the branch-and-bound tree, a lower bound is obtained by solving the continuous relaxation of the reformulated problem, which is a convex quadratic program. 

One of the earliest references to these methods is found in a paper by Hammer and Rubin~\cite{hr:70}, in which the following reformulation for~\eqref{problem_statement_binary} is proposed:
\begin{equation}
\label{reformulation_Hammer_and_Rubin}
\begin{array}{cl}
 \underset{x \in {\cal X}_B}{\text{min}}\; &x^T Q_{\lambda} x + q_{\lambda} ^T x\\              
\end{array}
\end{equation}
where $Q_{\lambda} = Q - \min(0, \lambda_{\text{min}} (Q)) I_n$ and $q_{\lambda} = q + \min(0, \lambda_{\text{min}} (Q)) \mathbbm{1}$. It is simple to check that $Q_{\lambda} \succcurlyeq 0$, and that the objective functions of~\eqref{problem_statement_binary} and~\eqref{reformulation_Hammer_and_Rubin} are equivalent $\forall x \in \{0,1\}^n$. Another reformulation of~\eqref{problem_statement_binary} was proposed by Billionnet et al.~\cite{bep:09}:
\begin{equation}
\label{QCR_da}
\begin{array}{cl}
 \underset{x \in {\cal X}_B}{\text{min}}\; &x^T Q_{d_q, \Theta_q} x + q_{d_q, \Theta_q} ^T x             
\end{array}
\end{equation}
where $Q_{d_q, \Theta_q} = Q + \text{diag} (d_q) + \frac{1}{2} (\Theta_q^T A + A^T \Theta_q)$, $q_{d_q, \Theta_q} = q - d_q - \Theta_q^T b $, $d_q \in \R^n$ and $\Theta_q \in \R^{m \times n}$. The perturbation parameters $d_q$ and $\Theta_q$ are chosen such that $Q_{d_q, \Theta_q} \succcurlyeq 0$ and the bound of the continuous relaxation of~\eqref{QCR_da} is maximized. This is done by solving the SDP~\eqref{SDP_d_a_x}, and setting the entries of $d_q$ and $\Theta_q$ to the optimal values of the dual variables associated with the constraints~\eqref{SDP_d_c4} and~\eqref{SDP_d_a_x_c2}, respectively. Note that the continuous relaxation of~\eqref{QCR_da} provides the same bound as the SDP~\eqref{SDP_d_a_x}.

\section{Spectral relaxations for nonconvex QPs and MIQPs}
\label{spectral_relaxations}

In the following, we present a family of convex quadratic relaxations for problems of the form~(\ref{problem_statement}), and investigate their theoretical properties. Before providing a detailed derivation of these relaxations, we state two results which we will repeatedly use throughout this section. First, we recall that the minimum eigenvalue of a matrix $M$ and the minimum generalized eigenvalue of a pair of matrices $(M,N)$, with $M, N \in \Sy^n, N \succ 0$, can be expressed 
in terms of the Rayleigh quotient as~\cite{gvl96}:
\begin{equation}
    \lambda_{\text{min}} (M) = \underset{x \neq 0}{\text{min}} \;\; \frac{x^T M x}{x^T x} 
    \,\,\text{ and }\,\,
    \lambda_{\text{min}} (M, N) = \underset{x \neq 0}{\text{min}} \;\; \frac{x^T M x}{x^T N x}.
    \label{def_RayleighQuotient}
\end{equation}
Second, we provide a particularly useful formulation for the dual of a certain SDP.

\begin{proposition}\label{prop:SDP_dual}
Consider the following SDP 
\begin{subequations}
\label{prop:sdp}
\begin{align}
\underset{x \in {\cal X}, X}{\emph{\text{min}}}\;\; & \langle \hat{Q}, X \rangle + \hat{q}^T x \label{prop:sdp_obj} \\
  \st\;\;      & X - xx^T \succcurlyeq 0 \label{prop:sdp_c1} \\
			   & \langle \hat{C}_i, X \rangle + \hat{c}_i x + \hat{d}_i  = 0, \; i = 1,\ldots,m_1 \label{prop:sdp_c2} \\
			   & \langle \bar{C}_i, X \rangle + \bar{c}_i x + \bar{d}_i  \leq 0, \; i = 1,\ldots,p_1 \label{prop:sdp_c3}			   
\end{align}
\end{subequations}
for some $\hat{Q}, \hat{C}_i, \bar{C}_i \in \Sy^n, \hat{q}, \hat{c}_i, \bar{c}_i \in \R^n$ and $\hat{d}_i, \bar{d}_i \in \R$. The dual of~\eqref{prop:sdp} is given by
\begin{equation}
\underset{\alpha \in \R^{m_1}, \beta \in \R_{\geq 0}^{p_1} : \hat{Q}_{\alpha, \beta} \succcurlyeq 0}{\emph{\text{max}}} \left\lbrace
\begin{aligned}
\centering
& \underset{x \in {\cal X}}{\emph{\text{min}}}
& & x^T \hat{Q}_{\alpha, \beta} x + {\hat{q}_{\alpha, \beta}}^T x + \hat{d}_{\alpha, \beta}
\end{aligned}
\right\rbrace.
\notag
\end{equation}
\end{proposition}
where $\hat{Q}_{\alpha, \beta} = \hat{Q} + \sum\limits_{i=1}^{m_1} \alpha_i \hat{C}_i + \sum\limits_{i=1}^{p_1} \beta_i \bar{C}_i$,
$\hat{q}_{\alpha, \beta} = \hat{q} + \sum\limits_{i=1}^{m_1} \alpha_i \hat{c}_i + \sum\limits_{i=1}^{p_1} \beta_i \bar{c}_i$, and $\hat{d}_{\alpha, \beta} = \sum\limits_{i=1}^{m_1} \alpha_i \hat{d}_i + \sum\limits_{i=1}^{p_1} \beta_i \bar{d}_i$.

\begin{proof}
By dualizing the constraints~\eqref{prop:sdp_c2} using the multipliers $\alpha_i \in \R$, for $i = 1,\ldots,m_1$, and the constraints  ~\eqref{prop:sdp_c3} using the multipliers $\beta_i \in \R_{\geq 0}$, for $i = 1, \ldots, p_1$, we have that the Lagrangian dual of the SDP~\eqref{prop:sdp} is given by:
\begin{equation}
\underset{\alpha \in \R^{m_1}, \beta \in \R_{\geq 0}^{p_1}}{\text{max}} \left\lbrace
\begin{aligned}
\centering
& \underset{x \in {\cal X}}{\text{min}}
& & \langle \hat{Q}_{\alpha, \beta}, X \rangle + {\hat{q}_{\alpha, \beta}}^T x + \hat{d}_{\alpha, \beta} \\
& \text{s.t.} & & X - xx^T \succcurlyeq 0
\end{aligned}
\right\rbrace
\end{equation}
Since the variables $x$ have finite lower and upper bounds, the set ${\cal X}$ is bounded. Hence, for the inner minimization to be bounded below, we need to choose $\alpha$ and $\beta$ such that $\hat{Q}_{\alpha, \beta} \succcurlyeq 0$. This restriction on $\alpha$ and $\beta$ implies that the optimal solution of the inner minimization problem satisfies $X = xx^T$. The claim follows after substituting $X = xx^T$.
\end{proof}

\subsection{Eigenvalue relaxation}
\label{sec:eigrel}

We start by reformulating~(\ref{problem_statement}) as:
\begin{equation}
\label{EIGR1}
\begin{array}{cl}
 \underset{x \in {\cal X}}{\text{min}}\;\; &x^T Q x + q^T x + \alpha_e \sum\limits_{i = 1}^{n} x_i^2  - \alpha_e \sum\limits_{i = 1}^{n} x_i^2\\
 \st\;\;       & x_i \in \Z, \,\,\, \forall i \in J              
\end{array}
\end{equation}
where $\alpha_e \geq 0$. By dropping the integrality conditions from~(\ref{EIGR1}) and using the concave envelope of $x_i^2$ over $[l_i, u_i]$, we obtain the following relaxation: 
\begin{equation}
\label{EIGR3}
\begin{array}{cl}
\underset{x \in {\cal X}}{\text{min}}\;\; & x^T Q_{\alpha_e} x + q_{\alpha_e}^T x + k_{\alpha_e}
\end{array}
\end{equation}
where $Q_{\alpha_e} = Q + \alpha_e I_n$, $q_{\alpha_e} = q - \alpha_e (l + u)$, and $k_{\alpha_e} = \alpha_e l^T u$.

To ensure that~\eqref{EIGR3} is convex, it suffices to choose $\alpha_e \geq -\min(0, \lambda_{\text{min}} (Q))$, since this renders $Q_{\alpha_e}$ positive semidefinite. Moreover, it is simple to check that $\alpha_e = -\min(0, \lambda_{\text{min}} (Q))$ provides the tightest convex relaxation of the form~\eqref{EIGR3} for which $Q_{\alpha_e} \succcurlyeq 0$. We refer to this relaxation as the \textit{eigenvalue relaxation} of~\eqref{problem_statement}.

Two interesting observations can be made on~\eqref{EIGR3} when $\alpha_e \geq -\min(0, \lambda_{\text{min}} (Q))$. First, the derivation of~\eqref{EIGR3} can be seen as an application of the d.c. programming technique reviewed in \S\ref{sec:d_c_programming_relaxations}, whereby the objective function of~\eqref{problem_statement} is expressed as the difference of the convex quadratic functions $g(x) = x^T Q x + q^T x + \alpha_e \sum_{i = 1}^{n} x_i^2$ and $h(x) = \alpha_e \sum_{i = 1}^{n} x_i^2$. Second, \eqref{EIGR3} is equivalent to the \textalpha BB relaxation discussed in \S\ref{sec:d_c_programming_relaxations} if we set $\alpha_i = \alpha_e, \; \forall i = 1, \dots n$, in~\eqref{alphaBB_relaxation}. 

Note also that if all the variables in~\eqref{problem_statement} are binary and $\alpha_e = -\min(0, \lambda_{\text{min}} (Q))$, then~\eqref{EIGR3} is equivalent to the continuous relaxation of the convex binary quadratic program~\eqref{reformulation_Hammer_and_Rubin}, which was considered by Hammer and Rubin~\cite{hr:70}.

Even though the eigenvalue relaxation is relatively simple to construct, in many cases it can be significantly tighter than the polyhedral relaxations commonly used in state-of-the-art global optimization solvers (see~\S\ref{sec:results_relaxations}). Motivated by this observation, we further investigate the theoretical properties of this relaxation. In particular, we show that the eigenvalue relaxation is equivalent to the following SDP:
\begin{subequations}
\label{SDP_EIGR}
\begin{align}
\underset{x \in {\cal X}, X}{\text{min}}\;\; & \langle Q, X \rangle + q^T x \label{SDP_EIGR_obj} \\
  \st\;\;      & X - x x^T  \succcurlyeq 0 \label{SDP_EIGR_c3} \\
 			   & \langle I_n, X \rangle - {\left( l + u \right)}^T x + l^T u \label{SDP_EIGR_c4} \leq 0 
\end{align}
\end{subequations}

\begin{proposition}
\label{EIGR_SDP_EIGR_equivalence}
Suppose that the matrix $Q$ is indefinite. Denote by $\mu_{\emph{\text{EIG}}}(\alpha_e)$ and $\mu_{\emph{\text{SDP\_EIG}}}$ the optimal objective function values in~\eqref{EIGR3} and~\eqref{SDP_EIGR}, respectively. Then, the following hold:
\begin{enumerate}[(i)]
\item If $\alpha_e \geq -\lambda_{\emph{{\text{min}}}} (Q)$ in~\eqref{EIGR3}, then $\mu_{\emph{\text{EIG}}}(\alpha_e) \leq \mu_{\emph{\text{SDP\_EIG}}}$.
\item If $\alpha_e = -\lambda_{\emph{\text{min}}} (Q)$ in~\eqref{EIGR3}, then $\mu_{\emph{\text{EIG}}}(\alpha_e) = \mu_{\emph{\text{SDP\_EIG}}}$.
\end{enumerate}
\end{proposition}
\begin{proof}

The proof of this proposition relies on strong duality holding for~\eqref{SDP_EIGR}. We start by showing that~\eqref{SDP_EIGR} admits a strictly feasible solution. Let $\bar{x} \in \R^{n}$ be a vector such that $A \bar{x} = b$, $C \bar{x} < d$, and $l < \bar{x} < u$. Recall that the concave envelope of $x_i^2$ over $[l_i, u_i]$ is given by $(l_i + u_i) x_i - l_i u_i$. Since $l_i < \bar{x}_i < u_i, \forall i = 1, \dots, n$, it follows that $(l_i + u_i) \bar{x}_i - l_i u_i - \bar{x}_i^2 > 0, \forall i = 1, \dots, n$. Define $\epsilon := (l_1 + u_1) \bar{x}_1 - l_1 u_1 - \bar{x}_1^2$. Clearly, there exists $\delta \in \R$ such that $0 < \delta < \epsilon$. Let $\bar{X} \in \Sy^{n}$ be the matrix satisfying:
\begin{equation}
\label{SDP_EIGR_dual_1}
\begin{array}{l}
\bar{X}_{11} = (l_1 + u_1) \bar{x}_1 - l_1 u_1 - \delta \\ 
\bar{X}_{ii} = (l_i + u_i) \bar{x}_i - l_i u_i, \; i = 2, \dots , n \\ 
\bar{X}_{ij} = \bar{X}_{ji} = \bar{x}_i \bar{x}_j, \; i = 1, \dots, n, \; j = (i + 1), \dots, n
\end{array}
\end{equation}
It is simple to check that~\eqref{SDP_EIGR_c4} is strictly satisfied by $(\bar{x}, \bar{X})$. Define $\hat{X} := \bar{X} - \bar{x} \bar{x}^T$. Then, it follows that $\hat{X} \in \Sy^n$ is a diagonal matrix with entries given by:
\begin{equation}
\label{SDP_EIGR_dual_2}
\begin{array}{l}
\hat{X}_{11} = (l_1 + u_1) \bar{x}_1 - l_1 u_1 - \bar{x}_1^2 - \delta \\
\hat{X}_{ii} = (l_i + u_i) \bar{x}_i - l_i u_i - \bar{x}_i^2, \; i = 2, \dots , n \\ 
\end{array}
\end{equation}
It is clear that $\hat{X}_{ii} > 0$, $\forall i = 1, \dots, n$. Hence, $\hat{X} \succ 0$ and $(\bar{x}, \bar{X})$ is strictly feasible in~\eqref{SDP_EIGR}. Therefore, Slater's condition is satisfied by~(\ref{SDP_EIGR}), which implies that strong duality holds and the optimal value of the dual problem is attained. 

Now, we consider the dual of~\eqref{SDP_EIGR}. By Proposition~\ref{prop:SDP_dual}, this dual is given by:
\begin{equation}
\label{SDP_EIGR_dual_5}
\underset{\hat{\alpha}_e \in \R_{\geq 0}: Q_{\hat{\alpha}_e} \succcurlyeq 0}{\text{max}} \left\lbrace
\begin{aligned}
\centering
& \underset{x \in {\cal X}}{\text{min}}
& & x^T Q_{\hat{\alpha}_e} x + q_{\hat{\alpha}_e}^T x + k_{\hat{\alpha}_e}
\end{aligned}
\right\rbrace
\end{equation}
where $\hat{\alpha}_e$ is the multiplier for the constraint~\eqref{SDP_EIGR_c4}, $Q_{\hat{\alpha}_e} = Q + \hat{\alpha}_e I_n$, $q_{\hat{\alpha}_e} = q - \hat{\alpha}_e (l + u)$, and $k_{\hat{\alpha}_e} = \hat{\alpha}_e l^T u$. Since $Q$ is indefinite, it is clear that $Q_{\hat{\alpha}_e} \succcurlyeq 0$ for $\hat{\alpha}_e \geq - \lambda_{\text{min}} (Q)$. For any such choice of $\hat{\alpha}_e$, the inner minimization in~\eqref{SDP_EIGR_dual_5} takes the same form as the QP~\eqref{EIGR3}. Hence, the claim in (i) follows by weak duality. Let $\mu_{\text{DSDP\_EIG}}$ be the optimal objective function value in~\eqref{SDP_EIGR_dual_5}. Then, $\mu_{\text{DSDP\_EIG}} = \max\{ \mu_{\text{EIG}}(\hat{\alpha}_e) \,|\, \hat{\alpha}_e \in \R_{\geq 0} : Q_{\hat{\alpha}_e} \succcurlyeq 0\}$. It is easy to verify that $\mu_{\text{EIG}}(\hat{\alpha}_e)$ is monotonically decreasing in $\hat{\alpha}_e$. Hence, $\mu_{\text{DSDP\_EIG}} = \mu_{\text{EIG}}(\hat{\alpha}_e)$ occurs when $\hat{\alpha}_e = -\lambda_{\text{min}} (Q)$.  Combining this with $\mu_{\text{DSDP\_EIG}} = \mu_{\text{SDP\_EIG}}$, which holds by strong duality, proves the claim in (ii).
\end{proof}

\subsection{Generalized eigenvalue relaxation}
\label{sec:geigrel}

In this section, we consider a new quadratic relaxation which improves the bound of the eigenvalue relaxation by using information from the equality constraints. We start by reformulating~(\ref{problem_statement}) as:
\begin{equation}
\label{GEIGR1}
\begin{array}{cl}
 \underset{x \in {\cal X}}{\text{min}}\;\; &x^T Q x + q^T x + \alpha_g \sum\limits_{i = 1}^{n} x_i^2  - \alpha_g \sum\limits_{i = 1}^{n} x_i^2 + \alpha_g \|Ax - b\|^2 \\
 \st\;\;       & x_i \in \Z, \,\,\, \forall i \in J              
\end{array}
\end{equation}
where $\alpha_g \geq 0$. By dropping the integrality conditions from~(\ref{GEIGR1}) and using the concave envelope of $x_i^2$ over $[l_i, u_i]$, we obtain the following quadratic relaxation:
\begin{equation}
\label{GEIGR3}
\begin{array}{cl}
\underset{x \in {\cal X}}{\text{min}}\;\; & x^T Q_{\alpha_g} x + q_{\alpha_g}^T x + k_{\alpha_g} 
\end{array}
\end{equation}
where $Q_{\alpha_g} = Q + \alpha_g (I_n + A^T A )$, $q_{\alpha_g} = q - \alpha_g ( l + u + 2 A^T b )$, and $k_{\alpha_g} = \alpha_g ( l^T u + b^T b )$. 

In the following proposition, we provide a condition for choosing $\alpha_g$ which ensures that the above problem is a convex relaxation of~\eqref{problem_statement}.
\begin{proposition}
\label{GEIGR_convexity}
Let $\alpha_g  \geq -\min(0, \lambda_{\emph{\text{min}}} (Q, I_n + A^T A))$ in~\eqref{GEIGR3}. Then,~\eqref{GEIGR3} is a convex quadratic program.
\end{proposition}
\begin{proof}

To establish the convexity of~\eqref{GEIGR3}, it suffices to verify that $Q_{\alpha_g} = Q + \alpha_g (I_n + A^T A )$ is positive semidefinite. From the definition of the Rayleigh quotient for the generalized eigenvalue pair $(Q,I_n + A^TA)$ in~\eqref{def_RayleighQuotient} we obtain: 
\begin{equation}
\label{GEIGR_convexity_2}
\begin{aligned}
\lambda_{\text{min}} (Q, I_n + A^T A) \leq \frac{x^T Q x}{x^T \left( I_n + A^T A \right) x}, \,\,\, \forall x \neq 0
\end{aligned}
\end{equation}
which using the positive definiteness of $(I_n + A^TA)$ can be equivalently written as:
\begin{equation}
\label{GEIGR_convexity_3}
\begin{aligned}
x^T Q_{\alpha_g} x \geq \left(\alpha_g + \lambda_{\text{min}} (Q, I_n + A^T A) \right) x^T \left( I_n + A^T A \right) x, \,\,\, \forall x \neq 0.
\end{aligned}
\end{equation}
It is readily verified that $Q_{\alpha_g} \succcurlyeq 0$ for $\alpha_g \geq -\min(0,\lambda_{\text{min}} (Q, I_n + A^T A))$. 
\end{proof}

From Proposition~\ref{GEIGR_convexity}, it follows that $\alpha_g = -\min(0, \lambda_{\text{min}} (Q, I_n + A^T A))$ provides the tightest convex relaxation of the form~\eqref{GEIGR3} for which $Q_{\alpha_g} \succcurlyeq 0$. We refer to this convex relaxation as the \textit{generalized eigenvalue relaxation} of~(\ref{problem_statement}). Next, we show that this relaxation is at least as tight as the eigenvalue relaxation.
\begin{proposition}
\label{EIGR_GEIGR_comparison}
Suppose that $\alpha_e = -\min(0, \lambda_{\emph{\text{min}}} (Q))$ in~\eqref{EIGR3} and assume that $\alpha_g = -\min(0, \lambda_{\emph{\text{min}}} (Q, I_n + A^T A))$ in~\eqref{GEIGR3}. Denote by $\mu_{\emph{\text{EIG}}}$ and $\mu_{\emph{\text{GEIG}}}$ the optimal objective function values in~(\ref{EIGR3}) and~(\ref{GEIGR3}), respectively. Then, $\mu_{\emph{\text{GEIG}}} \geq \mu_{\emph{\text{EIG}}}$.

\end{proposition}

\begin{proof}

To prove that $\mu_{\text{GEIG}} \geq \mu_{\text{EIG}}$, it suffices to show that $\alpha_g \leq \alpha_e$. We will use the definition of the Rayleigh quotient in~\eqref{def_RayleighQuotient}. We consider the following cases:

\begin{enumerate}[(a)]

\item $\lambda_{\text{min}} (Q) \geq 0$. This implies that $x^T Q x \geq 0, \; \forall x \in \R^n \setminus\{0\}$. Moreover, it is clear that $x^T(I_n + A^T A) x > 0, \; \forall x \in \R^n \setminus\{0\}$. Then, from~\eqref{def_RayleighQuotient} it follows that $\lambda_{\text{min}} (Q, I_n + A^T A) \geq 0$. Hence, $\alpha_e = \alpha_g = 0$, which implies that $\mu_{\text{GEIG}} = \mu_{\text{EIG}}$.

\item $\lambda_{\text{min}} (Q) < 0$. This implies that $\exists x \in \R^n$ such that $x^T Q x < 0$. From~\eqref{def_RayleighQuotient}, it follows that $\lambda_{\text{min}} (Q, I_n + A^T A) < 0$. Then, it is clear that $\alpha_e = -\lambda_{\text{min}} (Q)$ and $\alpha_g = -\lambda_{\text{min}} (Q, I_n + A^T A)$. Define the set $D = \{ x \in \R^n : x \neq 0, x^T Q x < 0 \}$. Clearly, $D$ is nonempty. 
It is easy to verify that the minimum in~\eqref{def_RayleighQuotient} occurs for $x \in D$. Combining this observation with $x^T(I + A^TA)x \geq x^Tx$, we obtain
\begin{equation}
\label{EIGR_GEIGR_comparison_3}
\begin{array}{cc}
\dfrac{x^T Q x}{x^T \left( I_n + A^T A \right) x} \geq \dfrac{x^T Q x}{x^T x}, \;\;\; \forall x \in D.
\end{array}
\end{equation}
Hence, $\lambda_{\text{min}} (Q, I_n + A^T A) \geq \lambda_{\text{min}} (Q)$, which implies that $\alpha_g \leq \alpha_e$, and $\mu_{\text{GEIG}} \geq \mu_{\text{EIG}}$.
\end{enumerate}
\end{proof}

The idea of using information from the equality constraints to convexify the objective function of~\eqref{problem_statement} has been exploited before in the context of the QCR methods discussed in \S\ref{sec:qcr_based_relaxations}. Even though our approach also uses information from the equality constraints, it differs from the QCR techniques considered in~\cite{bel:13,bep:09} in three ways. First, we do not seek the development of a reformulation of the original problem but instead the construction of cheap quadratic relaxations which can be incorporated in a branch-and-bound framework. Second, under our approach, at a given node of the branch-and-bound tree, we update the perturbation parameters used to construct these quadratic relaxations. This is done by solving the eigenvalue or generalized eigenvalue problems involving the submatrices of $Q$ and $I_n + A^T A$ obtained after eliminating the rows and columns corresponding to the variables that have been fixed. This update results in tighter bounds, and as shown in \S\ref{sec:results_baron}--\ref{sec:results_comparison_with_qcr}, it can have a significant impact on the performance of branch-and-bound algorithms, especially in the binary case, in which our relaxations can be used in conjunction with the branching strategy introduced in \S\ref{spectral_branching}. By contrast, in the QCR methods, the perturbation parameters used to convexify the problem are calculated only once, prior to the initialization of the branch-and-bound tree, and are not updated during the execution of the branch-and-bound algorithm. Third, in our method, the perturbation parameters can be obtained by solving an eigenvalue or generalized eigenvalue problem, which is often inexpensive. Under the QCR approaches, calculating the perturbation parameters involves the solution of an SDP, which is more computationally expensive.

Note also that, unlike our approach, the separable and d.c. programming techniques described in \S\ref{sec:separable_programming_relaxation} and \S\ref{sec:d_c_programming_relaxations} do not use information from the equality constraints to improve the bound of the resulting relaxations.

We next show that the generalized eigenvalue relaxation is equivalent to the following SDP:
\begin{subequations}
\label{SDP_GEIGR}
\begin{align}
\underset{x \in {\cal X}, X}{\text{min}}\;\; & \langle Q, X \rangle + q^T x \label{SDP_GEIGR_obj} \\
\st\;\;     & X - x x^T \succcurlyeq 0 \label{SDP_GEIGR_c3} \\
 			& \langle I_n, X \rangle - {\left( l + u \right)}^T x + l^T u + \langle A^T A, X \rangle - {\left( 2 A^T b \right)}^T x +  b^T b \leq 0 \label{SDP_GEIGR_c4}
\end{align}
\end{subequations}

\begin{proposition}
\label{GEIGR_SDP_GEIGR_equivalence}
Suppose that the matrix $Q$ is indefinite. Denote by $\mu_{\emph{\text{GEIG}}}(\alpha_g)$ and $\mu_{\emph{\text{SDP\_GEIG}}}$ the optimal objective function values in~(\ref{GEIGR3}) and~(\ref{SDP_GEIGR}), respectively. Then, the following hold:
\begin{enumerate}[(i)]
\item If $\alpha_g \geq -\lambda_{\emph{\text{min}}} (Q, I_n + A^T A)$ in~\eqref{GEIGR3}, then $\mu_{\emph{\text{GEIG}}}(\alpha_g) \leq \mu_{\emph{\text{SDP\_GEIG}}}$.
\item If $\alpha_g = -\lambda_{\emph{\text{min}}} (Q, I_n + A^T A)$ in~\eqref{GEIGR3}, then $\mu_{\emph{\text{GEIG}}}(\alpha_g) = \mu_{\emph{\text{SDP\_GEIG}}}$.
\end{enumerate}
\end{proposition}

\begin{proof}
We will rely on strong duality holding for~\eqref{SDP_GEIGR} and follow the same line of arguments used in the proof of Proposition of~\ref{EIGR_SDP_EIGR_equivalence}. We start by showing that~\eqref{SDP_GEIGR} admits a strictly feasible solution. Let $\bar{x} \in \R^{n}$ be a vector such that $A \bar{x} = b$, $C \bar{x} < d$, and $l < \bar{x} < u$. Recall that the concave envelope of $x_i^2$ over $[l_i, u_i]$ is given by $(l_i + u_i) x_i - l_i u_i$. Since $l_i < \bar{x}_i < u_i, \forall i = 1, \dots, n$, it follows that $(l_i + u_i) \bar{x}_i - l_i u_i - \bar{x}_i^2 > 0, \forall i = 1, \dots, n$. Define $\epsilon := (l_1 + u_1) \bar{x}_1 - l_1 u_1 - \bar{x}_1^2$. Clearly, there exists $\delta \in \R$ such that $0 < \delta < \epsilon$. Let $\bar{X} \in \Sy^{n}$ be the matrix satisfying:
\begin{equation}
\label{SDP_GEIGR_dual_1}
\begin{array}{l}
\bar{X}_{11} = \dfrac{(l_1 + u_1) \bar{x}_1 - l_1 u_1 + \Phi_{11} \bar{x}_1^2 - \delta}{1 + \Phi_{11}}\\ 
\bar{X}_{ii} = \dfrac{(l_i + u_i) \bar{x}_i - l_i u_i + \Phi_{ii} \bar{x}_i^2}{1 + \Phi_{ii}} , \; i = 2, \dots , n, \\ 
\bar{X}_{ij} = \bar{X}_{ji} = \bar{x}_i \bar{x}_j, \; i = 1, \dots, n, \; j = i + 1, \dots, n
\end{array}
\end{equation}
where $\Phi_{ii}$ is the $(i,i)$-th entry of $A^T A$. It is simple to check that~\eqref{SDP_GEIGR_c4} is strictly satisfied by $(\bar{x}, \bar{X})$. Define $\hat{X} := \bar{X} - \bar{x} \bar{x}^T$. It is clear that $\hat{X}$ is diagonal with entries:
\begin{equation}
\label{SDP_GEIGR_dual_2}
\begin{array}{l}
\hat{X}_{11} = \dfrac{(l_1 + u_1) \bar{x}_1 - l_1 u_1 - \bar{x}_1^2 - \delta}{1 + \Phi_{11}} \\ 
\hat{X}_{ii} = \dfrac{(l_i + u_i) \bar{x}_i - l_i u_i - \bar{x}_i^2}{1 + \Phi_{ii}}, \; i = 2, \dots , n, \\ 
\end{array}
\end{equation}
Clearly, $\hat{X}_{ii} > 0$, $\forall i = 1, \dots, n$. Hence, $\hat{X} \succ$ and $(\bar{x}, \bar{X})$ is a strictly feasible solution to~\eqref{SDP_GEIGR}. Therefore, Slater's condition is satisfied by~(\ref{SDP_GEIGR}), which implies that strong duality holds and the optimal value of the dual problem is attained.

Now, we consider the dual of~\eqref{SDP_GEIGR}. By Proposition~\ref{prop:SDP_dual}, this dual is given by:
\begin{equation}
\label{SDP_GEIGR_dual_5}
\underset{\hat{\alpha}_g \in \R_{\geq 0}: Q_{\hat{\alpha}_g} \succcurlyeq 0}{\text{max}} \left\lbrace
\begin{aligned}
\centering
& \underset{x \in {\cal X}}{\text{min}}
& & x^T Q_{\hat{\alpha}_g} x + q_{\hat{\alpha}_g}^T x + k_{\hat{\alpha}_g} 
\end{aligned}
\right\rbrace
\end{equation}
where $\hat{\alpha}_g$ is the multiplier for the constraint~\eqref{SDP_GEIGR_c4}, $Q_{\hat{\alpha}_g} = Q + \hat{\alpha}_g (I_n + A^T A )$, $q_{\hat{\alpha}_g} = q - \hat{\alpha}_g ( l + u + 2 A^T b )$, and $k_{\hat{\alpha}_g} = \hat{\alpha}_g ( l^T u + b^T b )$. Since $Q$ is indefinite, Proposition~\ref{GEIGR_convexity} implies that $Q_{\hat{\alpha}_g} \succcurlyeq 0$ for $\hat{\alpha}_g \geq - \lambda_{\text{min}} (Q, I_n + A^T A)$. For any such choice of $\hat{\alpha}_g$, the inner minimization in~\eqref{SDP_GEIGR_dual_5} takes the same form as the QP~\eqref{GEIGR3}. Hence, the claim in (i) follows by weak duality. Let $\mu_{\text{DSDP\_GEIG}}$ be the optimal objective function value in~\eqref{SDP_GEIGR_dual_5}. Then, $\mu_{\text{DSDP\_GEIG}} = \max\{ \mu_{\text{GEIG}}(\hat{\alpha}_g) \,|\, \hat{\alpha}_g \in \R_{\geq 0}: Q_{\hat{\alpha}_g} \succcurlyeq 0\}$.  Clearly, $\mu_{\text{GEIG}}(\hat{\alpha}_g)$ is monotonically decreasing in $\hat{\alpha}_g$. Hence, $\mu_{\text{DSDP\_GEIG}} = \mu_{\text{GEIG}}(\hat{\alpha}_g)$ occurs when $\hat{\alpha}_g = -\lambda_{\text{min}} (Q, I_n + A^T A)$.  Combining this with $\mu_{\text{DSDP\_GEIG}} = \mu_{\text{SDP\_GEIG}}$, which holds by strong duality, proves the claim in (ii).
\end{proof}

\subsection{Eigenvalue relaxation in the nullspace of the equality constraints}
\label{sec:eigrel_ns}

In this section, we consider another convex quadratic relaxation of~\eqref{problem_statement} which also incorporates information from the equality constraints in order to convexify the objective function. This relaxation can be formulated as:
\begin{equation}
\label{EIGRNS1}
\begin{array}{cl}
\underset{x \in {\cal X}}{\text{min}}\;\; & x^T Q_{\alpha_{z}} x + q_{\alpha_{z}}^T x + k_{\alpha_{z}} 
\end{array}
\end{equation}
where $Q_{\alpha_{z}} = Q + \alpha_{z} I_n$, $q_{\alpha_{z}} = q - \alpha_{z} (l + u)$, $k_{\alpha_{z}} = \alpha_{z} l^T u$, and $\alpha_z \geq 0$. As discussed in \S\ref{sec:eigrel}, we must select a suitable $\alpha_z$ to ensure that~\eqref{EIGRNS1} is a convex relaxation of~\eqref{problem_statement}. As indicated previously, one such $\alpha_z$ can be determined by using the smallest eigenvalue of the matrix $Q$. However, as we show in the next proposition, there exists another method for constructing such $\alpha_z$ which makes use of the nullspace of $A$. 

\begin{proposition}
\label{EIGRNS_convexity}
Denote by $Z$ an orthonormal basis for the nullspace of the matrix $A$. Let $\alpha_z  \geq -\min(0, \lambda_{\emph{\text{min}}} (Z^T Q Z))$ in~\eqref{EIGRNS1}. Then,~\eqref{EIGRNS1} is a convex quadratic program when restricted to the nullspace of the matrix $A$.
\end{proposition}

\begin{proof}
Let $\H = \{x \in \R^n \; | \; Ax = b \}$, and denote by $r$ the rank of  $A$. It is clear than any point satisfying $Ax = b$ can be expressed as $x = x_h + Z x_z$, where $x_h \in \H$, $x_z \in \R^{n-r}$, and $Z \in \R^{n \times n-r}$. By using this transformation, we can write~\eqref{EIGRNS1} as:
\begin{equation}
\label{EIGRNS_convexity_1}
\begin{array}{cl}
\underset{x_z}{\text{min}}\;\; & {\left(x_h + Z x_z \right)}^T Q_{\alpha_{z}} \left(x_h + Z x_z \right) + q_{\alpha_{z}}^T \left(x_h + Z x_z \right) + k_{\alpha_{z}} \\
\st\;\;        & C \left(x_h + Z x_z \right) \leq d \\
               & l \leq \left(x_h + Z x_z \right) \leq u.              
\end{array}
\end{equation}

It is easily verified that~\eqref{EIGRNS_convexity_1} is convex for all $\alpha_z \geq -\min(0, \lambda_{\text{min}} (Z^T Q Z))$.
\end{proof}

From Proposition~\ref{EIGRNS_convexity}, it follows that the tightest relaxation of the form~\eqref{EIGRNS1} is obtained by setting $\alpha_z  = -\min(0, \lambda_{\text{min}} (Z^T Q Z))$. We refer to this convex relaxation of~\eqref{problem_statement} as the \textit{eigenvalue relaxation in the nullspace of $A$}. In the following proposition, we show that this relaxation is at least as tight as the generalized eigenvalue relaxation.

\begin{proposition}
\label{GEIGR_EIGRNS_comparison}
Assume that $\alpha_g = -\min(0, \lambda_{\emph{\text{min}}} (Q, I_n + A^T A))$ in~\eqref{GEIGR3} and $\alpha_z  = -\min(0, \lambda_{\emph{\text{min}}} (Z^T Q Z))$ in~\eqref{EIGRNS1}. Let $\mu_{\emph{\text{GEIG}}}$ and $\mu_{\emph{\text{EIGZ}}}$ denote the optimal objective function values in~\eqref{GEIGR3} and~\eqref{EIGRNS1}, respectively. Then, $\mu_{\emph{\text{EIGZ}}} \geq \mu_{\emph{\text{GEIG}}}$.
\end{proposition}

\begin{proof}
To prove that $\mu_{\text{EIGZ}} \geq \mu_{\text{GEIG}}$, it suffices to show that $\alpha_z \leq \alpha_g$. Similar to~\eqref{def_RayleighQuotient}, the smallest eigenvalue of $Z^TQZ$ can be expressed as:
\begin{equation}
\label{GEIGR_EIGRNS_comparison_2}
\begin{array}{cc}
\lambda_{\text{min}} (Z^T Q Z) = \underset{x \neq 0, Ax = 0}{\text{min}} \;\; \dfrac{x^T Q x}{x^T x}    
= \underset{x \neq 0, Ax = 0}{\text{min}} \;\; \dfrac{x^T Q x}{x^T \left( I_n + A^TA \right) x}
\end{array}
\end{equation}
where for the second equality we used the fact that the minimization is over vectors $x$ that lie in the null space of $A$. The restriction of vectors $x$ to the null space of $A$ also implies that 
$\lambda_{\text{min}} (Z^T Q Z) \geq \lambda_{\text{min}} (Q,I_n+A^TA)$.  This is easily seen by noting that the Rayleigh quotient expression for the generalized eigenvalue of the pair $(Q,I_n+A^TA)$ in~\eqref{def_RayleighQuotient} is over a larger domain. Hence, $\alpha_z \leq \alpha_g$, and $\mu_{\text{EIGZ}} \geq \mu_{\text{GEIG}}$. 
\end{proof}

From Proposition~\ref{GEIGR_EIGRNS_comparison}, it follows that the eigenvalue relaxation in the nullspace of $A$ can be potentially tighter than the generalized eigenvalue relaxation. However, the computation of $Z$ can be computationally expensive. Therefore, an important question is whether we can obtain a good approximation of $\lambda_{{\text{min}}} (Z^T Q Z)$ without having to explicitly compute $Z$. This question is addressed by the following proposition.

\begin{proposition}
\label{EIGRNS_approximation}
Let $\delta$ be a real scalar. Then, the following hold:

\begin{enumerate}[(i)]

\item If the matrix $Q$ is indefinite, $\lambda_{\emph{\text{min}}} (Q, I_n + \delta A^T A)$ is a strictly increasing function of $\delta$ for $\delta \geq 1$.

\item $\lim_{\delta \to\infty} \lambda_{\emph{\text{min}}} (Q, I_n + \delta A^T A) = \min(0, \lambda_{\emph{\text{min}}} (Z^T Q Z))$.

\end{enumerate}

\end{proposition}

\begin{proof}

We start with the proof of (i). Let $\delta_1, \delta_2 \in \R$ be two scalars such that $\delta_2 > \delta_1 \geq 1$. Define the set $D = \{ x \in \R^n : x \neq 0, x^T Q x < 0 \}$. Since the matrix $Q$ is indefinite by assumption, it is clear that $D \neq \emptyset$. From the definition of the set $D$, it is easy to check that the following inequality holds:
\begin{equation}
\label{EIGRNS_approximation_3}
\begin{array}{cc}
\dfrac{x^T Q x}{x^T \left( I_n + \delta_2 A^T A \right) x} > \dfrac{x^T Q x}{x^T \left( I_n + \delta_1 A^T A \right) x}, \;\;\; \forall x \in D
\end{array}
\end{equation}
Using the definition of the Rayleigh quotient in~\eqref{def_RayleighQuotient}, $D \neq \emptyset$  and~\eqref{EIGRNS_approximation_3}, it is simple to verify that $\lambda_{\text{min}} (Q, I_n + \delta_2 A^T A) > \lambda_{\text{min}} (Q, I_n + \delta_1 A^T A)$ which proves (i).

To prove (ii), consider the Rayleigh quotient definition in~\eqref{def_RayleighQuotient} for $(Q, I_n + \delta A^T A)$. Let $x = y + z$, where $y, z \in \R^n$ are orthogonal vectors which belong to the row space and nullspace of $A$, respectively. Then, by using this transformation in~\eqref{def_RayleighQuotient}, we have:
\begin{equation}
\label{EIGRNS_approximation_5}
\begin{array}{l}
\lim\limits_{\delta \to\infty} \lambda_{\text{min}} (Q, I_n + \delta A^T A) = \lim\limits_{\delta \to\infty} \underset{(y + z) \neq 0}{\text{min}} \;\; \dfrac{(y + z)^T Q (y + z)}{(y + z)^T (y + z) + \delta y^T A^T A y}. 
\end{array}
\end{equation}
To determine the limit in~\eqref{EIGRNS_approximation_5}, we consider the following cases:

\begin{enumerate}[(i)]

\item $y \neq 0$. In this case, we obtain:
\begin{equation}
\label{EIGRNS_approximation_6}
\begin{array}{l}
\underset{(y + z) \neq 0}{\text{min}} \lim\limits_{\delta \to\infty} \;\; \dfrac{(y + z)^T Q (y + z)}{(y + z)^T (y + z) + \delta y^T A^T A y} = 0.
\end{array}
\end{equation}
\item $y = 0$. In this case,~\eqref{EIGRNS_approximation_5} reduces to: 
\begin{equation}
\label{EIGRNS_approximation_7}
\begin{array}{l}
\lim\limits_{\delta \to\infty} \underset{z \neq 0}{\text{min}} \;\; \dfrac{z^T Q z}{z^T z} = \lim\limits_{\delta \to\infty} \underset{z \neq 0, Az = 0}{\text{min}} \;\; \dfrac{z^T Q z}{z^T z} = \lambda_{\text{min}} (Z^T Q Z).
\end{array}
\end{equation}
\end{enumerate}
Then, it follows that $\lim_{\delta \to\infty} \lambda_{\text{min}} (Q, I_n + \delta A^T A) = \min(0, \lambda_{\text{min}} (Z^T Q Z))$. 
\end{proof}

Proposition~\ref{EIGRNS_approximation} has very important consequences since it suggests we can approximate the bound given by the eigenvalue relaxation in the nullspace of $A$ by solving the following quadratic program for a sufficiently large value of $\delta$:
\begin{equation}
\label{GEIGR_delta_1}
\begin{array}{cl}
\underset{x \in {\cal X}}{\text{min}}\;\; & x^T Q x + q^T x + \alpha(\delta) (x^T x - (l + u)^T x + l^T u ) + \alpha(\delta) \cdot \delta \cdot \|Ax -b\|^2
\end{array}
\end{equation}
where $\alpha(\delta) = - \lambda_{\text{min}} (Q, I_n + \delta A^T A)$. Note that, for $\delta = 1$,~\eqref{GEIGR_delta_1} corresponds to the generalized eigenvalue relaxation introduced in \S\ref{sec:geigrel}.

Since $\lambda_{\text{min}} (Q, I_n + \delta A^T A)$ is a strictly increasing function of $\delta$ for $\delta \geq 1$, Proposition~\ref{EIGRNS_approximation} implies that as $\delta$ is increased, $\alpha(\delta)$ will converge to either $0$ or  $-\lambda_{\text{min}} (Z^T Q Z)$. The case in which $\alpha(\delta)$ converges to $0$ is particularly interesting since it indicates that $\lambda_{\text{min}} (Z^T Q Z) \geq 0$, and the continuous relaxation of~\eqref{problem_statement} is convex when restricted to the nullspace of $A$. Note that $\lambda_{\text{min}} (Q) < 0$ does not necessarily imply that $\lambda_{\text{min}} (Z^T Q Z) < 0$, and as a result, the continuous relaxation of~\eqref{problem_statement} may be convex when restricted to the nullspace of $A$, even if it is nonconvex in the space of the original problem variables.

The quadratic term $\alpha(\delta) \cdot \delta \cdot \|Ax -b\|^2$ vanishes for any $x$ feasible in~\eqref{GEIGR_delta_1}. This term is included in the objective function of~\eqref{GEIGR_delta_1} to ensure that $Q + \alpha(\delta) (I_n + \delta A^T A)$ is positive semidefinite. However, this term need not be included for~\eqref{GEIGR_delta_1} to be convex. Proposition~\ref{EIGRNS_convexity} implies that~\eqref{GEIGR_delta_1} is convex $\alpha(\delta) \geq -\min(0, \lambda_{\text{min}} (Z^T Q Z))$. The definition of $\alpha(\delta)$ and Proposition~\ref{EIGRNS_approximation} imply that $\alpha(\delta) \geq - \min(0, \lambda_{\text{min}} (Z^T Q Z))$ holds for any $\delta \geq 1$. 
As a result, the quadratic term $\alpha (\delta) \cdot \delta \cdot \|Ax -b\|^2$ can be dropped from the objective function of~\eqref{GEIGR_delta_1}, which simplifies this relaxation to:
\begin{equation}
\label{GEIGR_delta_2}
\begin{array}{cl}
\underset{x \in {\cal X}}{\text{min}}\;\; & x^T Q x + q^T x + \alpha(\delta) (x^T x - (l + u)^T x + l^T u ) \\
\end{array}
\end{equation}

This simplification has two significant advantages. First, it allows us to preserve the sparsity pattern defined by $Q$. Second, it prevents the relaxation from becoming ill-conditioned since $\delta$ does not figure in the objective function of~\eqref{GEIGR_delta_2} and is only used to determine $\alpha(\delta)$. We can use a simple iterative procedure to determine a value of $\delta$ which leads to a good approximation of the bound given by the eigenvalue relaxation in the nullspace of $A$. We detail such procedure in \S\ref{implementation}. 

By considering a quadratic relaxation of the form~\eqref{GEIGR_delta_2}, there is no need to project onto the nullspace of $A$. This is particularly advantageous in the context of the branching variable selection rules that we introduce in \S\ref{spectral_branching}, since the branching decisions are easier to interpret in the space of the original problem variables.

We finish this section by showing that the eigenvalue relaxation in the null space of $A$ is equivalent to the following SDP:
\begin{subequations}
\label{SDP_EIGRNS}
\begin{align}
\underset{x \in {\cal X}, X}{\text{min}}\;\; & \langle Q, X \rangle + q^T x \label{SDP_EIGRNS_obj} \\
\st\;\;     & X - x x^T \succcurlyeq 0 \label{SDP_EIGRNS_c3} \\
  			& \langle I_n, X \rangle - {\left( l + u\right)}^T x + l^T u \leq 0 \label{SDP_EIGRNS_c4} \\ 			 			
 			& \langle A^T A, X \rangle - {\left(2 A^T b \right)}^T x + b^T b = 0 \label{SDP_EIGRNS_c5} 			
\end{align}
\end{subequations}

\begin{proposition}
\label{EIGRNS_SDP_EIGRNS_equivalence}
Suppose that the matrix $Z^T Q Z$ is indefinite. Let $\mu_{\emph{\text{EIGZ}}}(\alpha_z)$ and $\mu_{\emph{\text{SDP\_EIGZ}}}$ be the optimal objective function values in~\eqref{EIGRNS1} and~\eqref{SDP_EIGRNS}, respectively. Then, the following hold:
\begin{enumerate}[(i)]

\item If $\alpha_z \geq -\lambda_{\emph{\text{min}}} (Z^T Q Z)$ in~\eqref{EIGRNS1}, then $\mu_{\emph{\text{EIGZ}}}(\alpha_z) \leq \mu_{\emph{\text{SDP\_EIGZ}}}$.
\item If $\alpha_z = -\lambda_{\emph{\text{min}}} (Z^T Q Z)$ in~\eqref{EIGRNS1}, then $\mu_{\emph{\text{EIGZ}}}(\alpha_z) = \mu_{\emph{\text{SDP\_EIGZ}}}$.

\end{enumerate}
\end{proposition}

\begin{proof}

Unlike~\eqref{SDP_EIGR} and~\eqref{SDP_GEIGR},~\eqref{SDP_EIGRNS} does not admit a strictly feasible solution. To illustrate this, note that, for any $x$ satisfying $A x = b$, \eqref{SDP_EIGRNS_c5} can be written as:
\begin{subequations}
\label{SDP_EIGRNS_equivalence_1}
\begin{align}
 & \langle A^T A, X \rangle - {\left(2 A^T b \right)}^T x + b^T b + \langle A^T A, xx^T \rangle - \langle A^T A, xx^T \rangle = 0 \label{SDP_EIGRNS_equivalence_1_c1} \\
\implies & \langle A^T A, X - xx^T \rangle = 0 \label{SDP_EIGRNS_equivalence_1_c3}
\end{align}
\end{subequations}
which implies that $X - xx^T$ cannot be positive definite for the pairs $(x,X)$ that are feasible in~\eqref{SDP_EIGRNS}. It follows that we cannot apply the strong duality theorem to~\eqref{SDP_EIGRNS}. As a result, the proof of this proposition relies on different arguments from those used in the proofs of Propositions~\ref{EIGR_SDP_EIGR_equivalence} and~\ref{GEIGR_SDP_GEIGR_equivalence}. We proceed in two steps:

\begin{enumerate}[(a)]

\item We show that the optimal objective function value of the dual problem of~\eqref{SDP_EIGRNS} provides an upper bound on~\eqref{EIGRNS1} when $\alpha_z \geq -\lambda_{{\text{min}}} (Z^T Q Z)$, and is equal to $\mu_\text{EIGZ}(\alpha_z)$ when $\alpha_z = -\lambda_{{\text{min}}} (Z^T Q Z)$. By weak duality, this implies that $\mu_{\text{SDP\_EIGZ}} \geq \mu_{\text{EIGZ}}(\alpha_z)$ for $\alpha_z \geq -\lambda_{{\text{min}}} (Z^T Q Z)$, proving the claim in (i).

\item We construct a feasible solution for~\eqref{SDP_EIGRNS} which attains the same objective function value as an optimal solution of~\eqref{EIGRNS1} when $\alpha_z = -\lambda_{{\text{min}}} (Z^T Q Z)$. This implies that $\mu_{\text{SDP\_EIGZ}} \leq \mu_{\text{EIGZ}}(\alpha_z)$ when $\alpha_z = -\lambda_{{\text{min}}} (Z^T Q Z)$. This observation combined with the result from (a) completes the proof of the claim in (ii).

\end{enumerate}

To prove (a), we use Proposition~\ref{prop:SDP_dual} to write the dual of~\eqref{SDP_EIGRNS} as:
\begin{equation}
\label{SDP_EIGRNS_equivalence_4}
\underset{\hat{\alpha}_z \in \R_{\geq 0}, \hat{\beta}_z \in \R: Q_{\hat{\alpha}_z, \hat{\beta}_z} \succcurlyeq 0}{\text{max}} \left\lbrace
\begin{aligned}
\centering
& \underset{x \in {\cal X}}{\text{min}}
& & x^T Q_{\hat{\alpha}_z, \hat{\beta}_z} x + q_{\hat{\alpha}_z, \hat{\beta}_z}^T x + k_{\hat{\alpha}_z, \hat{\beta}_z}  
\end{aligned}
\right\rbrace
\end{equation}
where $\hat{\alpha}_z$ and $\hat{\beta}_z$ are multipliers for~\eqref{SDP_EIGRNS_c4} and \eqref{SDP_EIGRNS_c5}, respectively, $Q_{\hat{\alpha}_z, \hat{\beta}_z} = Q + \hat{\alpha}_z I_n + \hat{\beta}_z A^T A$, $q_{\hat{\alpha}_z, \hat{\beta}_z} = q - \hat{\alpha}_z (l + u) - 2 \hat{\beta}_z A^T b$, and $k_{\hat{\alpha}_z, \hat{\beta}_z} = \hat{\alpha}_z l^T u + \hat{\beta}_z b^T b$. Let $\hat{\delta}_z = \hat{\beta}_z / \hat{\alpha}_z $. By substituting $\hat{\beta}_z = \hat{\delta}_z \hat{\alpha}_z$ in~\eqref{SDP_EIGRNS_equivalence_4}, the dual becomes:
\begin{equation}
\label{SDP_EIGRNS_equivalence_5}
\underset{\hat{\alpha}_z \in \R_{\geq 0}, \hat{\delta}_z \in \R : Q_{\hat{\alpha}_z, \hat{\delta}_z} \succcurlyeq 0}{\text{max}} \left\lbrace
\begin{aligned}
\centering
& \underset{x \in {\cal X}}{\text{min}}
& & x^T Q_{\hat{\alpha}, \hat{\delta}_z} x + q_{\hat{\alpha}, \hat{\delta}_z}^T x + k_{\hat{\alpha}, \hat{\delta}_z}  
\end{aligned}
\right\rbrace
\end{equation}
where $Q_{\hat{\alpha}_z, \hat{\delta}_z} = Q + \hat{\alpha}_z (I_n + \hat{\delta}_z A^T A )$, $q_{\hat{\alpha}_z, \hat{\delta}_z} = q - \hat{\alpha}_z ( l + u + 2 \hat{\delta}_z A^T b)$, and $k_{\hat{\alpha}_z, \hat{\delta}_z} = \hat{\alpha}_z ( l^T u + \hat{\delta}_z b^T b )$. Note that the quadratic term $\hat{\alpha}_z \hat{\delta}_z \|Ax - b\|^2$ vanishes for any $x$ feasible in the inner minimization problem. As a result,~\eqref{SDP_EIGRNS_equivalence_5} can be posed as:
\begin{equation}
\label{SDP_EIGRNS_equivalence_6}
\underset{\hat{\alpha}_z \in \R_{\geq 0}, \hat{\delta}_z \in \R : Q_{\hat{\alpha}_z, \hat{\delta}_z} \succcurlyeq 0}{\text{max}} \left\lbrace
\begin{aligned}
\centering
& \underset{x \in {\cal X}}{\text{min}}
& & x^T Q_{\hat{\alpha}_{z}} x + q_{\hat{\alpha}_{z}}^T x + k_{\hat{\alpha}_{z}}
\end{aligned}
\right\rbrace
\end{equation}
where $Q_{\hat{\alpha}_{z}} = Q + \hat{\alpha}_{z} I_n$, $q_{\hat{\alpha}_{z}} = q - \hat{\alpha}_{z} (l + u)$, and $k_{\hat{\alpha}_{z}} = \hat{\alpha}_{z} l^T u$. Since $Z^T Q Z$ is indefinite, $Q$ is indefinite as well. Proposition~\ref{EIGRNS_convexity} implies that, for a given value of $\hat{\delta}_z$, $Q_{{\alpha}_z, \hat{\delta}_z} \succcurlyeq 0$ when $\hat{\alpha}_z \geq - \lambda_{\text{min}} (Q, I_n + \hat{\delta}_z A^T A)$. For all such $\hat{\alpha}_z$, $\hat{\delta}_z$ it is simple to check that the objective function of the inner minimization problem of~\eqref{SDP_EIGRNS_equivalence_6} is monotonically decreasing in $\hat{\alpha}_z$. By using the fact that $\lambda_{\text{min}} (Z^T Q Z) < 0$ and Proposition~\ref{EIGRNS_approximation}, it is easy to verify that the maximum of~\eqref{SDP_EIGRNS_equivalence_6} is attained when $\hat{\alpha}_z = -\lim_{\hat{\delta}_z \to\infty} \lambda_{\text{min}} (Q, I_n + \hat{\delta}_z A^T A) = -\lambda_{\text{min}} (Z^T Q Z)$. Let $\mu_{\text{DSDP\_EIGZ}}$ be the optimal objective function value in~\eqref{SDP_EIGRNS_equivalence_6}. Since the inner problem in~\eqref{SDP_EIGRNS_equivalence_6} has the same form as the QP~\eqref{EIGRNS1}, it is simple to check that $\mu_{\text{DSDP\_EIGZ}} \geq \mu_{\text{EIGZ}}(\alpha_z)$ when $\alpha_z \geq -\lambda_{{\text{min}}} (Z^T Q Z)$, and $\mu_{\text{DSDP\_EIGZ}} = \mu_{\text{EIGZ}}(\alpha_z)$ when $\alpha_z = -\lambda_{{\text{min}}} (Z^T Q Z)$. By weak duality, it follows that $\mu_{\text{SDP\_EIGZ}} \geq \mu_{\text{EIGZ}}(\alpha_z)$ for $\alpha_z \geq -\lambda_{{\text{min}}} (Z^T Q Z)$. This proves the claim in (i).

Next, we prove (b). Let $\hat{x}$ denote an optimal solution of~\eqref{EIGRNS1} when $\alpha_z = -\lambda_{{\text{min}}} (Z^T Q Z)$. Define $\hat{X} = \hat{x} \hat{x}^T + \gamma Z v (Z v)^T$, where $\gamma = (l + u)^T \hat{x} - l^T u - \hat{x}^T \hat{x}$, and $v$ denotes the eigenvector corresponding to the smallest eigenvalue of $Z^T Q Z$. We first show that $(\hat{x}, \hat{X})$ is feasible in~\eqref{SDP_EIGRNS}. By definition, $\hat{x} \in {\cal X}$. Consider~\eqref{SDP_EIGRNS_c3}. Recall that the concave envelope of $x_i^2$ over $[l_i, u_i]$ is given by $(l_i + u_i) x_i - l_i u_i$. As a result, it is clear that each term $(l_i + u_i) \hat{x}_i - l_i u_i - \hat{x}_i^2$ is nonnegative, which in turn implies that $\gamma \geq 0$. Moreover, since $Z v (Z v)^T \succcurlyeq 0$, it follows that $\hat{X} - \hat{x} \hat{x}^T \succcurlyeq 0$.

Consider~\eqref{SDP_EIGRNS_c4} and~\eqref{SDP_EIGRNS_c5}. Substituting $(\hat{x}, \hat{X})$ in~\eqref{SDP_EIGRNS_c4}, we obtain:
\begin{equation}
\begin{aligned}
&\; \langle I_n, \hat{x} \hat{x}^T + \gamma Z v (Z v)^T \rangle - {\left( l + u\right)}^T \hat{x} + l^T u \\
=&\; \hat{x}^T \hat{x} + \gamma v^T Z^T Z v - {\left( l + u\right)}^T \hat{x} + l^T u 
= \hat{x}^T \hat{x} + \gamma - {\left( l + u\right)}^T \hat{x} + l^T u = 0.
\end{aligned}
\notag
\end{equation}
Similarly, substituting $(\hat{x}, \hat{X})$ in~\eqref{SDP_EIGRNS_c5} yields: 
\begin{equation}
\begin{aligned}
&\; \langle A^T A, \hat{x} \hat{x}^T + \gamma Z v (Z v)^T \rangle - {\left(2 A^T b \right)}^T \hat{x} + b^T b \\
= &\; \hat{x}^T A^T A \hat{x} - {\left(2 A^T b \right)}^T \hat{x} + b^T b + \gamma v^T Z^T A^T A Z v 
= (A \hat{x} - b)^T (A \hat{x} - b) = 0. 
\end{aligned}
\notag
\end{equation}
Let $f(x, X)$ be the objective function of~\eqref{SDP_EIGRNS}. The value of $f$ at $(\hat{x}, \hat{X})$ is:
\begin{equation}
\label{SDP_EIGRNS_equivalence_7}
\begin{aligned}
f(\hat{x}, \hat{X}) & = \langle Q, \hat{x} \hat{x}^T + \gamma Z v (Z v)^T \rangle + q^T \hat{x} \\
& = \hat{x}^T Q \hat{x} - \gamma \alpha_z + q^T \hat{x} \\
& = \hat{x}^T (Q + \alpha_z I_n) \hat{x} + (q - \alpha_z (l + u))^T \hat{x} + \alpha_z l^T u = \mu_{\text{EIGZ}}
\end{aligned}
\end{equation}
where we have used the fact that $v^T Z^T Q Z v = -\alpha_z = \lambda_{\text{min}} (Z^T Q Z)$. Since $(\hat{x}, \hat{X})$ is feasible in~\eqref{SDP_EIGRNS}, from~\eqref{SDP_EIGRNS_equivalence_7} it follows that $\mu_{\text{SDP\_EIGZ}} \leq \mu_{\text{EIGZ}}(\alpha_z)$ when $\alpha_z = -\lambda_{\text{min}} (Z^T Q Z)$. This observation combined with the result from (a) proves the claim in (ii), by showing that $\mu_{\text{SDP\_EIGZ}} = \mu_{\text{EIGZ}}(\alpha_z)$ when $\alpha_z = -\lambda_{\text{min}} (Z^T Q Z)$.
\end{proof}

\subsection{Further insights into the proposed quadratic relaxations}
\label{sec:further_insights}

The relaxations introduced in \S\ref{sec:eigrel}--\ref{sec:eigrel_ns} can be derived through the following four-step recipe:

\begin{enumerate}[(R1)]
    \item\label{recipe:step1} identify a (possibly empty) set $\J$ of quadratic functions of the form $f_j(x) = x^TS_jx + s_j^Tx + \eta_j$, where $S_j \in \Sy^n, s_j \in \R^n, \eta_j \in \R$, such that $f_j(x) = 0$ for $x \in \Omega := \{x \in \R^n \,|\, Ax = b\}$;
    \item\label{recipe:step2} construct an initial relaxation for~\eqref{problem_statement} as
\begin{equation}
\label{initial_relaxation_recipe}
\begin{array}{cl}
 \underset{x \in {\cal X}}{\text{min}}\;\; & x^TQx + q^Tx + \alpha (x^Tx - (l+u)^Tx + l^Tu) +  \sum_{j \in \J} \beta_j f_j(x)
\end{array}
\end{equation}
    where $\alpha \in \R_{\geq 0}, \beta_j \in \R$, such that $Q + \alpha I_n + \sum_{j \in \J} \beta_j S_j \succcurlyeq 0$;
    \item\label{recipe:step3} find $\alpha^*, \beta^*$ such that the bound given by the relaxation~\eqref{initial_relaxation_recipe} is maximized
    \begin{equation}
    \label{EIGRClass1}
    (\alpha^*,\beta^*) = \text{arg}
    \underset{\alpha \in \R_{\geq 0}, \beta \in \R^{|\J|}: Q_{\alpha,\beta} \succcurlyeq 0}{\text{max}} \left\lbrace
    \begin{aligned}
    \centering
    & \underset{x \in {\cal X}}{\text{min}}
    & & x^T Q_{\alpha} x + q_{\alpha}^T x + k_{\alpha} 
    \end{aligned}
    \right\rbrace
    \end{equation}
    where $Q_{\alpha} = Q + \alpha I_n$, $Q_{\alpha,\beta} = Q_{\alpha} + \sum_{j \in \J} \beta_j S_j$, $q_{\alpha} = q - \alpha (l+u)$, $k_{\alpha} = \alpha l^Tu$, and $\beta$ is the $|\J|$-dimensional vector whose entries are the parameters $\beta_j$;
    \item\label{recipe:step4} obtain the relaxation
    \begin{equation}
    \label{EIGRClass2}
    \begin{aligned}
    \centering
    & \underset{x \in {\cal X}}{\text{min}}
    & & x^T Q_{\alpha^*} x + q_{\alpha^*}^T x + k_{\alpha^*}. 
    \end{aligned}
    \end{equation}
\end{enumerate}
Observe that the parameters $\beta_j$ are not present in the objective function of the inner minimization problem in~\eqref{EIGRClass1} and the objective function in~\eqref{EIGRClass2} since $f_j(x) = 0$ for $x \in {\cal X} \subset \Omega$ (due to~(R\ref{recipe:step1})).  The three spectral relaxations presented in \S\ref{sec:eigrel}--\ref{sec:eigrel_ns} can be identified with~\eqref{EIGRClass1} by noting that:
\begin{itemize}
    \item $\J = \emptyset$, $\alpha^* = -\min(0,\lambda_{{\text{min}}} (Q))$ for the eigenvalue relaxation~\eqref{EIGR3}; 
    \item $\J = \{1\}$, $f_1(x) = \sum_{i=1}^m (A_{i\cdot} x - b_i)^2$, $\alpha^* = -\min(0,\lambda_{{\text{min}}} (Q, I_n + A^T A))$, $\beta_1^* = \alpha^*$ for the generalized eigenvalue relaxation~\eqref{GEIGR3}.  Note that in this case a further restriction that $\beta_1 = \alpha$ is imposed in~\eqref{EIGRClass1}; and 
    \item $\J = \{1\}$, $f_1(x) = \sum_{i=1}^m (A_{i\cdot} x - b_i)^2$, $\alpha^*  = -\min(0,\lambda_{{\text{min}}} (Z^T Q Z))$ and $\beta_1^* = +\infty$ for the eigenvalue relaxation on the nullspace of $A$~\eqref{EIGRNS1}.
\end{itemize}
From Propositions~\ref{EIGR_GEIGR_comparison} and~\ref{GEIGR_EIGRNS_comparison} we know that the lower bound obtained from the eigenvalue relaxation in the nullspace of $A$~\eqref{EIGRNS1} is at least as large as those provided by the other spectral relaxations. Further, the computation of $\alpha^*$ can be done efficiently.

The recipe~(R\ref{recipe:step1})-(R\ref{recipe:step4}) is preferable from a computational standpoint since the resulting relaxation is a quadratic program inheriting the sparsity of the problem. However, the step~(R\ref{recipe:step1}) allows for other choice for the functions $f_j(x)$ that have been considered in the literature (see Faye and Roupin~\cite{fr:07}). Some examples for the functions satisfying~(R\ref{recipe:step1}) are~\cite{fr:07}: $\left(x_j(A_{i\cdot}x - b_i)\right)$, $\left((A_{j\cdot}x - b_j)(A_{i\cdot}x-b_i)\right)$, $\left(x^TA_{j\cdot}^TA_{i\cdot}x - b_jb_i\right)$. 
This naturally raises the question: \emph{Can we improve on the bound provided by~\eqref{EIGRNS1} when restricted to the class of relaxations in~\eqref{EIGRClass1}}? In the rest of the section, we show that we cannot improve on the bound provided by the eigenvalue relaxation on the nullspace of $A$~\eqref{EIGRNS1}.  Thus, establishing that~\eqref{EIGRNS1} is the best among the class of relaxations in~\eqref{EIGRClass1}.

We begin by recalling the properties of functions satisfying~(R\ref{recipe:step1}). 
\begin{proposition}\label{prop:null_quad}
Let $f(x) = x^TSx + s^Tx + \eta$ be a quadratic function.  Then, $f(x) = 0$ for all $x \in \Omega := \{ x \in \R^n \,|\, Ax = b\}$ if and only if $S = A^TW^T + WA$, $s = A^T \nu - 2Wb$, $\eta = -b^T \nu$ for some $W \in \R^{n \times m}$ and $\nu \in \R^m$.
\end{proposition}
\begin{proof}
This follows from Theorem 1 in~\cite{fr:07}.
\end{proof}
Following Proposition~\ref{prop:null_quad}, we 
assume without loss of generality that $S_j = A^TW_j^T + W_jA$ for some $W_j \in \R^{n \times m}$ in the rest of this section.
We will compare the relaxations in the class~\eqref{EIGRClass1} with the eigenvalue relaxation in the nullspace of $A$~\eqref{EIGRNS1} through the respective SDP formulations.  To this end, consider the SDP:
\begin{subequations}
\label{SDP_EIGRJ}
\begin{align}
\underset{x \in {\cal X}, X}{\text{min}}\;\; & \langle Q, X \rangle + q^T x \label{SDP_EIGRJ_obj} \\
\st\;\;     & X - x x^T \succcurlyeq 0 \label{SDP_EIGRJ_c3} \\
  			& \langle I_n, X \rangle - {\left( l + u\right)}^T x + l^T u \leq 0 \label{SDP_EIGRJ_c4} \\ 			 			
 			& \langle S_j, X \rangle + s_j^T x + \eta_j = 0, \; j \in \J. \label{SDP_EIGRJ_c5} 
\end{align}
\end{subequations}
The next proposition shows that SDP~\eqref{SDP_EIGRJ} is the dual of~\eqref{EIGRClass1}.
\begin{proposition}\label{prop:sdp_dual_J}
Let $\J \neq \emptyset$ be a set of quadratic functions satisfying (R\ref{recipe:step1}). The dual of the SDP~\eqref{SDP_EIGRJ} is given by~\eqref{EIGRClass1}.
\end{proposition}
\begin{proof}
By dualizing the constraints~\eqref{SDP_EIGRJ_c4} and~\eqref{SDP_EIGRJ_c5} with the multipliers $\alpha \in \R_{\geq 0}$ and $\beta_{j} \in \R, j \in \J$, respectively, we can use Proposition~\ref{prop:SDP_dual} to obtain the claim.
\end{proof}

The next result shows that the feasible set of the 
SDP~\eqref{SDP_EIGRNS} is in general a subset of the feasible set of the SDP~\eqref{SDP_EIGRJ}.  Further, we provide conditions on the choice of quadratic functions in $\J$ so that equality holds. 
\begin{proposition}\label{prop:sdp_relation_Z_J}
Let ${\cal F}_{\emph{\text{SDP{\_}EIGZ}}}$ and ${\cal F}_{\emph{\text{SDP\_EIGJ}}}$ denote the feasible regions of the SDPs in~\eqref{SDP_EIGRNS} and~\eqref{SDP_EIGRJ}, respectively.  Then, the following holds:
\begin{enumerate}[(i)]
    \item ${\cal F}_{\emph{\text{SDP\_EIGZ}}} \subseteq {\cal F}_{\emph{\text{SDP\_EIGJ}}}$.
    \item If $\exists \, \omega_j, j \in \J$ such that $\sum_{j \in \J} \omega_j W_j = A^T$ then ${\cal F}_{\emph{\text{SDP\_EIGZ}}} = {\cal F}_{\emph{\text{SDP\_EIGJ}}}$.
\end{enumerate}
\end{proposition}
\begin{proof}
We start by proving (i).  Recall from~\eqref{SDP_EIGRNS_equivalence_1} that any $(\bar{x},\bar{X}) \in {\cal F}_{\text{SDP\_EIGZ}}$ satisfies $\langle A^TA, \bar{X} - \bar{x} \bar{x}^T \rangle = 0$.  Hence, $\bar{X}$ takes the form $\bar{X} = \bar{x} \bar{x}^T + ZVZ^T$ 
for all $(\bar{x},\bar{X}) \in {\cal F}_{\text{SDP\_EIGZ}}$, where $Z \in \R^{n \times n-r}$ is a basis for the null space of $A$ and $V \in \Sy^{n-r}$. For any $(\bar{x},\bar{X}) \in {\cal F}_{\text{SDP\_EIGZ}}$ it follows that for all $j \in \J$:
\begin{subequations}
\label{SDP_EIGRNS_EIGRJ_1}
\begin{align}
 \;\; & \langle S_j, \bar{X} \rangle + s_j^T \bar{x} + \eta_j  \label{SDP_EIGRNS_EIGRJ_1_c1} \\
=\;\; & \langle S_j, \bar{X} - \bar{x}\bar{x}^T \rangle + \bar{x}^T S_j \bar{x} + s_j^T \bar{x} + \eta_j  \label{SDP_EIGRNS_EIGRJ_1_c2} \\
=\;\; & \langle S_j, \bar{X} - \bar{x}\bar{x}^T \rangle = \langle A^TW_j^T + W_jA, ZVZ^T \rangle \label{SDP_EIGRNS_EIGRJ_1_c3} = 0
\end{align}
\end{subequations}
where~\eqref{SDP_EIGRNS_EIGRJ_1_c2} follows from adding 
and subtracting $\bar{x}^T S_j \bar{x}$, the first equality in~\eqref{SDP_EIGRNS_EIGRJ_1_c3} follows from~(R\ref{recipe:step1}), the second equality in~\eqref{SDP_EIGRNS_EIGRJ_1_c3} from Proposition~\ref{prop:null_quad} and the final equality due to $Z$ being a basis for the nullspace of $A$. Thus $(\bar{x},\bar{X}) \in {\cal F}_{\text{SDP\_EIGJ}}$ proving the claim in (i).

Consider the claim in (ii).  Suppose that there exist $\omega_j, j \in \J$ such that the condition in (ii) holds. We perform a linear combination of the inequalities in~\eqref{SDP_EIGRJ_c5} using $\omega_j$ to obtain for any $(\bar{x},\bar{X}) \in {\cal F}_{\text{SDP\_EIGJ}}$: 
\begin{subequations}
\label{SDP_EIGRNS_EIGRJ_2}
\begin{align}
0 = \;\; & \sum_{j \in \J} \omega_j \left(\langle S_j, \bar{X} \rangle + s_j^T \bar{x} + \eta_j \right)  \label{SDP_EIGRNS_EIGRJ_2_c1} \\
=\;\; & \sum_{j \in \J} \omega_j \left( \langle S_j, \bar{X} - \bar{x} \bar{x}^T \rangle + \bar{x}^T S_j \bar{x} + s_j^T \bar{x} + \eta_j \right) \label{SDP_EIGRNS_EIGRJ_2_c2} \\
=\;\; & \sum_{j \in \J} \omega_j\langle S_j, \bar{X} - \bar{x} \bar{x}^T \rangle = 2 \langle A^TA, \bar{X} - \bar{x} \bar{x}^T \rangle \label{SDP_EIGRNS_EIGRJ_2_c3}
\end{align}
\end{subequations}
where~\eqref{SDP_EIGRNS_EIGRJ_2_c2} follows from adding 
and subtracting $\bar{x}^T S_j \bar{x}$, the first equality in~\eqref{SDP_EIGRNS_EIGRJ_2_c3} follows from~(R\ref{recipe:step1}), the second equality in~\eqref{SDP_EIGRNS_EIGRJ_2_c3} from Proposition~\ref{prop:null_quad} and the condition in (ii). 
Thus $(\bar{x},\bar{X}) \in {\cal F}_{\text{SDP\_EIGZ}}$ proving the claim in (ii).
\end{proof}
Faye and Roupin~\cite{fr:07} proved the equivalence between the SDP~\eqref{SDP_EIGRNS} and a similar SDP where~\eqref{SDP_EIGRNS_c5} is replaced by the constraints derived by lifting the quadratic functions $x_j(A_{i \cdot} x-b_i) = 0$, $i = 1,\ldots,m$, $j = 1,\ldots,n$ to the space of $(x,X)$. Proposition~\ref{prop:sdp_relation_Z_J} considerably expands the set of quadratic functions for which the feasible set of the resulting SDP is equal to ${\cal F}_{\text{SDP\_EIGZ}}$ (claim in (ii)). It is easy to verify that all of the examples of quadratic functions satisfying (R\ref{recipe:step1}) described in~\cite{fr:07} do satisfy the condition in (ii).  Further, the claim in (i) shows that there exist no quadratic functions satisfying (R\ref{recipe:step1}) for which the resulting SDP can have a smaller feasible region than the SDP~\eqref{SDP_EIGRNS}. This brings us to the main result on the claim that the relaxation~\eqref{EIGRNS1} is indeed the best among the class of relaxations in~\eqref{EIGRClass1}.
\begin{theorem}
Suppose that $Z^T Q Z$ is indefinite and that the set $\J$ is chosen such that (R\ref{recipe:step1}) holds. Assume that $\alpha_z = -\lambda_{\emph{\text{min}}} (Z^T Q Z)$ in~\eqref{EIGRNS1}.  Denote by $\mu_{\emph{\text{EIGZ}}}$ and $\mu_{\emph{\text{EIGJ}}}$ the optimal objective function values in~\eqref{EIGRNS1} and~\eqref{EIGRClass1}, respectively. Then, $\mu_{\emph{\text{EIGJ}}} \leq \mu_{\emph{\text{EIGZ}}}$.
\end{theorem}
\begin{proof}
Let $\mu_{\text{SDP\_EIGZ}}$ and $\mu_{\text{SDP\_EIGJ}}$ be the optimal objective values of~\eqref{SDP_EIGRNS} and~\eqref{SDP_EIGRJ}, respectively.  
By Proposition~\ref{EIGRNS_SDP_EIGRNS_equivalence}(ii) we have that $\mu_{\text{EIGZ}} = \mu_{\text{SDP\_EIGZ}}$.  
By Proposition~\eqref{prop:sdp_relation_Z_J}(i) we have that 
$\mu_{\text{SDP\_EIGJ}} \leq \mu_{\text{SDP\_EIGZ}}$.  
By Proposition~\ref{prop:sdp_dual_J} and weak duality we have that $\mu_{\text{EIGJ}} \leq \mu_{\text{SDP\_EIGJ}}$. 
Hence, $\mu_{\text{EIGJ}} \leq \mu_{\text{EIGZ}}$.  
\end{proof}
We finish this section by providing a theoretical comparison between the spectral relaxations studied in \S\ref{sec:eigrel}--\ref{sec:eigrel_ns} and some SDP relaxations described in \S\ref{sec:semidefinite_relaxations}.

\begin{theorem}
\label{Theorem_2}
Assume that $Z^T Q Z$ is indefinite. Suppose that $\alpha_e = -\lambda_{\emph{\text{min}}} (Q)$ in~\eqref{EIGR3}, $\alpha_g = -\lambda_{\emph{\text{min}}} (Q, I_n + A^T A)$ in~\eqref{GEIGR3}, and $\alpha_z  = -\lambda_{\emph{\text{min}}} (Z^T Q Z)$ in~\eqref{EIGRNS1}. Denote by $\mu_{\emph{\text{EIG}}}$, $\mu_{\emph{\text{GEIG}}}$, $\mu_{\emph{\text{EIGZ}}}$, $\mu_{\emph{\text{SDP\_d}}}$, $\mu_{\emph{\text{SDP\_dax}}}$, and $\mu_{\emph{\text{SDP\_da}}}$ the optimal objective function values of~\eqref{EIGR3}, \eqref{GEIGR3}, \eqref{EIGRNS1}, \eqref{SDP_d}, \eqref{SDP_d_a_x}, and \eqref{SDP_d_a}, respectively. Then, the following hold:

\begin{enumerate}[(i)]

\item $\mu_{\emph{\text{SDP\_d}}} \geq \mu_{\emph{\text{EIG}}}$.

\item $\mu_{\emph{\text{SDP\_dax}}} = \mu_{\emph{\text{SDP\_da}}} \geq \mu_{\emph{\text{EIGZ}}} \geq \mu_{\emph{\text{GEIG}}} \geq \mu_{\emph{\text{EIG}}}$.

\end{enumerate}

\end{theorem}

\begin{proof}
We start by proving (i). Denote by $\mu_{\text{SDP\_EIG}}$ the optimal objective function value in~\eqref{SDP_EIGR}. By Proposition~\ref{EIGR_SDP_EIGR_equivalence}(ii), we have that $\mu_{{\text{EIG}}} = \mu_{\text{SDP\_EIG}}$. Hence, we can prove (ii) by comparing the SDPs~\eqref{SDP_d} and~\eqref{SDP_EIGR}. The constraints~\eqref{SDP_d_c2} and~\eqref{SDP_d_c3} are implied by~\eqref{SDP_d_c1}, and as a result, can be droped from~\eqref{SDP_d}. Therefore, \eqref{SDP_d} and \eqref{SDP_EIGR} only differ in the constraints~\eqref{SDP_d_c4} and~\eqref{SDP_EIGR_c4}. It is simple to verify that the inequality~\eqref{SDP_EIGR_c4} can be obtained by aggregating the McCormick inequalities~\eqref{SDP_d_c4}, which implies that $\mu_{{\text{SDP\_d}}} \geq \mu_{\text{SDP\_EIG}} = \mu_{\emph{\text{EIG}}}$.

Now, we prove (ii). As stated in \S\ref{sec:separable_programming_relaxation}, the relationship $\mu_{{\text{SDP\_dax}}} = \mu_{{\text{SDP\_da}}}$ follows from a result given in~\cite{fr:07}. To show that $\mu_{{\text{SDP\_da}}} \geq \mu_{{\text{EIGZ}}}$, we follow the same line of arguments used for proving (i). Let $\mu_{\text{SDP\_EIGZ}}$ be the optimal objective function value in~\eqref{SDP_EIGRNS}. Proposition~\ref{EIGRNS_SDP_EIGRNS_equivalence}(ii) implies that $\mu_{{\text{EIGZ}}} = \mu_{\text{SDP\_EIGZ}}$. Therefore, to prove (ii), we can simply compare the SDPs~\eqref{SDP_d_a} and~\eqref{SDP_EIGRNS}. The constraints~\eqref{SDP_d_c2} and~\eqref{SDP_d_c3} are redundant in~\eqref{SDP_d_a}, and can be dropped from this formulation. Similar to the previous case, \eqref{SDP_d_a} and~\eqref{SDP_EIGRNS} only differ in the constraints~\eqref{SDP_d_c4} and~\eqref{SDP_EIGRNS_c4}. As stated above, the inequality~\eqref{SDP_EIGRNS_c4} is implied by the inequalities~\eqref{SDP_d_c4}. Hence, $\mu_{{\text{SDP\_da}}} \geq \mu_{\text{SDP\_EIGZ}} = \mu_{\emph{\text{EIGZ}}}$. The inequalities $\mu_{\emph{\text{GEIG}}} \geq \mu_{\emph{\text{EIG}}}$ and $\mu_{\emph{\text{EIGZ}}} \geq \mu_{\emph{\text{GEIG}}}$, follow from Propositions~\ref{EIGR_GEIGR_comparison} and~\ref{GEIGR_EIGRNS_comparison}, respectively.
\end{proof}

\section{Spectral branching for nonconvex binary QPs}
\label{spectral_branching}

In this section, we introduce new eigenvalue-based branching variable selection strategies for nonconvex binary QPs. These strategies are inspired by the strong branching rule which was initially proposed for mixed-integer linear programs~\cite{abcc:95}, and can be used along with the quadratic relaxations discussed in \S\ref{sec:eigrel}--\ref{sec:eigrel_ns}. For simplicity, we only describe our branching strategies for the eigenvalue relaxation, which rely on the smallest eigenvalue of $Q$ and its associated eigenvector. The branching rules for the quadratic relaxations described in \S\ref{sec:geigrel} and~\ref{sec:eigrel_ns} are similar, but they make use of the smallest generalized eigenvalue of the pair $(Q, I + \delta A^T A)$ and its corresponding eigenvector.

We first introduce some notation. Let $\F$ be the set of indices of the variables that are fixed at the current node. Denote by $\B = \{1, \dots, n\} \setminus \F$ the set of branching candidates. Let $\bar{Q}$ be the $\R^{|\B| \times |\B|}$ sub-matrix of $Q$ obtained by eliminating the rows and columns corresponding to the variables in $\F$. Define the bijection $\sigma : \B \rightarrow \{1, \ldots, |\B|\}$, which maps $i \in \B$ to the $\sigma(i)$-th row and $\sigma(i)$-th column of $\bar{Q}$.

Assume that we branch on variable $x_i$, $i \in \B$ by creating two nodes, one where $x_i = 0$ and another where $x_i = 1$. At these descendant nodes, the eigenvalue relaxation is constructed by considering the smallest eigenvalue of the submatrix obtained by eliminating the $\sigma(i)$-th row and $\sigma(i)$-th column of $\bar{Q}$. We denote this submatrix by $\hat{Q}$. In this context, a potentially good branching rule may consist in branching on the variable which leads to the largest increase in the smallest eigenvalue of $\hat{Q}$. Note that, at a given node of the branch-and-bound tree, this rule requires the solution of $|\B|$ eigenvalue problems, each one involving a submatrix of $\bar{Q}$ obtained by eliminating the row and column corresponding to a particular index $i \in \B$. We call this rule \textit{spectral branching with complete enumeration}. The index corresponding to this branching rule, denoted as $i_{\text{exact}} \in C$, can be mathematically expressed as:
\begin{equation}
\label{i_exact_definition}
\begin{aligned}
i_{\text{exact}} = \arg \max_{i \in \B} \lambda_{\text{min}} \left( P_{\sigma(i)} \bar{Q} P_{\sigma(i)}^T \right)
\end{aligned}
\end{equation}
where $P_{\sigma(i)}$ is a $(|\B|-1) \times |\B|$ matrix obtained by removing the $\sigma(i)$-th row from the $|\B| \times |\B|$ identity matrix. Note that $\hat{Q} = P_{\sigma(i)} \bar{Q} P_{\sigma(i)}^T$ results in a matrix where the $\sigma(i)$-th row and $\sigma(i)$-th column of $\bar{Q}$ are removed.  

The computational complexity of complete enumeration is $\Omega(|\B|^3)$. We are not aware of any efficient approach for obtaining $i_{\text{exact}}$ that avoids complete enumeration. We instead rely on a lower bound for $\lambda_{\text{min}}(\cdot)$ that will be obtained without computing an eigenvalue and is computationally inexpensive. Gershgorin's Circle Theorem (GCT)~\cite{gvl96} provides such a lower bound estimate. The GCT states that: \emph{every eigenvalue of a $t \times t$ matrix $T$ lies in one of the circles ${\cal C}_k(T) = \{ \lambda \; : \; |\lambda - T_{kk}| \leq \sum_{l \neq  k} |T_{kl}| \}$ for $k = 1,\ldots, t$}. A lower bound estimate for the smallest eigenvalue of the matrix $T$ based on the GCT, denoted as $\underline{\lambda}^{\text{GCT}}_{\text{min}}(T)$ is:
\begin{equation}
\label{lambda_GCT}
\begin{aligned}
\underline{\lambda}^{\text{GCT}}_{\text{min}}(T) = \min\limits_{k \in \{1, \ldots, t\}} \left( T_{kk} - \sum\limits_{l \neq k} |T_{kl}|\right)
\end{aligned}
\end{equation}

Using the GCT-based lower bound estimate we can then define a branching variable index as:
\begin{equation}
\label{i_GCT_definition}
\begin{aligned}
i_{\text{GCT}} = \arg\max\limits_{i \in \B} \underline{\lambda}^{\text{GCT}}_{\text{min}} \left(P_{\sigma(i)} \bar{Q} P_{\sigma(i)}^T \right). 
\end{aligned}
\end{equation}
Note that the index $i_{\text{GCT}}$ can be determined without having to compute the matrix $P_{\sigma(i)} \bar{Q} P_{\sigma(i)}^T$. This approach has a computational complexity of $O(|\B|^2)$ and is computationally inexpensive compared to complete enumeration.  

The choice of $i_{GCT}$ can be viewed as a pessimistic estimate since it is obtained by maximizing the worst-case bound for the smallest eigenvalue.  Instead, we employ a different approach to determine the branching variable. Let $v$ be the eigenvector corresponding to the smallest eigenvalue of $\bar{Q}$. Then, we select as a branching variable, denoted by $i_{\text{approx}}$, the one which corresponds to the entry of $v$ with the largest absolute value, i.e. 
\begin{equation}
\label{i_approx_definition}
\begin{aligned}
i_{\text{approx}} = \arg\max_{i \in \B} |v_{\sigma(i)}|
\end{aligned}
\end{equation}
where $v_{\sigma(i)}$ denotes the $\sigma(i)$-th component of $v$. We call this rule \textit{approximate spectral branching}.  The computational complexity of this rule is $O(|\B|)$.

To appreciate the intuition behind this choice, we recall the proof for the GCT.  From the definition of the eigenvalue, we have 
\begin{subequations}
\label{i_approx_intuition}
\begin{align}
&	\bar{Q} v = \lambda_{\text{min}}(\bar{Q}) v \notag \\
\implies & \sum\limits_{j \in \B} \bar{Q}_{ \sigma(i_{\text{approx}}) \, \sigma(j)} v_{\sigma(j)} = \lambda_{\text{min}}(\bar{Q}) v_{\sigma(i_{\text{approx}})} \notag \\
\implies &	\lambda_{\text{min}}(\bar{Q}) - Q_{\sigma(i_{\text{approx}}) \, \sigma(i_{\text{approx}})} = \sum\limits_{\substack{j \in \B, \notag \\
 j \neq i_{\text{approx}}}} \bar{Q}_{\sigma(i_{\text{approx}}) \, \sigma(j)} \frac{v_{\sigma(j)}}{v_{\sigma(i_{\text{approx}})}} \notag \\ 
\implies &	|\lambda_{\text{min}}(\bar{Q}) - Q_{\sigma(i_{\text{approx}}) \, \sigma(i_{\text{approx}})}| 
 \leq \sum\limits_{\substack{j \in \B, \\ j \neq i_{\text{approx}}}} | \bar{Q}_{\sigma(i_{\text{approx}}) \, \sigma(j)}| \notag \\
\implies & \lambda_{\text{min}}(\bar{Q}) \in {\cal C}_{\sigma(i_{\text{approx}})}(\bar{Q}) \notag
\end{align}
\end{subequations}
where the first implication follows from the $\sigma(i_{\text{approx}})$-th row of the equality, the second implication is obtained by rearranging and dividing by $v_{\sigma(i_{\text{approx}})}$ and the inequality follows from $|v_{\sigma(j)}/ v_{\sigma(i_{\text{approx}})}| \leq 1$ by definition of $i_{\text{approx}}$. In essence, $i_{\text{approx}}$ identifies the particular Gershgorin circle that bounds the 
smallest eigenvalue $\lambda_{\text{min}}(\bar{Q})$.  Thus, the choice of $i_{\text{approx}}$ as the branching variable can be interpreted as eliminating the particular Gershgorin circle to which $\lambda_{\text{min}}(\bar{Q})$ belongs.  In that sense, this can be viewed as an optimistic estimate.

To illustrate the effectiveness of $i_{\text{GCT}}$ and $i_{\text{approx}}$ in mimicking $i_{\text{exact}}$, we performed some numerical experiments.  We generated matrices $Q$ of sizes $n \in \{50,100\}$ and densities $\rho \in \{0.25, 0.50, 1.00 \}$, and computed $i_{\text{exact}}$ by complete enumeration. Denote by $i_{\text{worst}}$ the index corresponding to the worst choice of branching variable, i.e.:
\begin{equation}
\label{i_worst_definition}
\begin{aligned}
i_{\text{worst}} = \arg \min_{i \in \B} \lambda_{\text{min}} \left(P_{\sigma(i)} Q P_{\sigma(i)}^T \right) 
\end{aligned}
\end{equation}
Then, the effectiveness of $i_{\text{x}}$ is measured using the metric:
\begin{equation}
\label{effectiveness_metric}
\begin{aligned}
	\% \text{ gap} = \frac{ \lambda_{\text{min}} \left(P_{\sigma(i_{\text{x}})} Q  P_{\sigma(i_{\text{x}})}^T \right) 
										- \lambda_{\text{min}}\left(P_{\sigma(i_{\text{exact}})} Q  P_{\sigma(i_{\text{exact}})}^T\right)  }{
									\lambda_{\text{min}}\left(P_{\sigma(i_{\text{worst}})} Q  P_{\sigma(i_{\text{worst}})}^T\right) 
										- \lambda_{\text{min}}\left(P_{\sigma(i_{\text{exact}})} Q  P_{\sigma(i_{\text{exact}})}^T\right) }
							\times 100
\end{aligned}
\end{equation}
where $\text{x} \in \{\text{approx},\text{GCT}\}$. A smaller value of $\%$ gap for $i_{\text{x}}$ represents a better approximation of $i_{\text{exact}}$. To obtain a statistic of the effectiveness of these approaches, we generated 100 different instances of $Q$ for each matrix size and density. Figure~\ref{cumulative_plots_branching} shows cumulative plots of the percentage of instances for which the $\%$ gap is below a certain value. It is evident from the plots that the approximate spectral branching strategy is a better choice than the GCT-based branching rule.

\begin{figure}[htp]
\centering
\subfloat[Instances of $Q$ with $n = 50$]{%
  \includegraphics[scale=0.15]{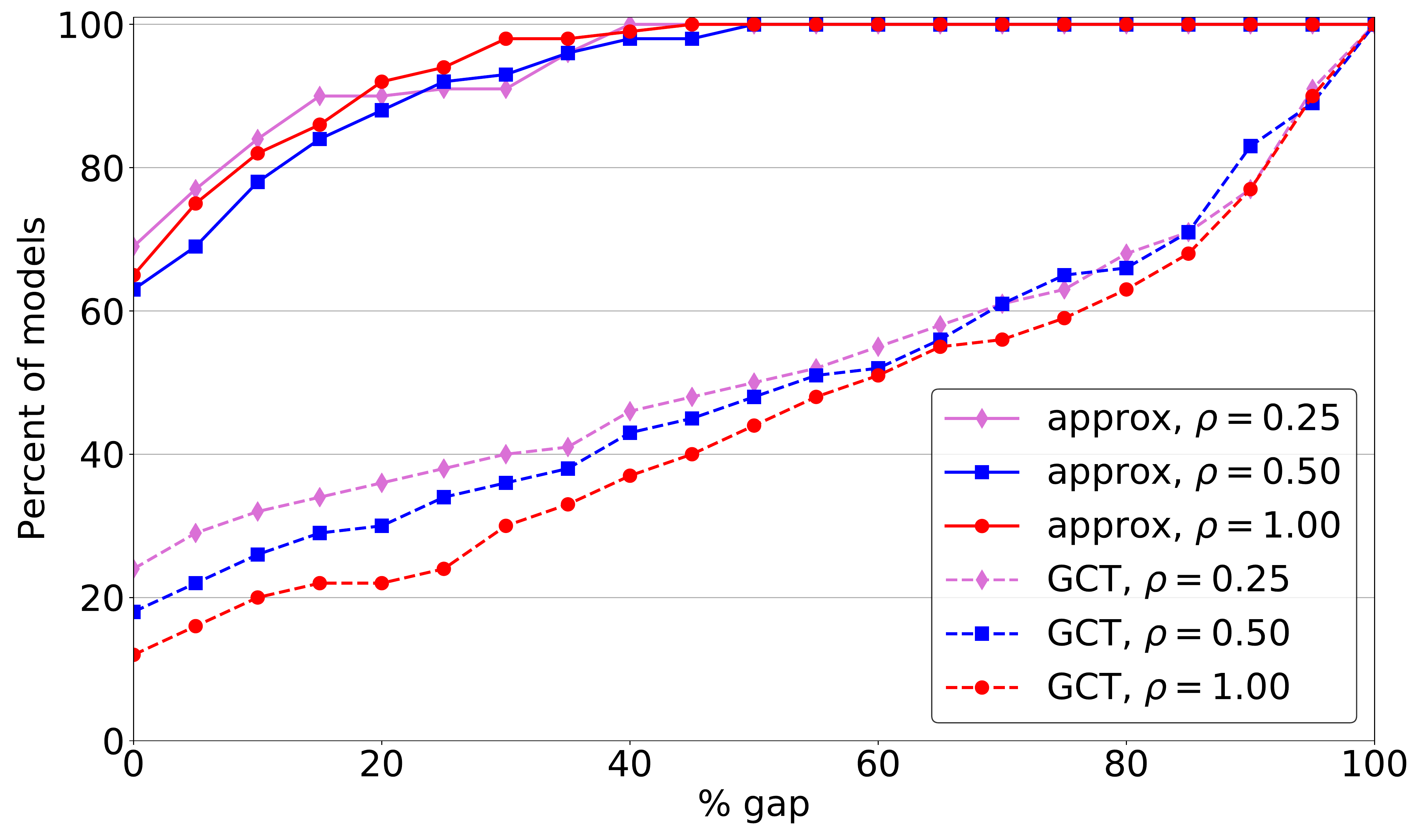}%
}
\subfloat[Instances of $Q$ with $n = 100$]{%
  \includegraphics[scale=0.15]{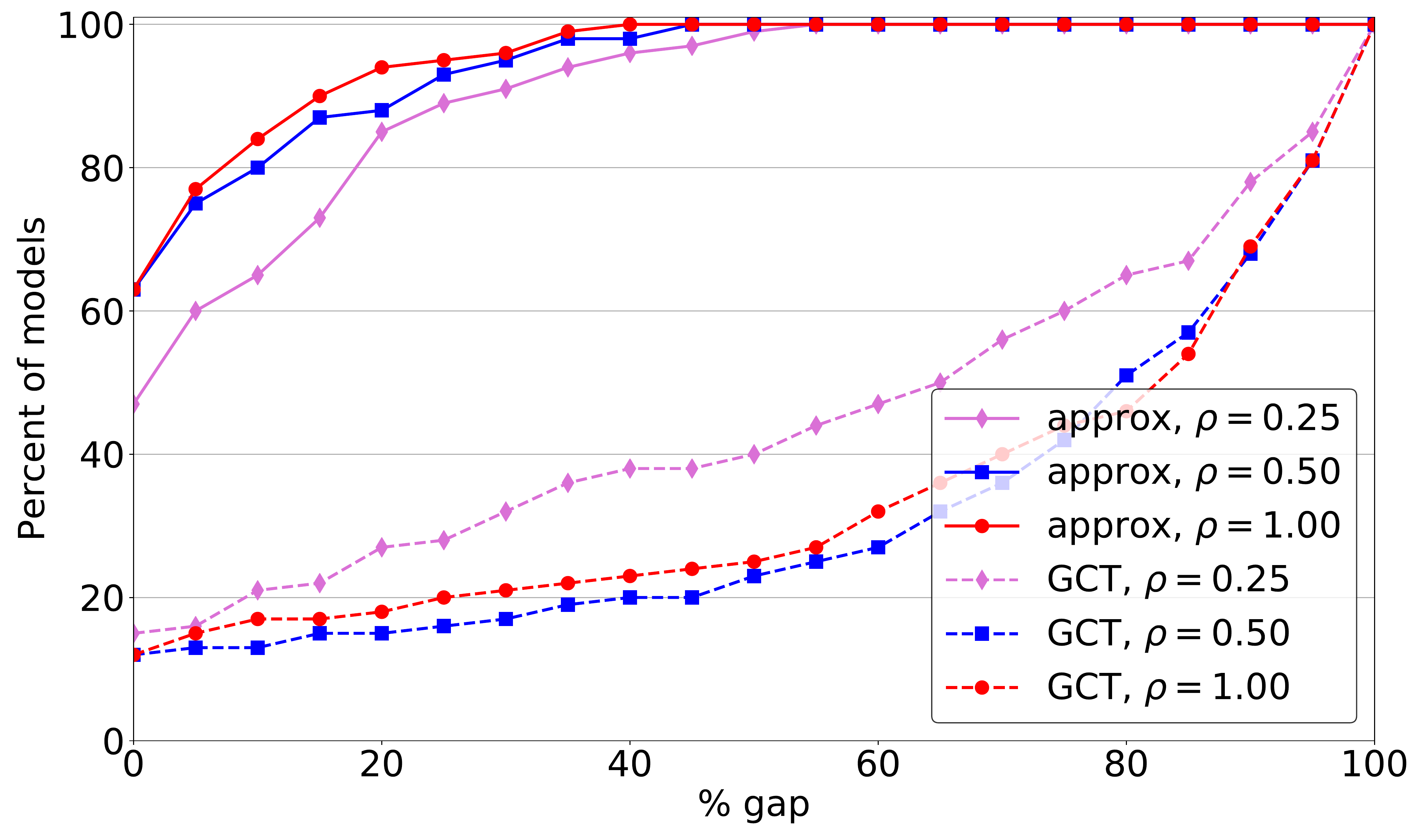}%
}
\caption{Cumulative plots comparing the effectiveness of the approximate spectral branching and the GCT-based branching strategies.}
\label{cumulative_plots_branching}
\end{figure}

\section{Implementation of the proposed relaxation and branching strategies in BARON}
\label{implementation}

By default, BARON's portfolio of relaxations consists of linear programming (LP), nonlinear programming (NLP) and mixed-integer linear programming (MILP) relaxations~\cite{ks:18,ks19,ts:comp:04}. In our implementation, we have expanded this portfolio by adding a new class of convex QP relaxations. These relaxations are constructed whenever the original model supplied to BARON is of the form~(\ref{problem_statement}). We take advantage of BARON's convexity detector (see~\cite{ks:18} for details) in order to determine the type of QP relaxation that will be constructed at a given node in the branch-and-bound tree. If the current node is convex, our QP relaxation is the continuous relaxation of~\eqref{problem_statement} subject to the variable bounds of the current node. On the other hand, if the current node is nonconvex, we construct one of the QP relaxations introduced in \S\ref{sec:eigrel}--\ref{sec:eigrel_ns}. The relaxation~\eqref{GEIGR_delta_2} is selected by default if the original problem contains equality constraints. Otherwise, our QP relaxation constructor automatically switches to the eigenvalue relaxation~\eqref{EIGR3}.

To solve the eigenvalue and generalized eigenvalue problems that arise during the construction of the relaxations discussed in \S\ref{sec:eigrel}--\ref{sec:eigrel_ns}, we use the subroutines included in the linear algebra library LAPACK~\cite{1999lapack}. When constructing these quadratic relaxations, we only consider the variables that have not been fixed at the current node. We use CPLEX as a subsolver for the new QP relaxations. The relaxation solution returned by the QP subsolver is used at the current node only if it satisfies the KKT conditions. This KKT test is similar to the optimality checks that BARON performs on the solutions returned by the LP and NLP subsolvers (see~\cite{ks:18} for details).

Another important component of our implementation is the approximate spectral branching rule described in \S\ref{spectral_branching}. This strategy is activated whenever the original problem supplied to BARON is a nonconvex binary QP. When this strategy is disabled, BARON uses reliability branching~\cite{akm:05} to select among binary branching variables.

\subsection*{Finding $\delta$}
When constructing the quadratic relaxation~\eqref{GEIGR_delta_2}, we use a sufficiently large value of $\delta$ to obtain a good approximation of the bound given by the eigenvalue relaxation in the nullspace of $A$. We use an iterative procedure to determine such value of $\delta$. We start by setting $\delta = 1$ and computing $\lambda_{\text{min}} (Q, I_n + \delta A^T A)$. Then, in each iteration, we increase $\delta$ by a factor of $\sigma$ and we use the resulting $\delta$ to compute a new value of $\lambda_{\text{min}} (Q, I_n + \delta A^T A)$. The procedure terminates when either the relative change in $\lambda_{\text{min}} (Q, I_n + \delta A^T A)$ is within a tolerance $relTol$ or the number of iterations reaches $maxIter$. In our experiments, we set $\sigma = 10$,  $maxIter = 5$, and $relTol = 10^{-3}$. This procedure is executed at the root node only, and the value of $\delta$ determined during its execution is used throughout the entire branch-and-bound tree.

\subsection*{Dynamic relaxation selection strategy}

We have implemented a dynamic relaxation selection strategy which is used for problems of the form~(\ref{problem_statement}) and switches between polyhedral and quadratic relaxations based on their relative strength. This dynamic strategy is motivated by two key observations. First, the strength of a given relaxation may depend on particular characteristics of the problem under consideration. Second, a particular type of relaxation may become stronger than other classes of relaxations as we move down the branch-and-bound tree.

We dynamically adjust the frequencies at which we solve the different types of relaxations during the branch-and-bound search. Denote by $\omega_{lp} \in [1, \bar{\omega}_{lp}]$ and $\omega_{qp} \in [1, \bar{\omega}_{qp}]$ the frequencies with which we solve the LP and QP relaxations, respectively. Let $f_{lp}$ and $f_{qp}$ be the optimal objective function values of the LP and QP relaxations, respectively. At the beginning of the global search, we set $\omega_{lp} = 1$ and $\omega_{qp} = 1$, which indicates that both the LP and QP relaxations will be solved at every node of the branch-and-bound tree. At nodes where both LP and QP relaxations are solved, we compare their corresponding objective function values. If $f_{qp} - f_{lp} \geq absTol$, we increase $\omega_{qp}$ and decrease $\omega_{lp}$ by setting $\omega_{qp} = \max \left(1, \omega_{qp} / \sigma_{qp} \right)$ and $\omega_{lp} = \min \left(\bar{\omega}_{lp}, \omega_{lp} \cdot \sigma_{lp} \right)$. Conversely, if $f_{qp} - f_{lp} < absTol$, we increase $\omega_{lp}$ by setting $\omega_{lp} = \max \left(1, \omega_{lp} / \sigma_{lp} \right)$, and decrease $\omega_{qp}$ by setting $\omega_{qp} = \min \left(\bar{\omega}_{qp}, \omega_{qp} \cdot \sigma_{qp} \right)$. In our experiments, we set $\sigma_{lp} = 10$, $\sigma_{qp} = 2$, $\bar{\omega}_{lp} = 1000$, $\bar{\omega}_{qp} = 10$, and $absTol = 10^{-3}$.

Although BARON also makes use of MILP relaxations, in our dynamic relaxation selection strategy, we only compare the bounds given by the LP and QP relaxations. As MILP relaxations can be expensive, BARON uses a heuristic to decide if an MILP relaxation will be solved at the current node~\cite{ks19}. In our implementation, this heuristic is invoked only if at the current node the QP relaxation is weaker than the LP relaxation. If the converse is true, the MILP relaxation is skipped altogether.

\section{Computational results}
\label{computational_results}

In this section, we present the results of a computational study conducted to investigate the impact of the techniques proposed in this paper on the performance of branch-and-bound algorithms. We start in \S\ref{sec:results_relaxations} with a numerical comparison between the spectral relaxations introduced in \S\ref{sec:eigrel}--\ref{sec:eigrel_ns} and some relaxations reviewed in \S\ref{polyhedral_sdp_quadratic_relaxations}. In \S\ref{sec:results_baron}, we analyze the impact of the implementation described in \S\ref{implementation} on the performance of the global optimization solver BARON. Then, in \S\ref{sec:results_all_solvers}, we compare several state-of-the-art global optimization solvers. Finally, in \S\ref{sec:results_comparison_with_qcr}, we compare BARON and the QCR approach discussed in \S\ref{sec:qcr_based_relaxations}.

Throughout this section, all experiments are conducted under GAMS 30.1.0 on a 64-bit Intel Xeon X5650 2.66GHz processor with a single-thread. We solve all problems in minimization form. For the experiments described in \S\ref{sec:results_relaxations}, the linear and convex quadratic programs are solved using CPLEX 12.10, whereas the SDPs are solved using MOSEK 9.1.9. For the experiments considered in \S\ref{sec:results_baron}--\ref{sec:results_comparison_with_qcr}, we consider the following global optimization solvers: ANTIGONE 1.1, BARON 19.12, COUENNE 0.5, CPLEX 12.10, GUROBI 9.0, LINDOGLOBAL 12.0 and SCIP 6.0. When dealing with nonconvex problems, we: (i) run all solvers with relative and absolute optimality tolerances of 10\textsuperscript{-6} and a time limit of 500 seconds, and (ii) set the CPLEX option {\tt optimalitytarget} to 3 and the GUROBI option {\tt nonconvex} to 2 to ensure that these two solvers search for a globally optimal solution. For other algorithmic parameters, we use default settings. The computational times reported in our experiments do not include the time required by GAMS to generate problems and interface with solvers; only times taken by the solvers are reported.

For our experiments, we use a large test set consisting of 960 Cardinality Binary Quadratic Programs (CBQPs), 30 Quadratic Semi-Assignment Problems (QSAPs), 246 Box-Constrained Quadratic Programs (BoxQPs), and 315 Equality Integer Quadratic Programs (EIQPs). These test libraries are described in detail in~\cite{qp_miqp_tests}.

\subsection{Comparison between relaxations}
\label{sec:results_relaxations}

In this section, we provide a comparison between the spectral relaxations introduced in \S\ref{sec:eigrel}--\ref{sec:eigrel_ns}, the convex quadratic relaxation~\eqref{EIGDC2}, the first-level RLT relaxation~\eqref{RLT_1}, and the SDP relaxations~\eqref{SDP_d} and~\eqref{SDP_d_a}. We construct performance profiles based on the root-node relaxation gap:
\begin{equation}
\label{optimality_gap}
\text{GAP} = \left(\dfrac{f_{UBD} - f_{LBD}}{\max(|f_{LBD}|, 10^{-3})}\right) \times 100
\end{equation}
where $f_{LBD}$ is the root-node relaxation lower bound, and $f_{UBD}$ is the best upper bound available for a given instance. The following notation is used to refer to the different relaxations:

\begin{itemize}

\item EIG: Eigenvalue relaxation~\eqref{EIGR3}.

\item GEIG: Generalized eigenvalue relaxation~\eqref{GEIGR3}.

\item EIGNS: Eigenvalue relaxation in the nullspace of $A$~\eqref{EIGRNS1}.

\item EIGDC: Quadratic relaxation~\eqref{EIGDC2} based on the eigdecomposition of $Q$.

\item RLT: First-level RLT relaxation~\eqref{RLT_1}.

\item SDPd: SDP relaxation~\eqref{SDP_d}.

\item SDPda: SDP relaxation~\eqref{SDP_d_a}.

\end{itemize}

Performance profiles are presented in Figures~\ref{fig:root_node_gaps_CBQP}--\ref{fig:root_node_gaps_EIQP}. These profiles show the percentage of models for which the gap defined in~\eqref{optimality_gap} is below a certain threshold. As seen in the figures, the SDP relaxations give the tighter bounds, followed by the spectral relaxations. For these instances, both the RLT relaxation~\eqref{RLT_1} and the quadratic relaxation~\eqref{EIGDC2} provide relatively weak bounds.

\begin{figure}[htp]
\centering
\subfloat[960 CBQP instances.\label{fig:root_node_gaps_CBQP}]{%
  \includegraphics[scale=0.15]{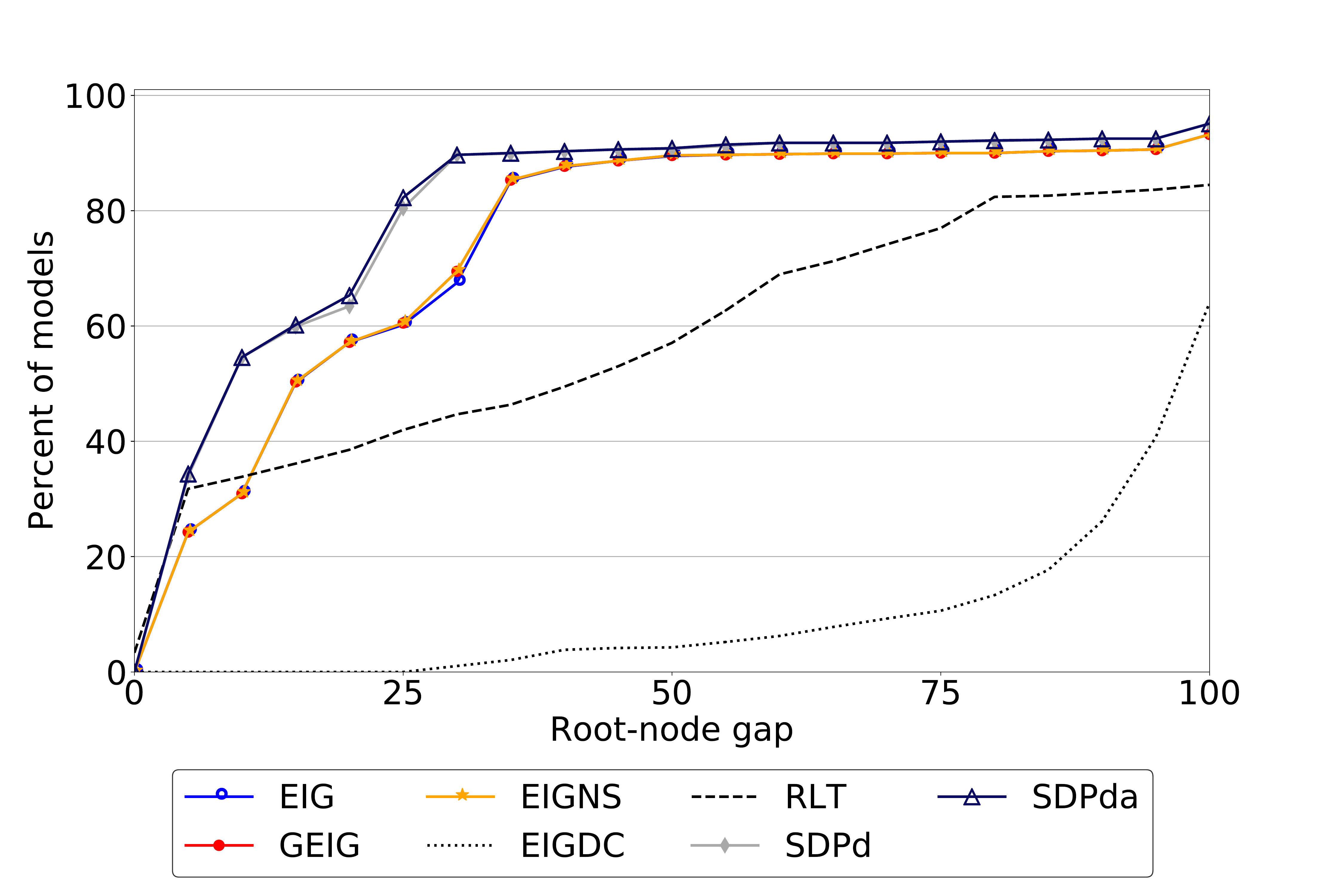}%
}
\subfloat[30 QSAP instances.\label{fig:root_node_gaps_QSAP}]{%
  \includegraphics[scale=0.15]{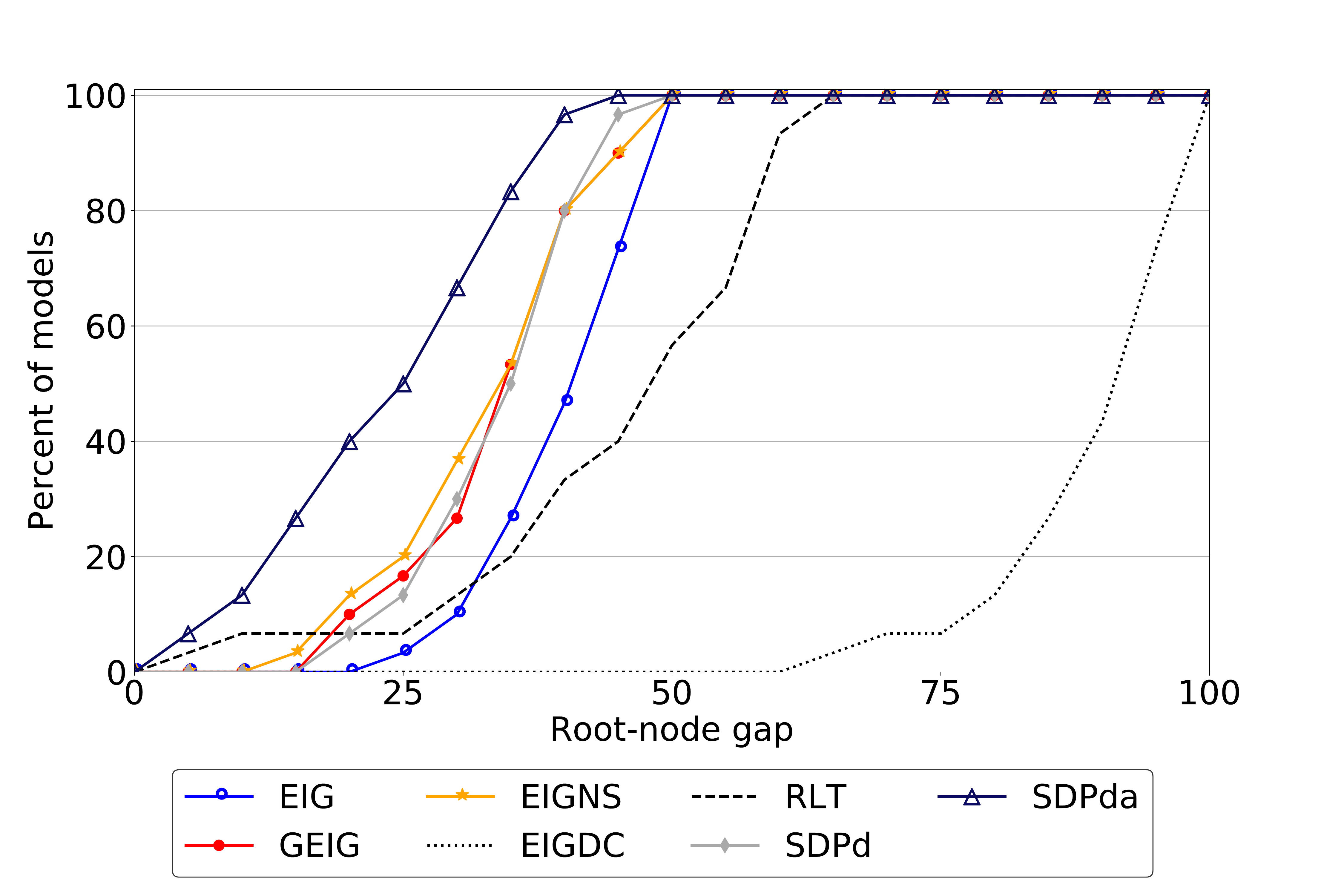}%
} \\
\subfloat[246 BoxQP instances.\label{fig:root_node_gaps_BoxQP}]{%
  \includegraphics[scale=0.15]{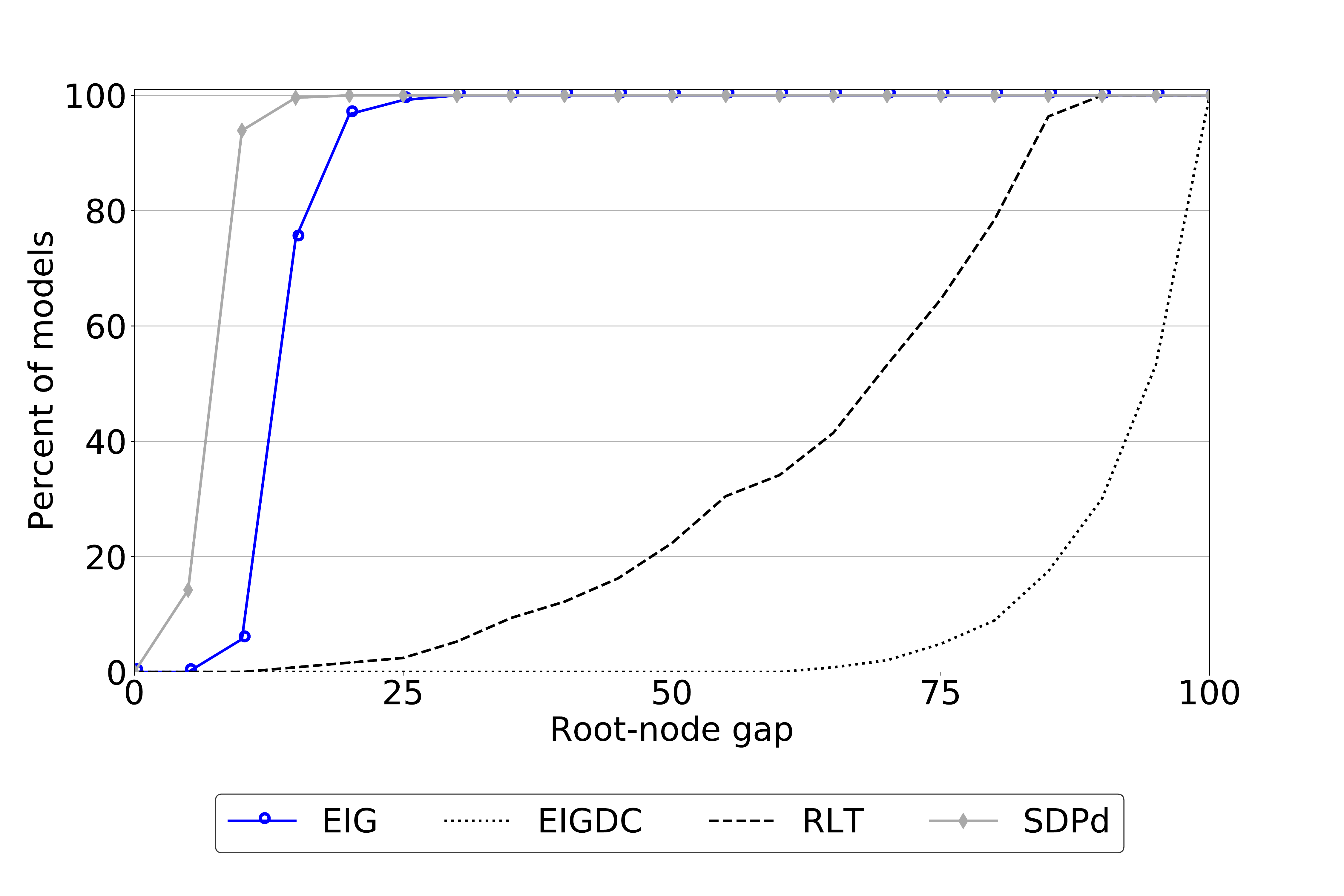}%
}
\subfloat[315 EIQP instances.\label{fig:root_node_gaps_EIQP}]{%
  \includegraphics[scale=0.15]{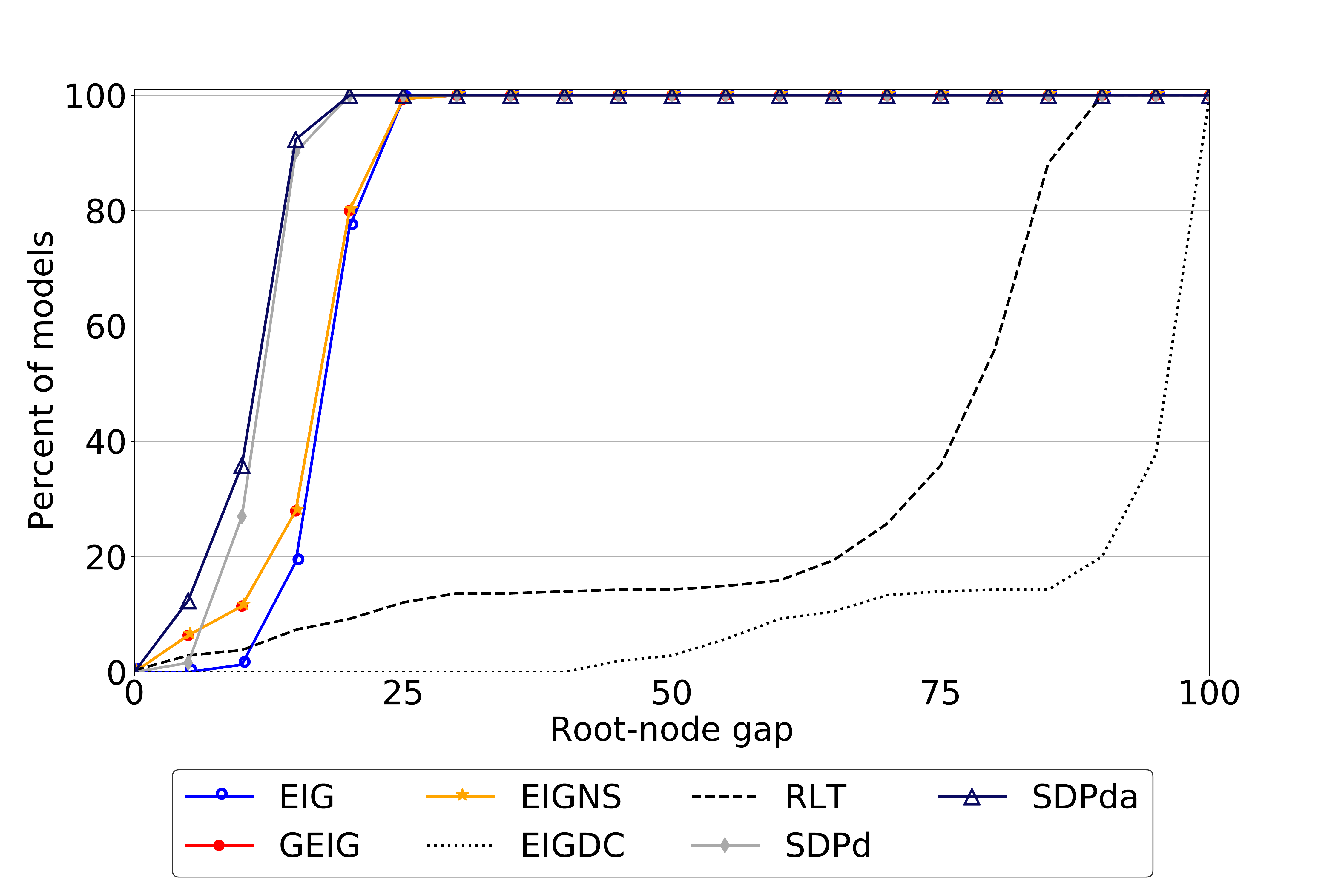}%
}
\caption{Comparison between the root-node relaxations gaps.}
\label{fig:root_node_gaps}
\end{figure}

We also compare these root-node relaxations in terms of their solution times. To that end, in Figures~\ref{fig:root_node_times_CBQP}--\ref{fig:root_node_times_EIQP}, we present the geometric means of the CPU times required to solve the different classes of relaxations. For the quadratic relaxation based on the eigdecomposition of $Q$, the CPU time includes the time required to solve the convex QP~\eqref{EIGDC2} and the time taken to solve the LPs used to determine the bounds on the $y_i$ variables. We group the instances based on their size. As the figures indicate, the spectral relaxations are relatively inexpensive regardless of the characteristics of the problem. As the size of the problem increases, the RLT relaxations become more expensive to solve, and in some cases, these RLT relaxations are orders of magnitude more expensive than the other relaxations. Note that the separable programming procedure described in \S\ref{sec:separable_programming_relaxation} does not only lead to relatively weak bounds, but it is also computationally expensive since it requires the solution of $2n$ linear programs. Even though for most of the problems considered in the experiments the SDP relaxations can be solved within 10 seconds, they are between one and two orders of magnitude more expensive than the spectral relaxations.

\begin{figure}[htp]
\centering
\subfloat[960 CBQP instances.\label{fig:root_node_times_CBQP}]{%
  \includegraphics[scale=0.18]{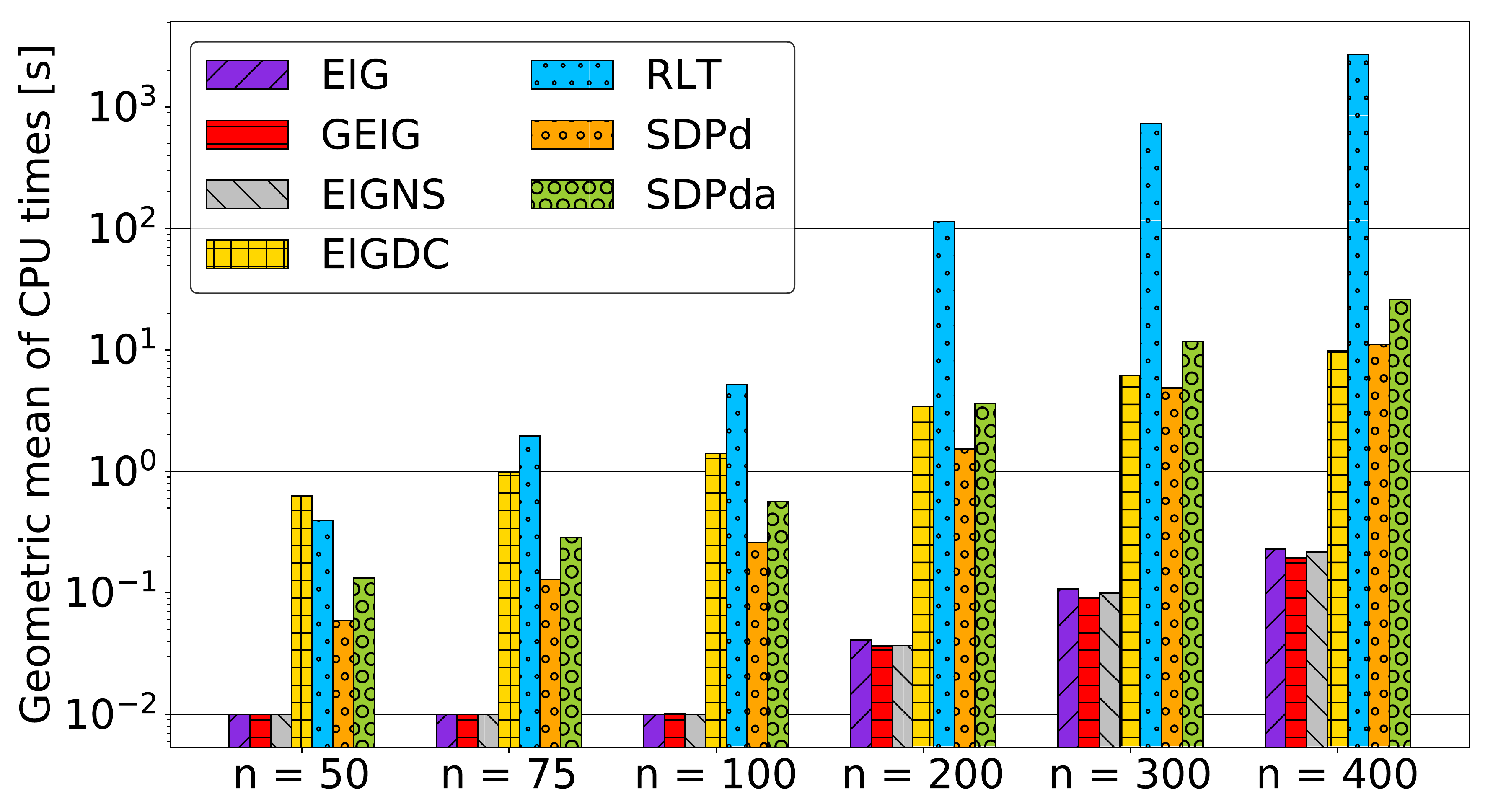}%
}
\subfloat[30 QSAP instances.\label{fig:root_node_times_QSAP}]{%
  \includegraphics[scale=0.18]{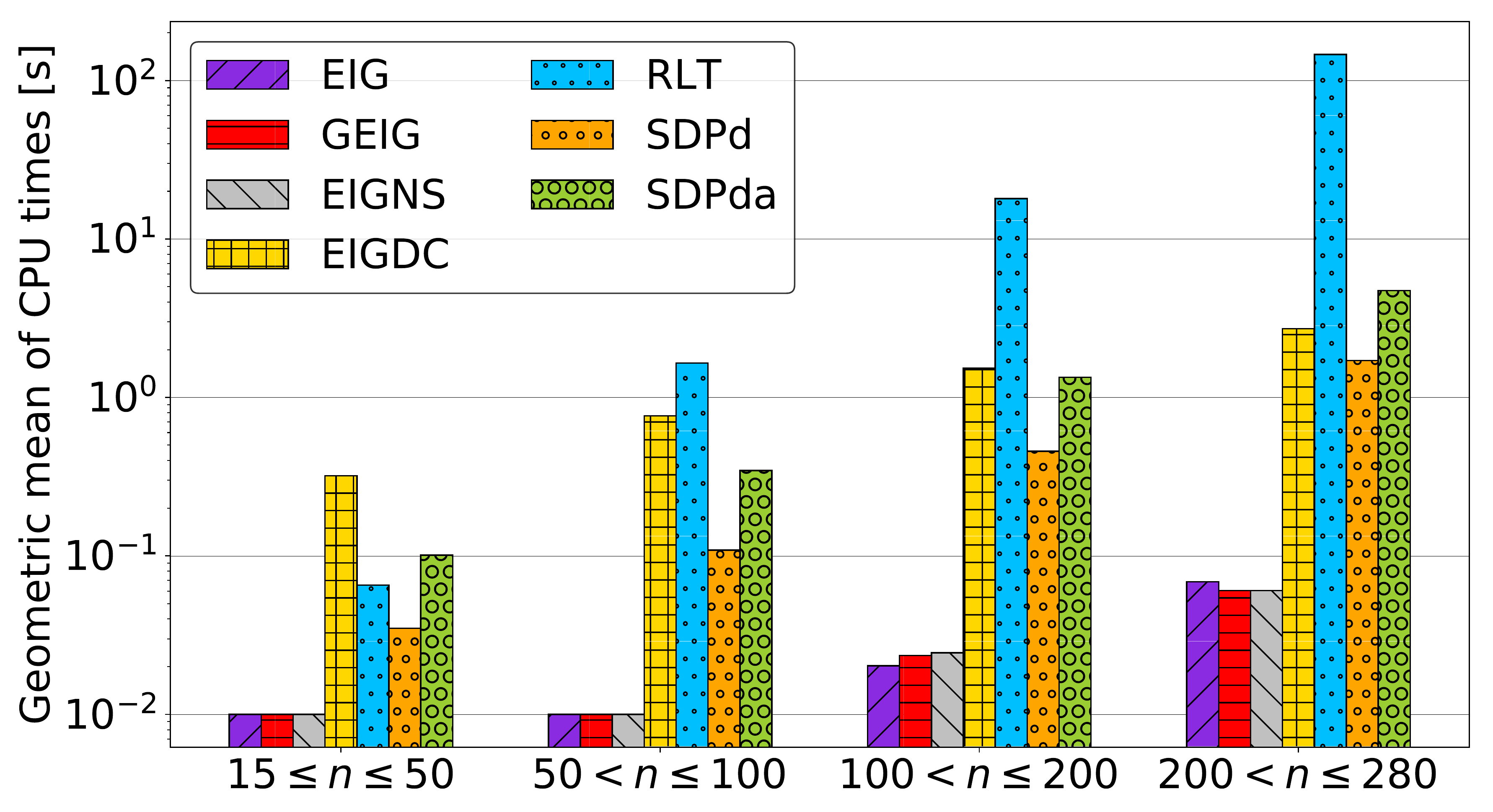}%
} \\
\subfloat[246 BoxQP instances.\label{fig:root_node_times_BoxQP}]{%
  \includegraphics[scale=0.18]{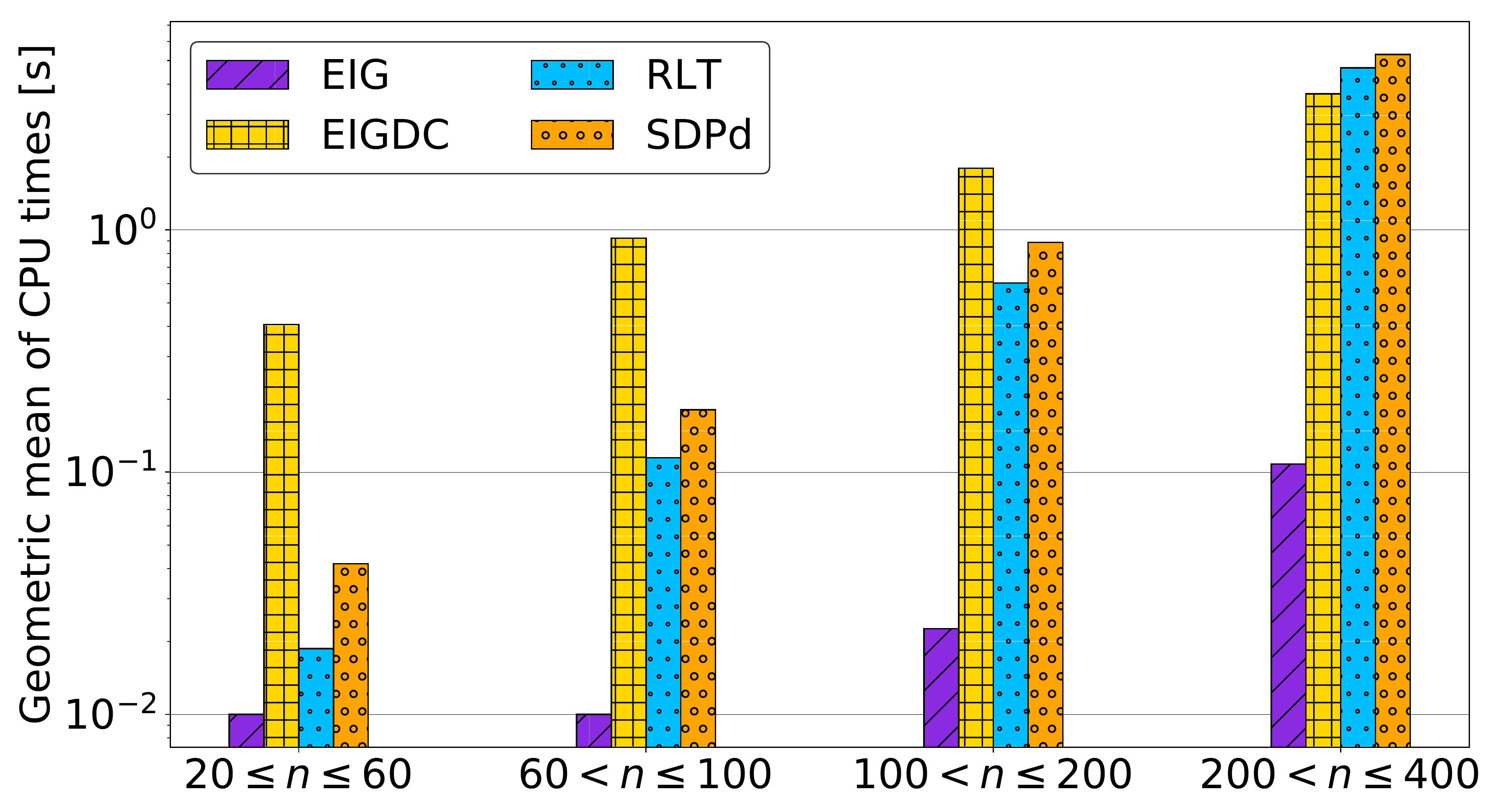}%
}
\subfloat[315 EIQP instances.\label{fig:root_node_times_EIQP}]{%
  \includegraphics[scale=0.18]{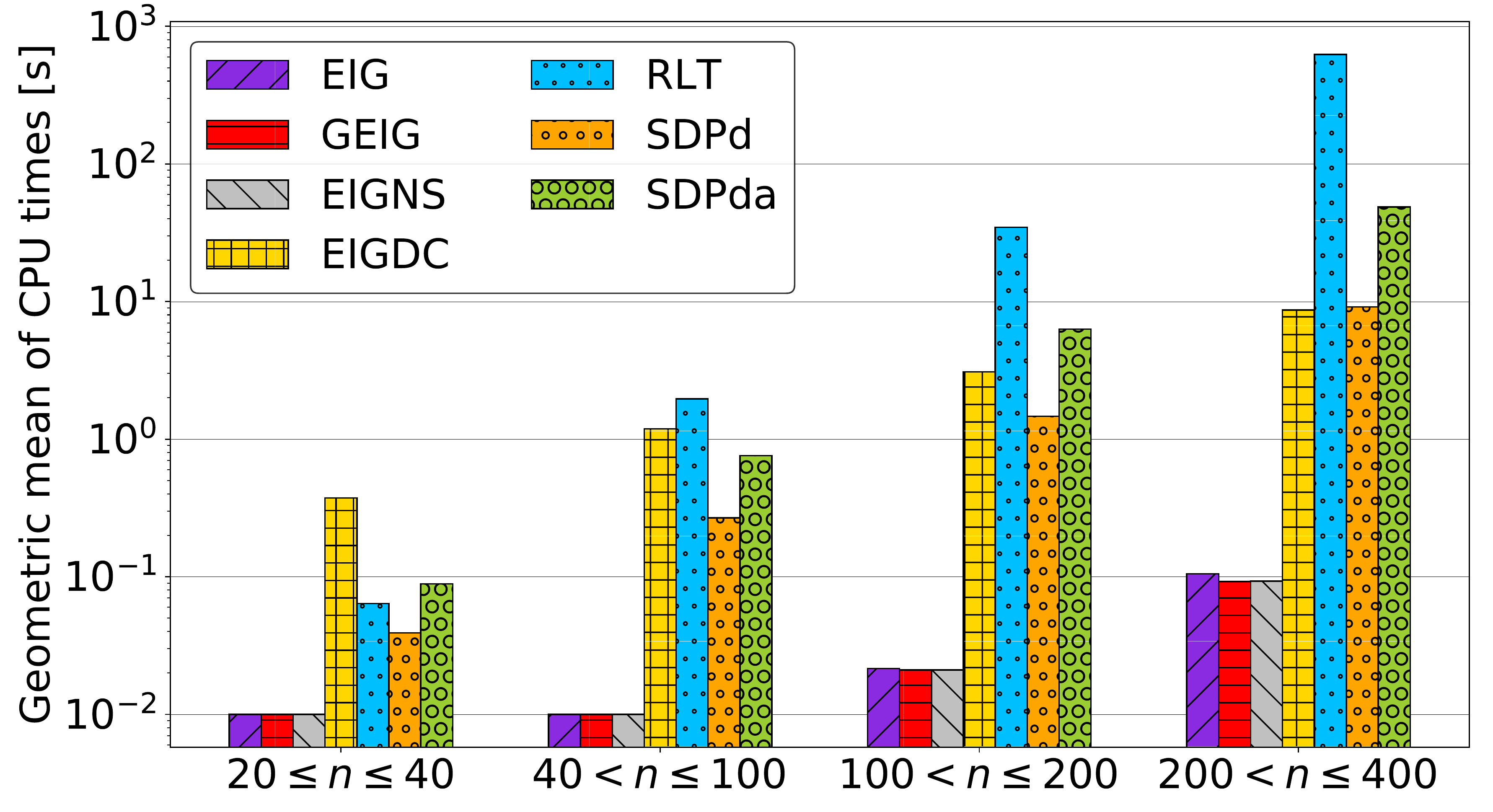}%
}
\caption{Geometric means of the CPU times required to solve the root-node relaxations.}
\label{fig:root_node_times}
\end{figure}

\subsection{Impact of the implementation on BARON's performance}
\label{sec:results_baron}

In this section, we demonstrate the benefits the proposed relaxation and branching techniques on the performance of the global optimization solver BARON. We consider the following versions of BARON 19.12:

\begin{itemize}

\item BARONnoqp: BARON without the spectral relaxations and without the spectral branching rule.

\item BARONnosb: BARON with the spectral relaxations but without the spectral branching rule.

\item BARON: BARON with the spectral relaxations and the approximate spectral branching rule. This is the default version of this solver.

\end{itemize}

As mentioned previously, the spectral branching rule introduced in \S\ref{spectral_branching} is only used for the binary instances. In order to analyze the impact of our implementation, we start by comparing the different versions of BARON through performance profiles. For instances which can be solved to global optimality within the time limit of 500 seconds, we use performance profiles based on CPU times. In this case, for a given solver, we plot the percentage of models that can be solved within a certain amount of time. For problems for which global optimality cannot be proven within the time limit, we employ performance profiles based on the optimality gaps at termination. These gaps are determined according to~\eqref{optimality_gap} by using the best lower and upper bounds reported by the solver under consideration. In this case, for a given solver, we plot the percentage of models for which the remaining gap is below a given threshold.

The performance profiles are presented in Figures~\ref{fig:baron_profiles_CBQP}--\ref{fig:baron_profiles_EIQP}. As seen in the figures, our implementation leads to very significant improvements in the performance of BARON. Clearly, for the CBQP and QSAP instances, both the spectral relaxations and the spectral branching strategy result in a version of BARON which is able to solve many more problems to global optimality. In addition, for the four collections considered in this comparison, BARON terminates with much smaller relaxation gaps in cases in which global optimality cannot be proven within the time limit.

\begin{figure}[htp]
\centering
\subfloat[960 CBQP instances.\label{fig:baron_profiles_CBQP}]{%
  \includegraphics[scale=0.15]{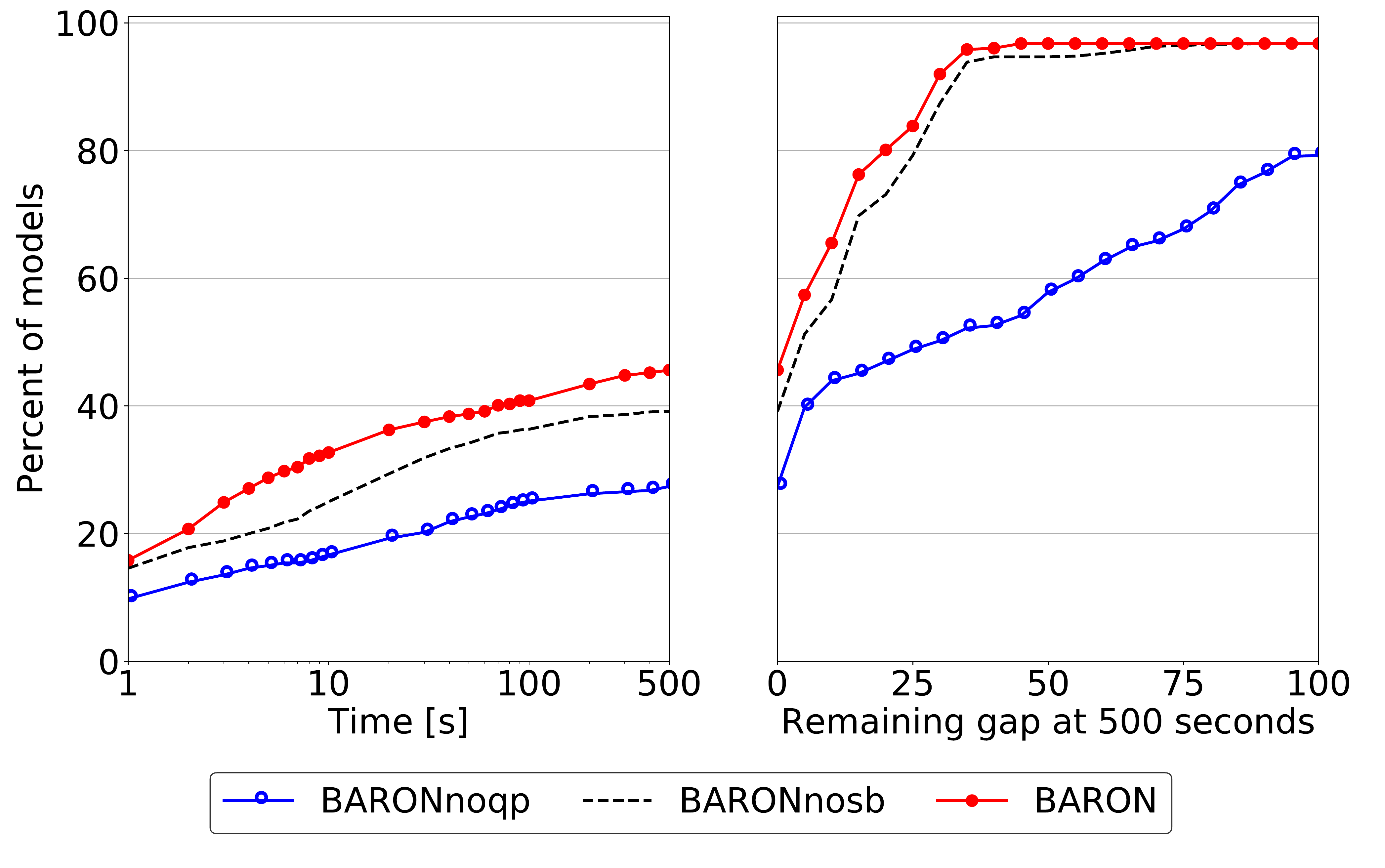}%
}
\subfloat[30 QSAP instances.\label{fig:baron_profiles_QSAP}]{%
  \includegraphics[scale=0.15]{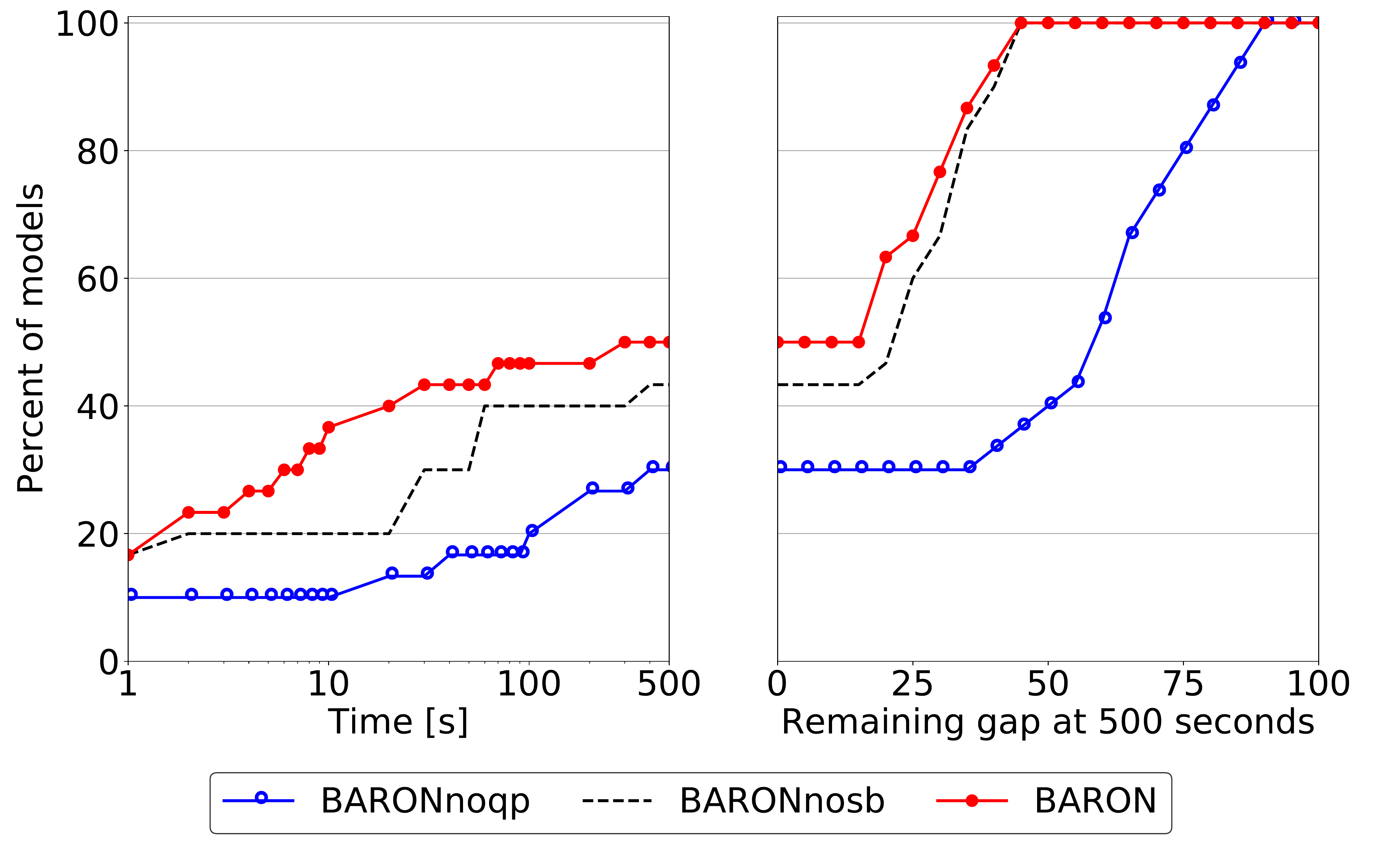}%
} \\
\subfloat[246 BoxQP instances.\label{fig:baron_profiles_BoxQP}]{%
  \includegraphics[scale=0.15]{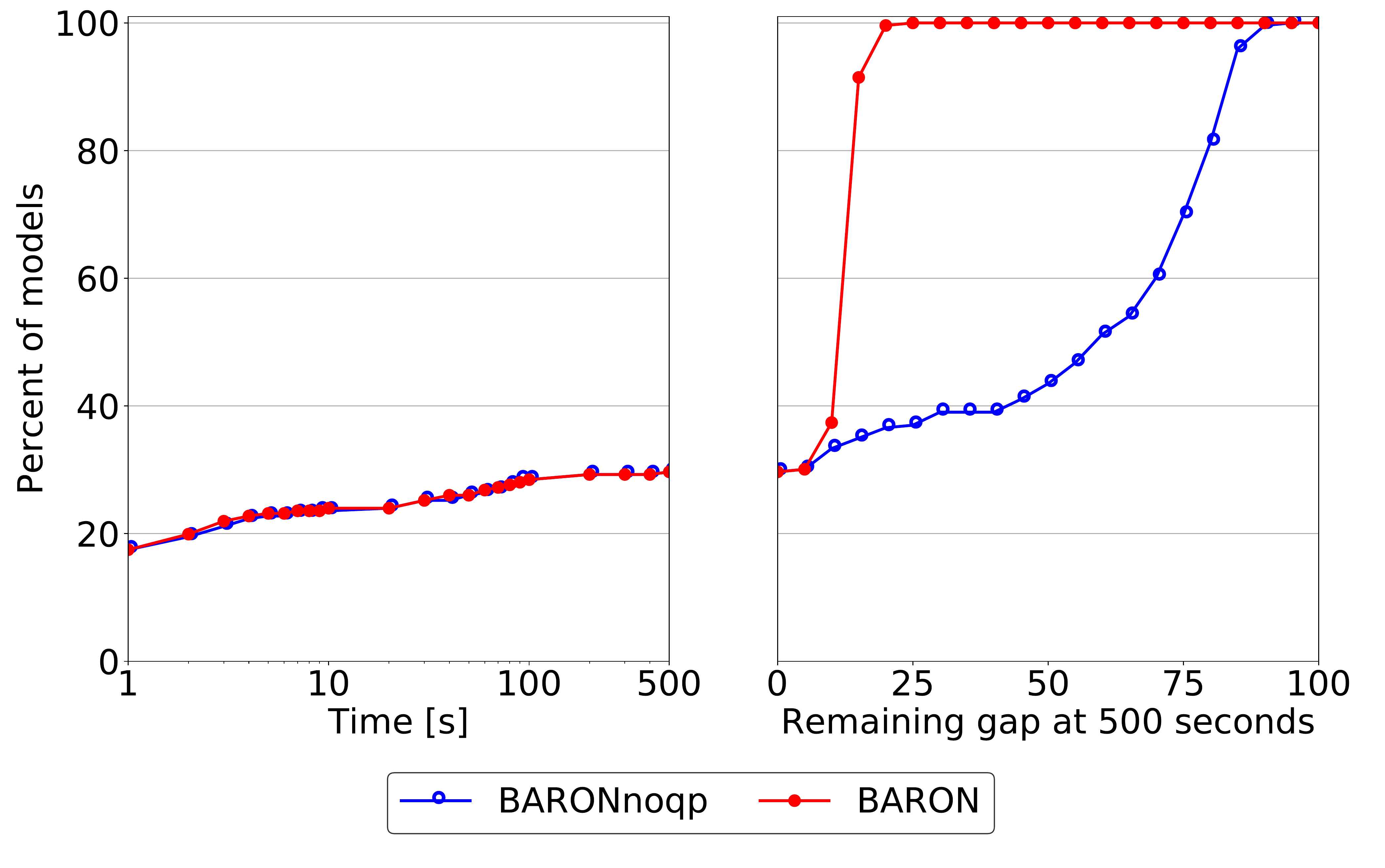}%
}
\subfloat[315 EIQP instances.\label{fig:baron_profiles_EIQP}]{%
  \includegraphics[scale=0.15]{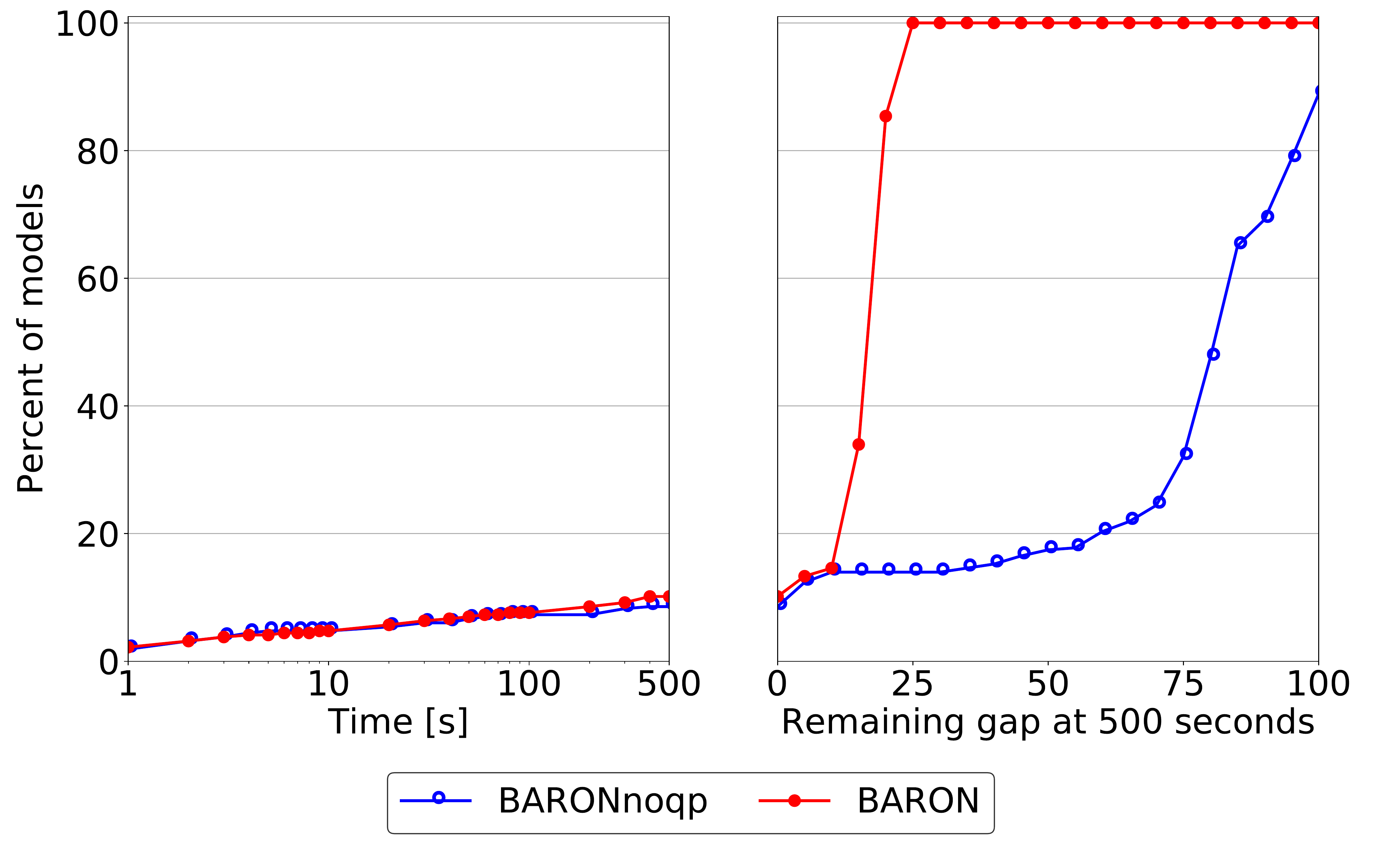}%
}
\caption{Comparison between the different versions of BARON.}
\label{fig:baron_profiles}
\end{figure}

Next, we provide a more detailed comparison between BARON and BARONnoqp. To this end, we eliminate from the test set all the problems that can be solved trivially by both solvers (146 instances). A problem is regarded as trivial if it can be solved by both solvers in less than one second. After eliminating all of these problems from the original test set, we obtain a new test set consisting of 1405 instances. 

We first consider the nontrivial problems that are solved to global optimality by at least one of the two the versions of the solver (412 instances). For this analysis, we compare the performance of the two solvers by considering their CPU times. In this comparison, we say that the two solvers perform similarly if their CPU times are within 10\% of each other. The results are presented in Figure~\ref{fig:baron_vs_baron_noqp_speedups}. As the figure indicates, BARON is significantly faster than BARONnoqp. For nearly 50\% of the problems considered in this comparison, BARON is at least one of magnitude faster than BARONnoqp

Now, we consider the nontrivial problems that neither of the two solvers are able to solve to global optimality within the time limit (993 instances). In this case, we analyze the performance of these solvers by comparing the gaps reported at termination. For the purposes of this comparison, we say that two solvers obtain similar gaps if their remaining gaps are within 10\% of each other. The results are presented in Figure~\ref{fig:baron_vs_baron_noqp_gaps}. As seen in the figure, for more than 90\% considered in this comparison, BARON reports significantly termination gaps than BARONnoqp.

\begin{figure}[htp]
\centering
\subfloat[CPU times (412 nontrivial instances). \label{fig:baron_vs_baron_noqp_speedups}]{%
  \includegraphics[scale=0.19]{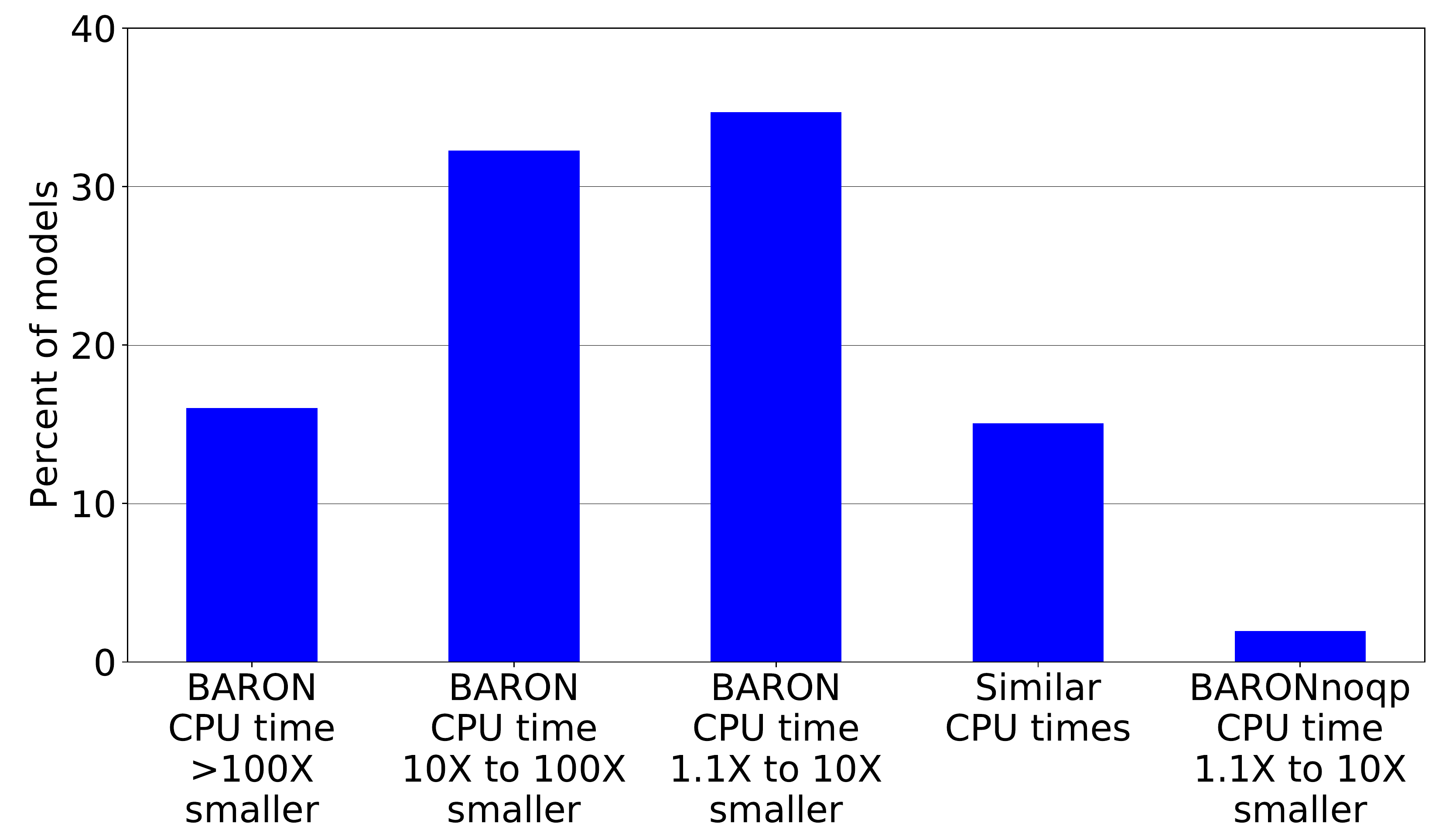}%
}
\subfloat[Relative gaps (993 nontrivial instances). \label{fig:baron_vs_baron_noqp_gaps}]{%
  \includegraphics[scale=0.19]{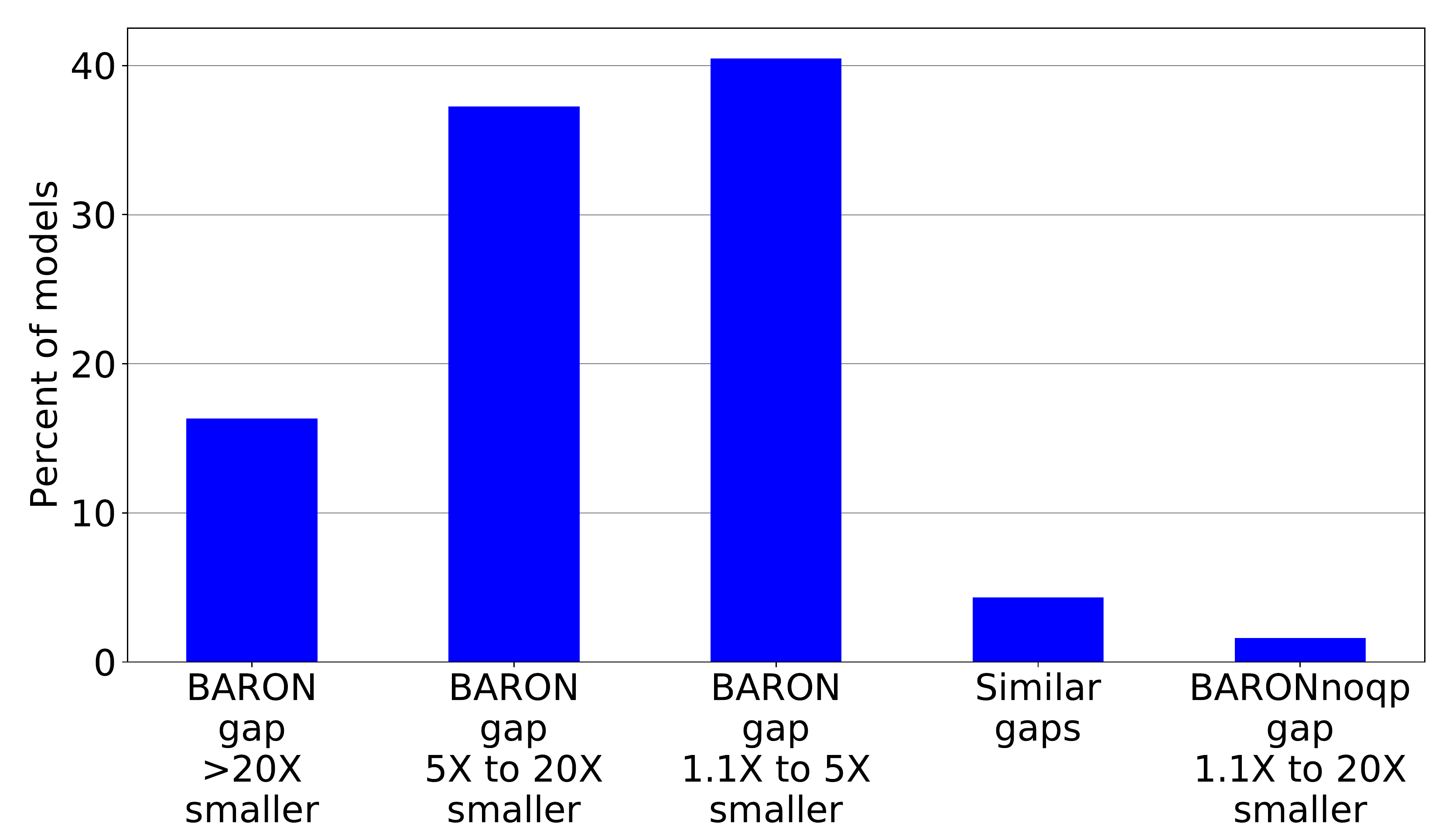}%
}
\caption{One-to-one comparison between BARON and BARONnoqp.}
\label{fig:baron_vs_baron_no_qp}
\end{figure}

\subsection{Comparison between global optimization solvers}
\label{sec:results_all_solvers}

In this section, we compare BARON with other global optimization solvers. We start by providing a comparison between different solvers using the same type of performance profiles considered in the previous section. These profiles are shown in Figures~\ref{fig:global_solvers_profiles_CBQP}--\ref{fig:global_solvers_profiles_EIQP}. As seen in these figures, BARON performs well in comparison to other solvers. For both the CBQP and QSAP instances, BARON is faster than the other solvers and solves many more problems to global optimality. For the QSAP and BoxQP instances, BARON terminates with smaller gaps than the other solvers in cases in which global optimality cannot be proven within the time limit. Many of the BoxQP and EIQP instances are very challenging and cannot be globally solved within the time limit by solvers considered in this analysis.

\begin{figure}[htp]
\centering
\subfloat[960 CBQP instances. \label{fig:global_solvers_profiles_CBQP}]{%
  \includegraphics[scale=0.15]{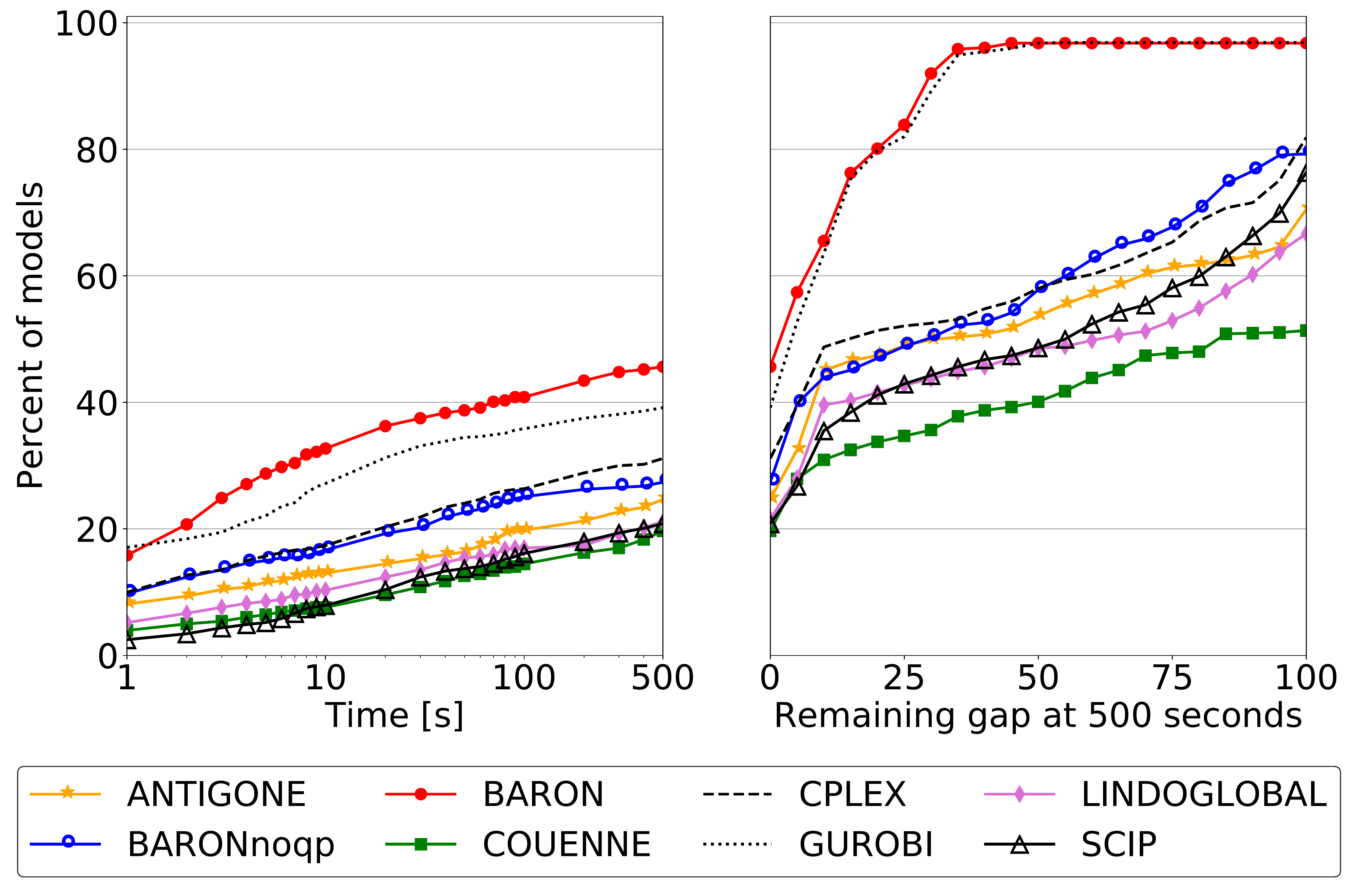}%
}
\subfloat[30 QSAP instances. \label{fig:global_solvers_profiles_QSAP}]{%
  \includegraphics[scale=0.15]{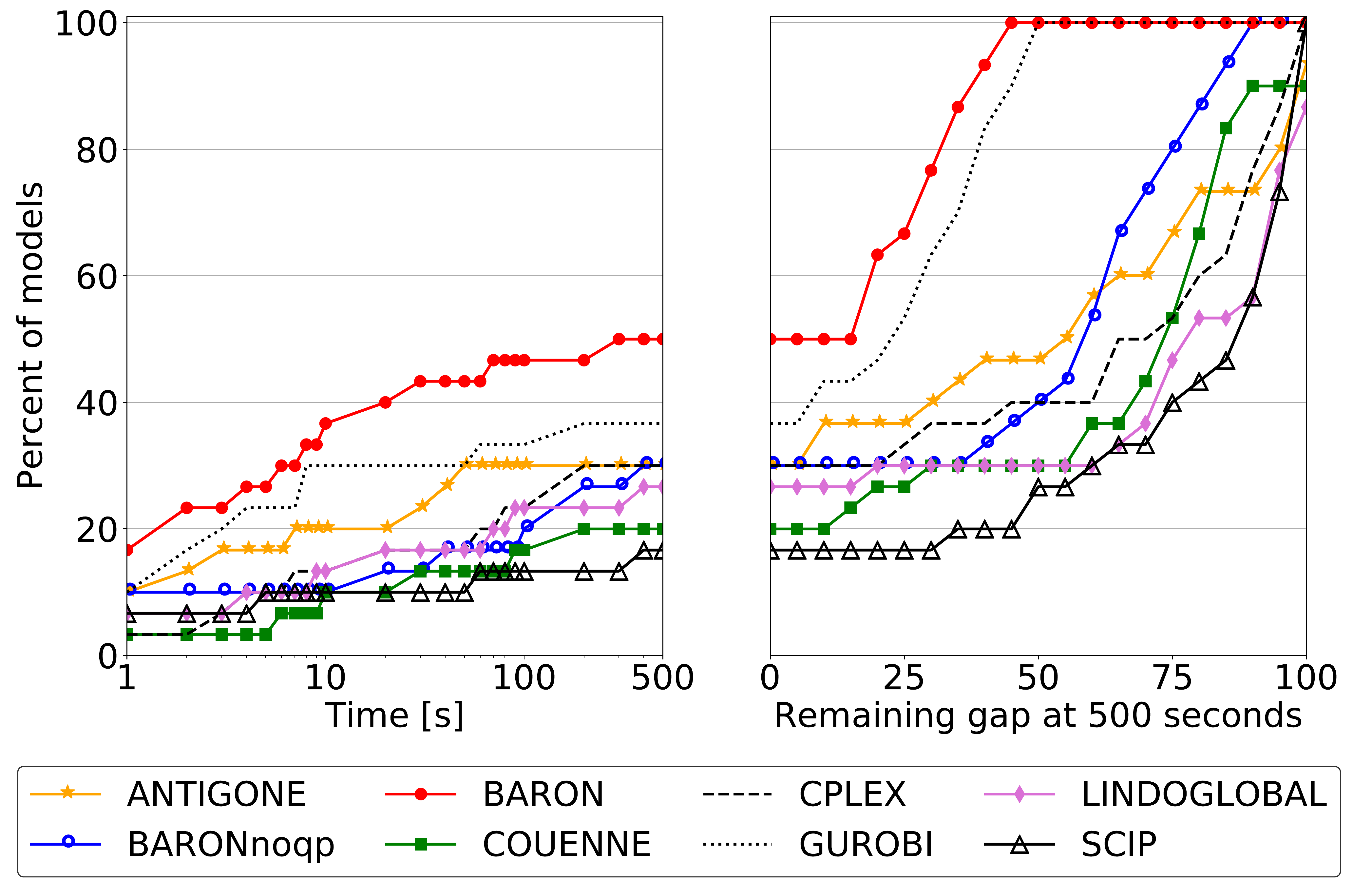}%
} \\
\subfloat[246 BoxQP instances. \label{fig:global_solvers_profiles_BoxQP}]{%
  \includegraphics[scale=0.15]{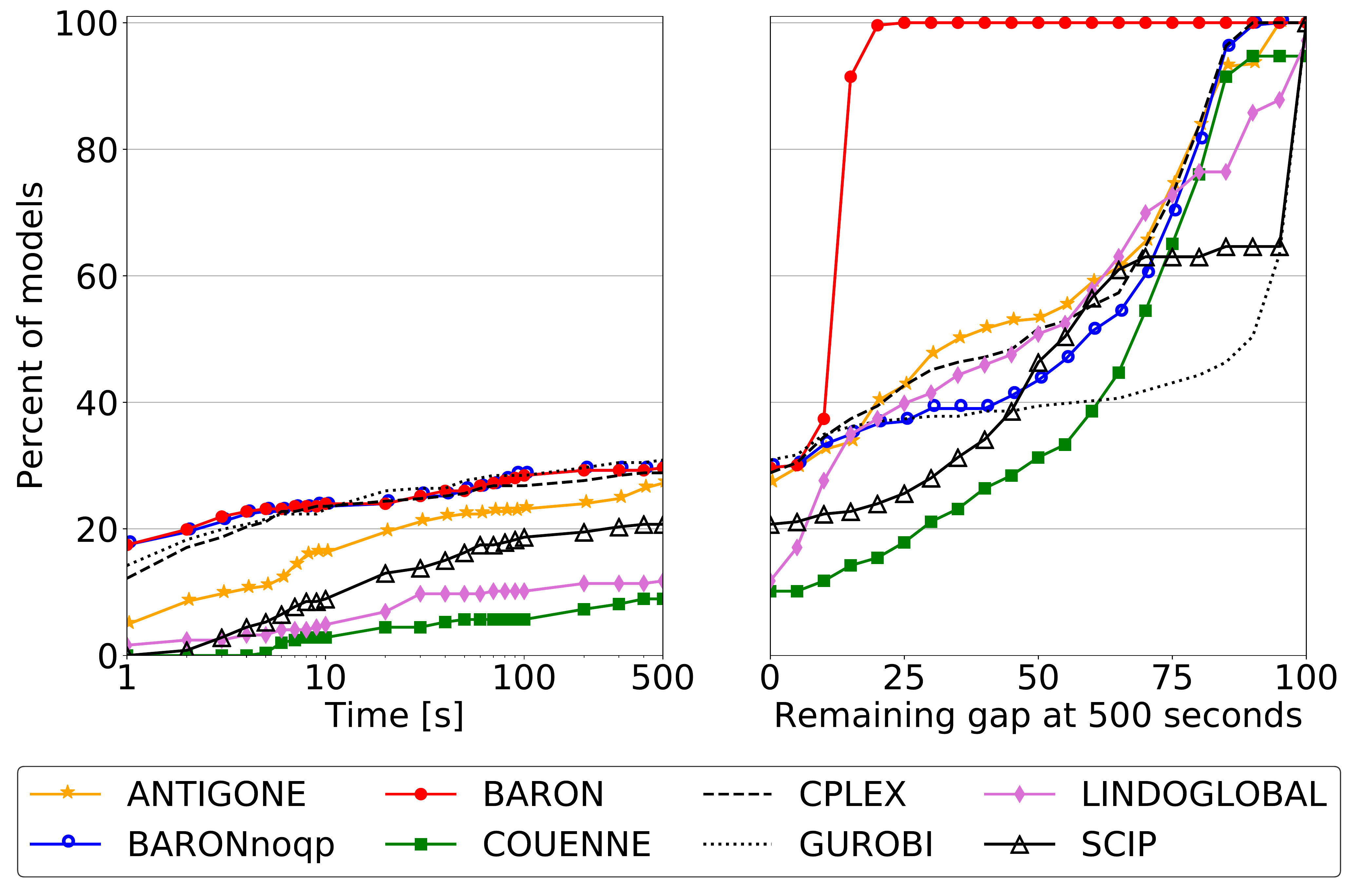}%
}
\subfloat[315 EIQP instances. \label{fig:global_solvers_profiles_EIQP}]{%
  \includegraphics[scale=0.15]{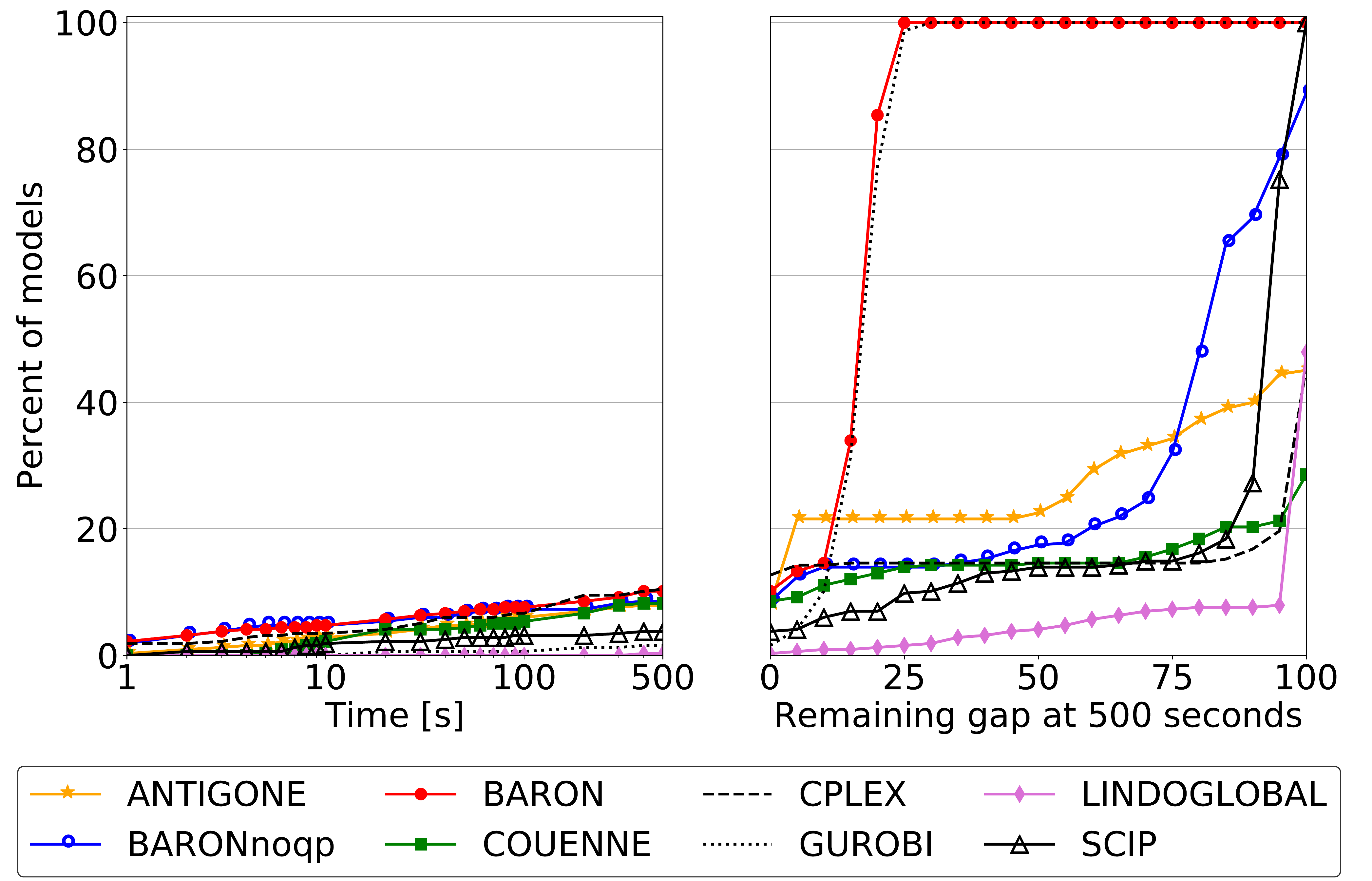}%
}
\caption{Comparison between global optimization solvers.}
\label{fig:global_solvers_profiles}
\end{figure}

Next, we provide a detailed analysis involving BARON, CPLEX and GUROBI. For this analysis, we use the same type of bar plots employed in Figure \ref{fig:baron_vs_baron_no_qp}. We start with a one-to-one comparison between BARON and CPLEX. First, we eliminate from the test set all the problems that can be solved trivially by both solvers (124 instances), obtaining a new test set with 1427 instances. In Figure~\ref{fig:baron_vs_cplex_speedups}, we consider the nontrivial problems that are solved to global optimality by at least one of the two solvers (445 instances), whereas in Figure~\ref{fig:baron_vs_cplex_gaps}, we consider nontrivial problems that neither of the two solvers are able to solve to global optimality within the time limit (982 instances). As both figures show, BARON performs significantly better than CPLEX. For 80\% of the instances considered in Figure~\ref{fig:baron_vs_cplex_speedups}, BARON is at least twice as fast as CPLEX. Similarly, for 90\% of the instances considered in Figure~\ref{fig:baron_vs_cplex_gaps}, BARON reports significantly smaller termination gaps than CPLEX.

\begin{figure}[htp]
\centering
\subfloat[CPU times (445 nontrivial instances). \label{fig:baron_vs_cplex_speedups}]{%
  \includegraphics[scale=0.19]{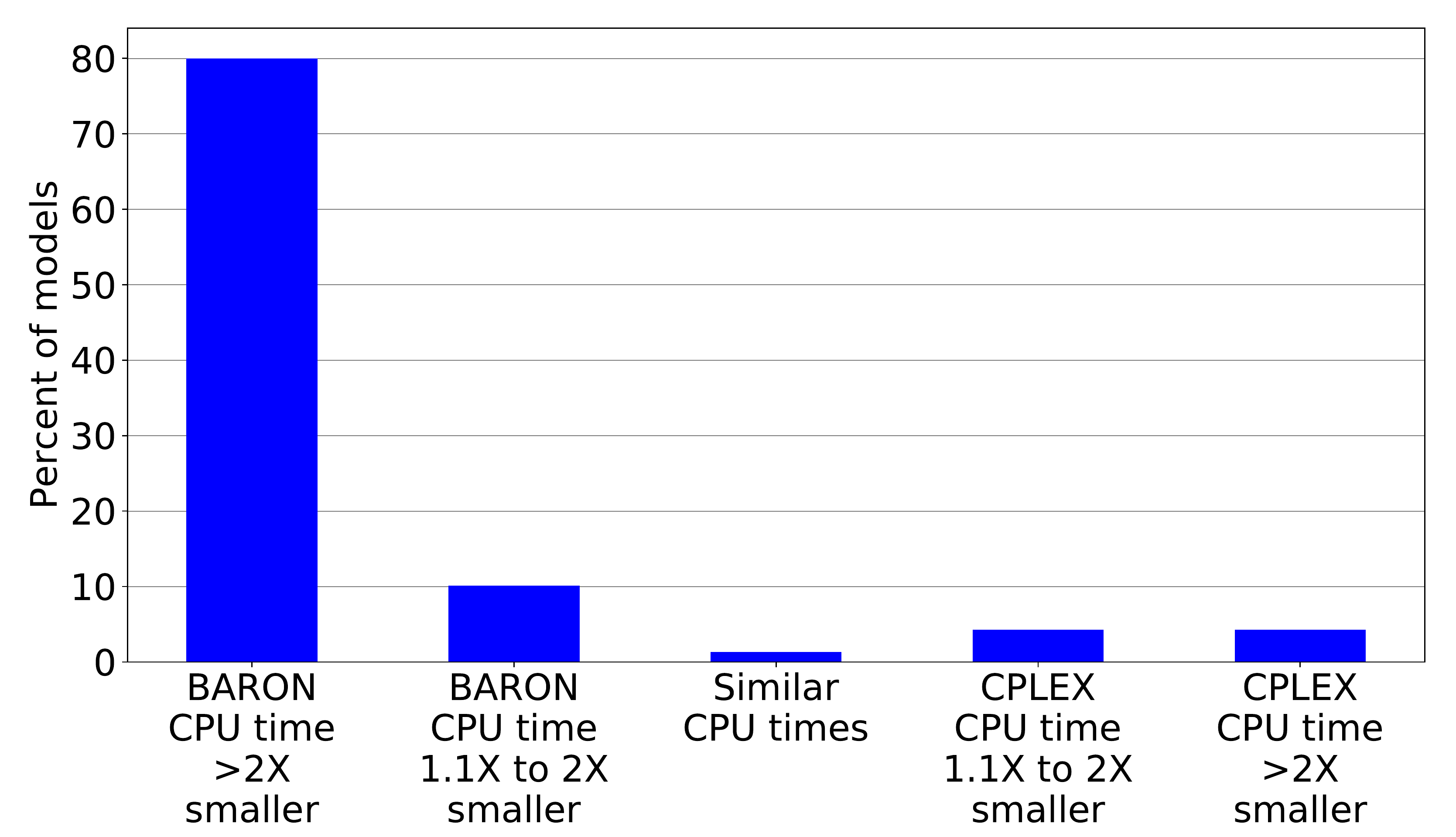}%
}
\subfloat[Relative gaps (982 nontrivial instances). \label{fig:baron_vs_cplex_gaps}]{%
  \includegraphics[scale=0.19]{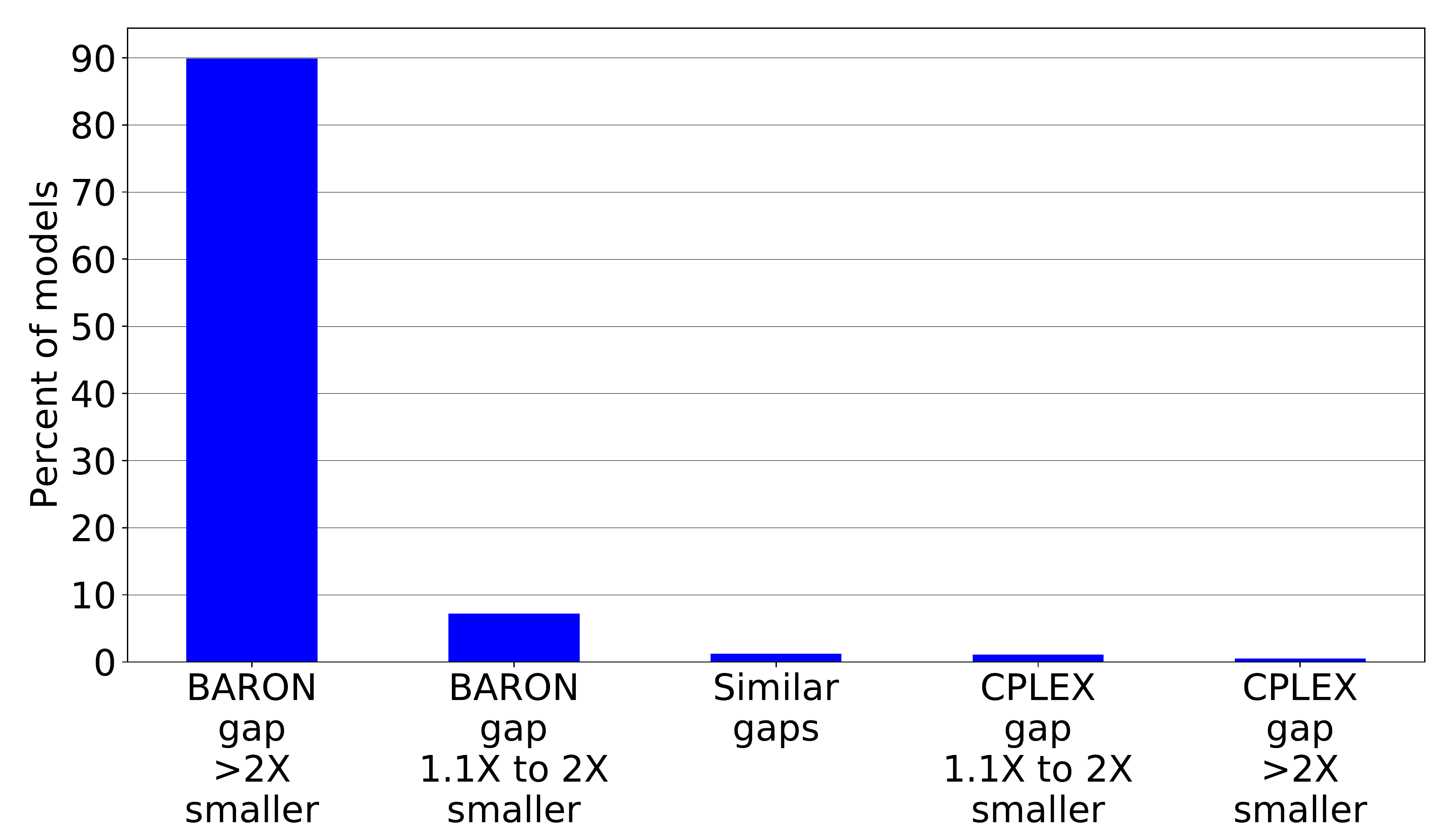}%
}
\caption{One-to-one comparison between BARON and CPLEX.}
\label{fig:baron_vs_cplex}
\end{figure}

Now, we present a similar one-to-one comparison between BARON and GUROBI. Once again, we eliminate from the test set all the problems that can be solved trivially by both solvers (185 instances), resulting in a new test set with 1366 instances. In Figure~\ref{fig:baron_vs_gurobi_speedups}, we consider the nontrivial problems that are solved to global optimality by at least one of the two solvers (380 instances), whereas in Figure~\ref{fig:baron_vs_gurobi_gaps}, we consider nontrivial problems that neither of the two solvers are able to solve to global optimality within the time limit (986 instances). For more than 60\% of the instances considered in Figure~\ref{fig:baron_vs_gurobi_speedups}, BARON is at least twice as fast as GUROBI, whereas for most of the problems considered in Figure~\ref{fig:baron_vs_gurobi_gaps}, the two solvers report similar termination gaps.

\begin{figure}[htp]
\centering
\subfloat[CPU times (380 nontrivial instances). \label{fig:baron_vs_gurobi_speedups}]{%
  \includegraphics[scale=0.19]{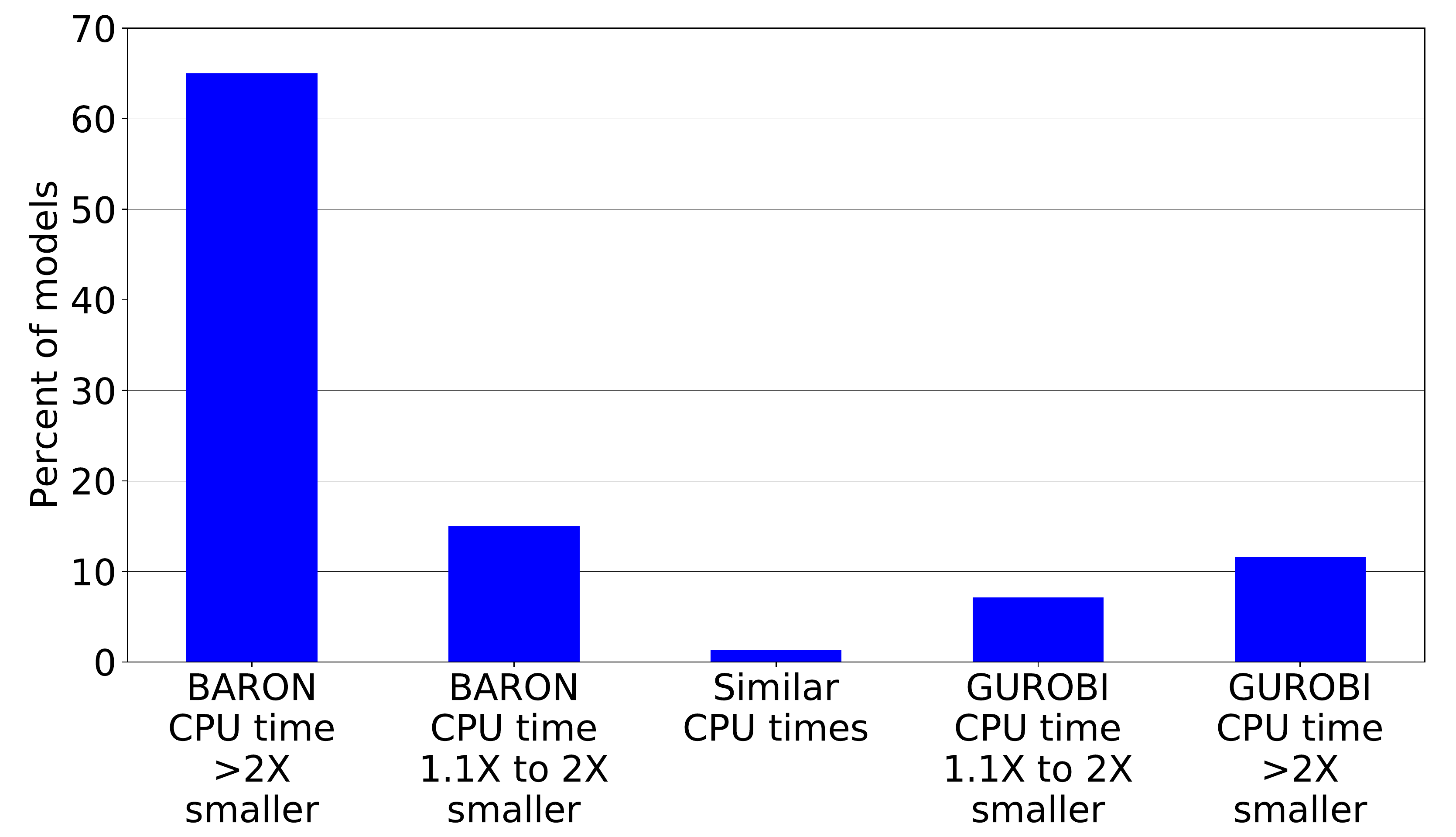}%
}
\subfloat[Relative gaps (986 nontrivial instances). \label{fig:baron_vs_gurobi_gaps}]{%
  \includegraphics[scale=0.19]{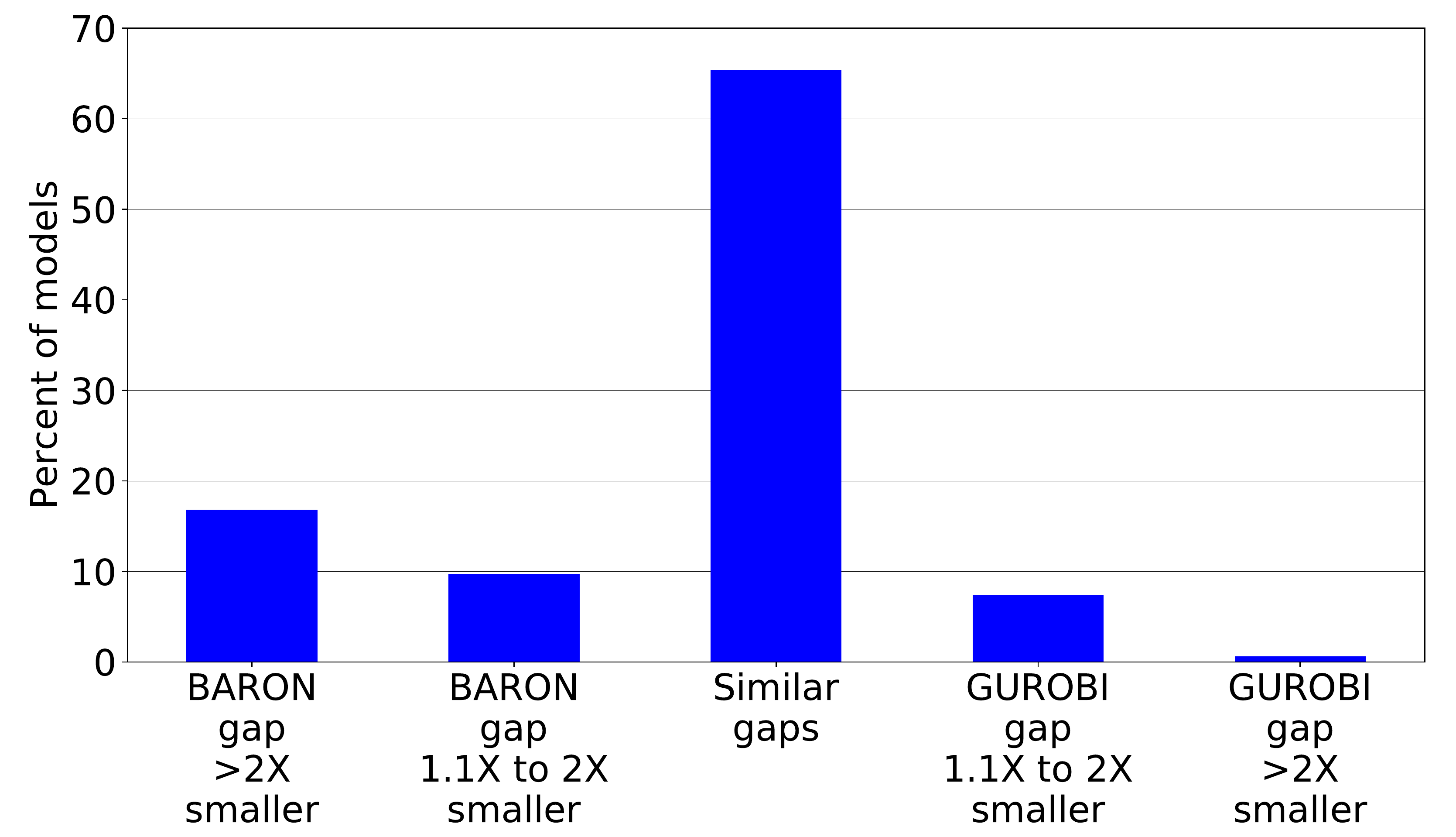}%
}
\caption{One-to-one comparison between BARON and GUROBI.}
\label{fig:baron_vs_gurobi}
\end{figure}

\subsection{Comparison with the QCR method}
\label{sec:results_comparison_with_qcr}

In this section, we provide a numerical comparison between the relaxation and branching strategies proposed in this paper, and the QCR approach reviewed in \S\ref{sec:qcr_based_relaxations}. To this end, we consider the 990 binary CBQP and QSAP instances described in~\cite{qp_miqp_tests}.

To apply the QCR method, we proceed in two steps. In the first step, for each test problem, we solve the SDP~\eqref{SDP_d_a_x} with MOSEK and use its dual solution to construct a reformulated convex binary quadratic program of the form~\eqref{QCR_da}. In the second step, we solve the reformulated problems using a customized version of BARON, which denote by BARONqcr. This version of BARON only differs from the default one in two aspects. First, BARONqcr is devised in a way such that, at a given node of the branch-and-bound tree, the lower bound is obtained by solving the continuous relaxation of the reformulated problem, which is a convex QP. To this end, in the lower bounding routines of BARONqcr, we have disabled the LP, NLP, MILP and recently introduced spectral relaxations. Second, in BARONqcr we have also turned off the spectral branching rule and replaced it with the reliability branching strategy described in~\cite{ks19}. Recall that, under the QCR approach, the perturbation parameters used to derive the reformulated problem are not updated during the execution of the branch-and-bound algorithm. As a result, the QP relaxations constructed at different nodes of the branch-and-bound tree of BARONqcr only differ from one another in the variables that are fixed. These convex QP relaxations are solved using CPLEX.

In our experiments, we run BARONqcr with the same relative/absolute tolerances and time limit used for BARON. For all of the considered instances, the amount of time required to solve the SDP relaxation involved in the first step of the QCR method was much smaller than the CPU time corresponding to BARONqcr. As a result, when comparing BARON and BARONqcr, we ignore the time required to solve these SDPs.

We first compare BARON and BARONqcr in terms of the lower bounds reported at the root-node. In this case, we say that a given solver obtains a better lower bound if the relative lower bound difference is greater than $10^{-3}$. For cases in which the magnitude of the lower bound is below one, we use absolute differences. The results are presented in Table~\ref{tab:baron_vs_baron_qcr_root_node}. In this table, for each test library, we provide the number, and in parentheses the percentage, of problems for which a given solver reports better root-node lower bounds. As the results in this table indicate, BARONqcr obtains better root-node lower bounds for most of the instances considered in this comparison.

At the root node of the branch-and-bound tree, the lower bound obtained by BARONqcr is given by the continuous relaxation of~\eqref{QCR_da}, and as a result, it is equal to the bound provided by the SDP~\eqref{SDP_d_a_x}. On the other hand, for many of the problems considered in this comparison, the eigenvalue relaxation in the nullspace of $A$ is tighter than the polyhedral relaxations implemented in BARON. Hence, in these cases, BARON relies on this quadratic relaxation to obtain lower bounds. As shown in Theorem~\ref{Theorem_2}, the SDP~\eqref{SDP_d_a_x} is at least as tight as the eigenvalue relaxation in the nullspace of $A$. Therefore, it is not surprising that, for many of the problems considered in Table~\ref{tab:baron_vs_baron_qcr_root_node}, BARONqcr provides tighter root-node bounds than BARON.

\begin{table}[htbp]
\centering
\caption{Root-node lower bounds given by BARON and BARONqcr.}
\label{tab:baron_vs_baron_qcr_root_node}
\begin{tabular}{ccc}
\toprule
Test library & BARON better & BARONqcr better \\
\midrule
CBQP         & 141 (15\%)   & 819 (85\%) \\
QSAP         & 2 (7\%)      & 28 (93\%)  \\ 
\bottomrule          
\end{tabular}
\end{table}

Next, we compare BARON and BARONqcr using the same type of bar plots employed in previous sections. We start by eliminating from the test set all the problems that can be solved trivially by both solvers (11 instances), obtaining a new test set with 979 instances. In Figure~\ref{fig:baron_vs_baron_qcr_speedups}, we consider the nontrivial problems that are solved to global optimality by at least one of the two solvers (442 instances), whereas in Figure~\ref{fig:baron_vs_baron_qcr_gaps}, we consider nontrivial problems that neither of the two solvers are able to solve to global optimality within the time limit (537 instances). As both figures show, BARON performs significantly better than BARONqcr. For nearly 70\% of the instances considered in Figure~\ref{fig:baron_vs_baron_qcr_speedups}, BARON is at least an order of magnitude faster than BARONqcr. Similarly, for more than 80\% of the instances considered in Figure~\ref{fig:baron_vs_baron_qcr_gaps}, BARON reports smaller termination gaps than BARONqcr.

Even though BARONqcr reports tighter root-node lower bounds than BARON for most of the instances considered in this comparison, during the branch-and-bound search, the lower bounds obtained by BARON improve much more quickly than those provided by BARONqcr. This is due to the fact that, in BARON, we update the perturbation parameters used to construct the quadratic relaxations as we branch. In addition, BARON makes use of the approximate spectral rule introduced in \S\ref{spectral_branching}, which as shown in \S\ref{sec:results_baron}, also has a significant impact on the performance of this solver.

\begin{figure}[htp]
\centering
\subfloat[CPU times (442 nontrivial instances). \label{fig:baron_vs_baron_qcr_speedups}]{%
  \includegraphics[scale=0.19]{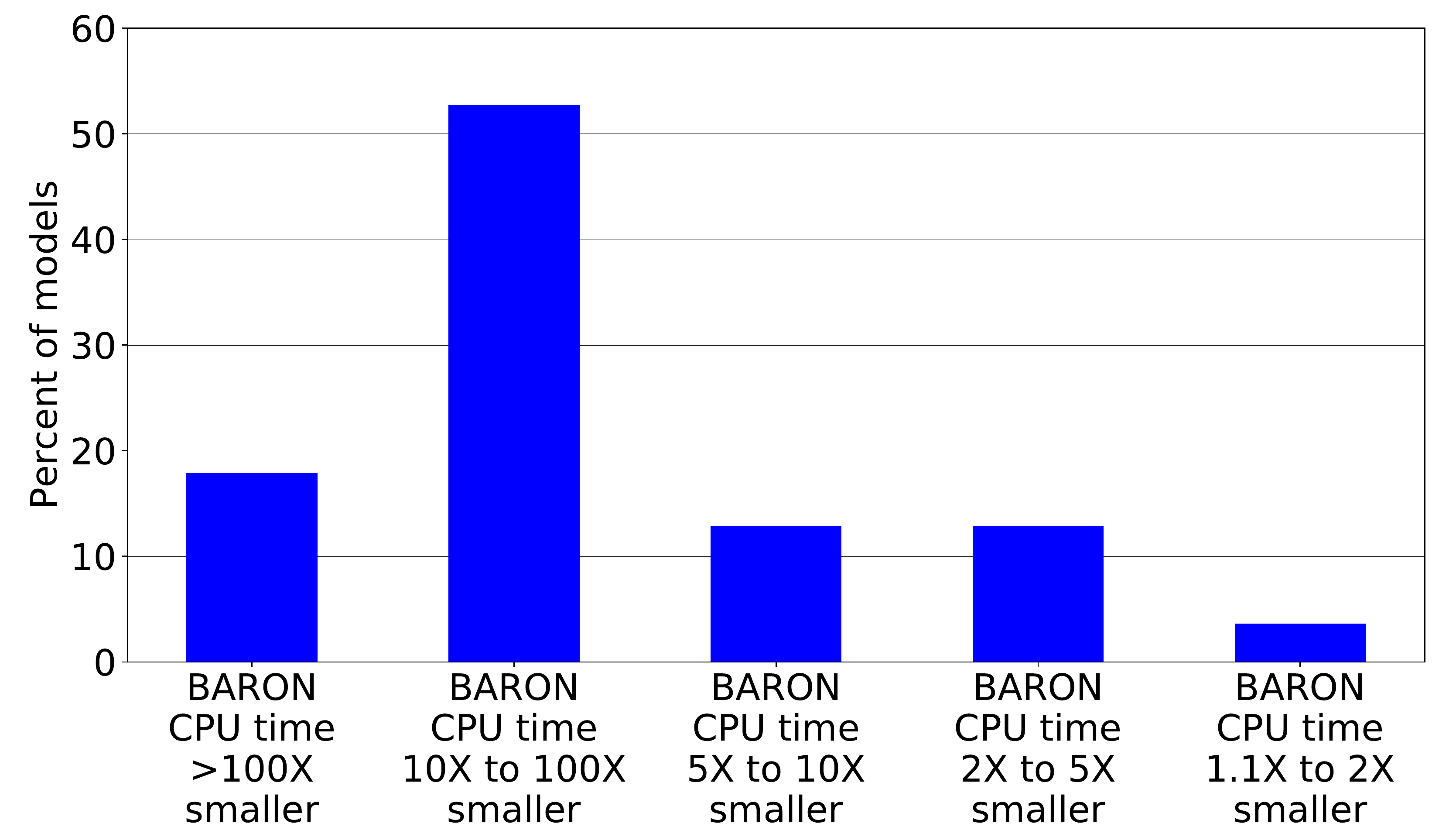}%
}
\subfloat[Relative gaps (537 nontrivial instances). \label{fig:baron_vs_baron_qcr_gaps}]{%
  \includegraphics[scale=0.19]{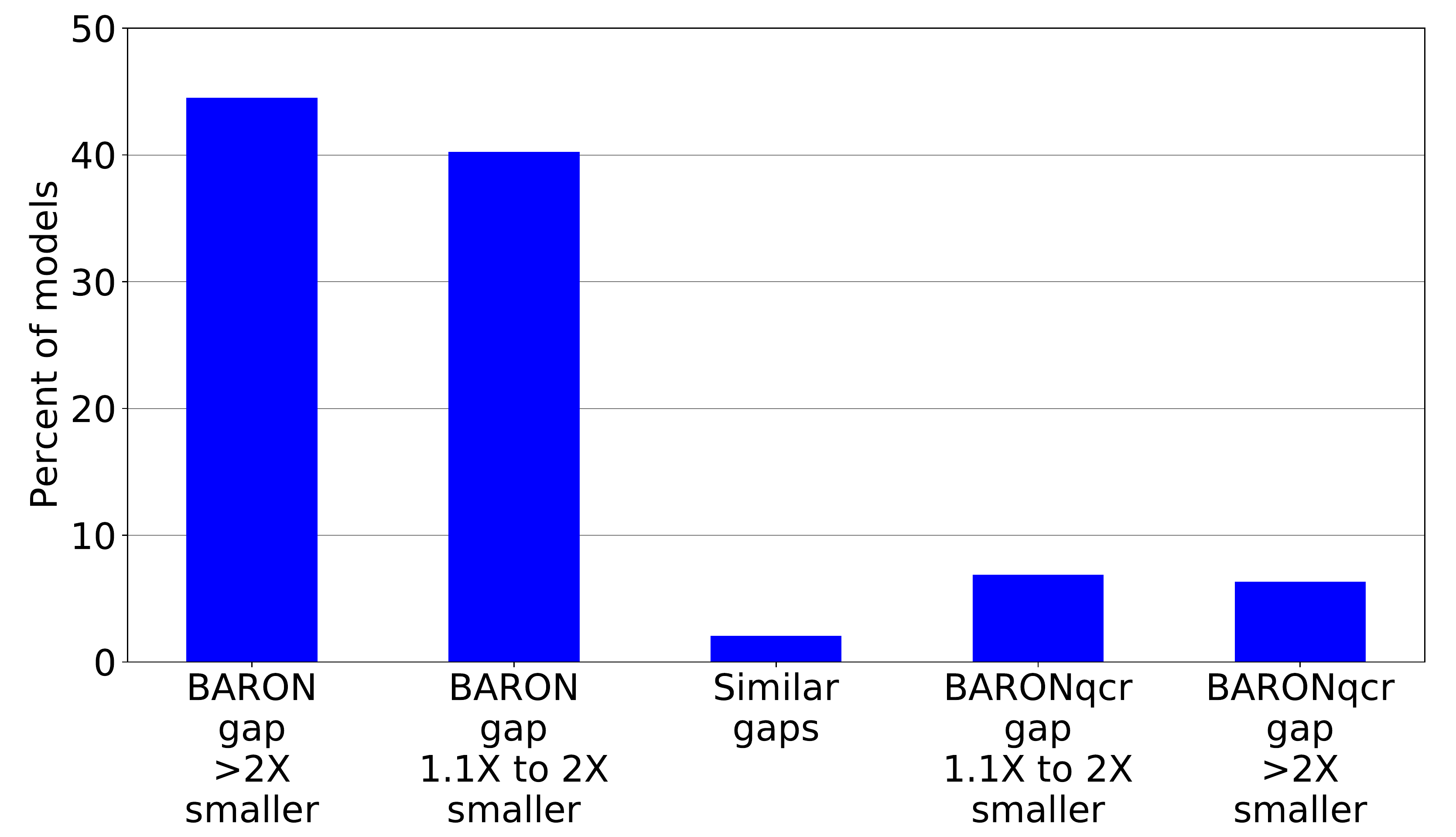}%
}
\caption{One-to-one comparison between BARON and BARONqcr.}
\label{fig:baron_vs_baron_qcr}
\end{figure}

\section{Conclusions}
\label{conclusions}

We introduced a family of convex quadratic relaxations for nonconvex QPs and MIQPs. We studied the theoretical properties of these relaxations, showing that they are equivalent to certain SDPs. We devised a novel branching variable selection strategy which approximates the impact of the branching decisions on the quality of these relaxations. To assess the benefits of our approach, we incorporated the proposed relaxation and branching techniques in the global solver BARON, and tested our implementation on a large collection of problems. Results demonstrated that, our implementation leads to a significant improvement in the performance of BARON, enabling it to solve many more problems to global optimality.

\bibliographystyle{plain}
\bibliography{carlos,reference}
\end{document}